\documentclass[english]{amsart}

\usepackage{amsmath, enumerate}

\usepackage{amssymb}

\usepackage{amscd} \usepackage{tabularx}
\usepackage[all,cmtip,line]{xy}

\newcommand{\F}{\mathbb{F}} 
\def\RCS$#1: #2 ${\expandafter\def\csname RCS#1\endcsname{#2}}
\RCS$Revision: 303 $
\RCS$Date: 2010-11-03 21:45:47 -0400 (Wed, 03 Nov 2010) $
 
\DeclareMathOperator{\Frob}{Frob}

\DeclareMathOperator{\sgn}{sgn}
 
 \DeclareMathOperator{\ab}{ab}
 \newcommand{\eps}{\epsilon}
\newcommand{\notdiv}{\nmid}

\newcommand{\wtilde}{\tilde{\omega}}
\newcommand{\To}{\longrightarrow}\newcommand{\into}{\hookrightarrow}
\newcommand{\m}{\mathfrak{m}}

\newcommand{\onto}{\twoheadrightarrow}
\newcommand{\isoto}{\stackrel{\sim}{\To}} 
 
 \newcommand{\rec}{\operatorname{rec}}
\newcommand{\St}{\operatorname{St}}
 
\newcommand{\cL}{\mathcal{L}}
\newcommand{\bigO}{\mathcal{O}}

 \newcommand{\R}{\mathbb{R}}
\newcommand{\Z}{\mathbb{Z}} \newcommand{\A}{\mathbb{A}}
\newcommand{\Q}{\mathbb{Q}}

\newcommand{\fq}{\mathfrak{q}} \newcommand{\fQ}{\mathfrak{Q}}
\newcommand{\C}{\mathbb{C}} 
\newcommand{\Gn}{\mathcal{G}_n}

\newcommand{\Favoid}{F^{(\mathrm{avoid})}} 
\newcommand{\tensor}{\otimes} 
\newcommand{\thbarpr}{\bar{\theta'}}

\newcommand{\bb}{\mathbb} 
\newcommand{\mc}{\mathcal}
\newcommand{\wt}{\widetilde} 
\newcommand{\mf}{\mathfrak}

\newcommand{\TT}{\bb{T}^{T}_{\lambda}(U,\mc{O})}
\DeclareMathOperator{\Iw}{Iw}
\DeclareMathOperator{\ord}{ord}  
\DeclareMathOperator{\Art}{Art}
\DeclareMathOperator{\univ}{univ}
\DeclareMathOperator{\red}{red}

\DeclareMathOperator{\ssg}{ss}

\DeclareMathOperator{\loc}{loc}
\DeclareMathOperator{\WD}{WD}
\DeclareMathOperator{\JL}{JL}

\newcommand{\rhobar}{\overline{\rho}} \newcommand{\rbar}{\bar{r}}
\newcommand{\rbarwtv}{\rbar|_{G_{F_{\wt{v}}}}}
\newcommand{\rbarwtvL}{\rbar|_{G_{L_{\wt{v}}}}}

\newcommand{\Gal}{\operatorname{Gal}}
\newcommand{\GL}{\operatorname{GL}}

\newcommand{\PGL}{\operatorname{PGL}}
\newcommand{\HT}{\operatorname{HT}}

 \newcommand{\Qbar}{\overline{\Q}}
 \newcommand{\Qp}{\Q_p}
\newcommand{\Ql}{\Q_l} 
\newcommand{\Qpbar}{\overline{\Q}_p}
\newcommand{\Qlbar}{\overline{\Q}_{l}}

\newcommand{\Flbar}{\overline{\F}_l} \newcommand{\Fbar}{\overline{\F}}

\newcommand{\gr}{\operatorname{gr}}

\newcommand{\Spec}{\operatorname{Spec}}
\newcommand{\Ind}{\operatorname{Ind}}
\newcommand{\SL}{\operatorname{SL}}
\newcommand{\PSL}{\operatorname{PSL}}
\newcommand{\ad}{\operatorname{ad}}

\newcommand{\tr}{\operatorname{tr}}

\newcommand{\End}{\operatorname{End}}
\newcommand{\Hom}{\operatorname{Hom}}

\newcommand{\Aut}{\operatorname{Aut}}

\newcommand{\Fil}{\operatorname{Fil}}

\newcommand{\Sym}{\operatorname{Sym}}

\newcommand{\dR}{\mathrm{dR}} \newcommand{\textB}{\mathrm{B}}

 \newcommand{\BdR}{\textB_{\dR}}

\DeclareMathOperator{\Def}{Def}

\usepackage{amsthm}

\newtheorem{thm}{Theorem}[subsection]

\newtheorem{cor}[thm]{Corollary}

\newtheorem*{lemunnum}{Lemma}
\newtheorem{lem}[thm]{Lemma} \newtheorem{prop}[thm]{Proposition}
 \theoremstyle{definition}
 \theoremstyle{definition}
\newtheorem{defn}[thm]{Definition} \theoremstyle{remark}

\numberwithin{equation}{subsection}

\theoremstyle{definition}
\newtheorem{situation}[thm]{Situation}
\newtheorem*{sublem}{Sublemma}

\begin{document}
\title[Sato-Tate]{The Sato-Tate conjecture
  for Hilbert modular forms}

\author{Thomas Barnet-Lamb}\email{tbl@brandeis.edu}\address{Department of Mathematics, Brandeis University, 415 South Street MS 050, Waltham MA}
\author{Toby Gee} \email{tgee@math.harvard.edu} \address{Department of
  Mathematics, Harvard University} \author{David Geraghty}
\email{geraghty@math.harvard.edu}\address{Department of Mathematics,
  Harvard University} \thanks{The second author was partially supported
  by NSF grant DMS-0841491.}  
\begin{abstract}We prove the Sato-Tate conjecture for Hilbert modular
  forms. More precisely, we prove the natural generalisation of the
  Sato-Tate conjecture for regular algebraic cuspidal automorphic representations
  of $\GL_2(\A_F)$, $F$ a totally real field, which are not of CM
  type. The argument is based on the potential automorphy techniques
  developed by Taylor \emph{et.~al.}, but makes use of automorphy
  lifting theorems over ramified fields, together with a
  ``topological'' argument with local deformation rings. In
  particular, we give a new proof of the conjecture for modular forms,
  which does not make use of potential automorphy theorems for
  non-ordinary $n$-dimensional Galois representations.
\end{abstract}
\maketitle
\tableofcontents
\section{Introduction}
\subsection{}In this paper we prove the Sato-Tate conjecture for
Hilbert modular forms. More precisely, we prove the natural
generalisation of the Sato-Tate conjecture for regular algebraic
cuspidal automorphic representations of $\GL_2(\A_F)$, $F$ a totally real field,
which are not of CM type. 

Several special cases of this result were proved in the last few
years. The papers \cite{hsbt} and \cite{tay06} prove the result for elliptic curves
over totally real fields which have potentially multiplicative
reduction at some place, and it is straightforward to extend this
result to the case of cuspidal automorphic representations of weight 0
(i.e. those corresponding to Hilbert modular forms of parallel weight
2) which are a twist of the Steinberg representation at some finite
place. The case of modular forms (over $\Q$) of weight 3 whose
corresponding automorphic representations are a twist of the Steinberg
representation at some finite place was treated in \cite{GeeSTwt3},
via an argument that depends on the existence of infinitely many
ordinary places. The case of modular forms (again over $\Q$) was
proved in \cite{BLGHT} (with no assumption on the existence of a
Steinberg place). The main new features of the arguments of
\cite{BLGHT} were the use of an idea of Harris (\cite{harristrick}) to
ensure that potential automorphy need only be proved in weight 0,
together with a new potential automorphy theorem for $n$-dimensional
Galois representations which are symmetric powers of those attached to
non-ordinary modular forms. Recent developments in the theory of the
trace formula remove the need for an assumption of the existence of a
Steinberg place in both this theorem and in the case of elliptic
curves over totally real fields.

In summary, the Sato-Tate conjecture has been proved for modular forms
and for elliptic curves over totally real fields, but is not known in
any non-trivial case for Hilbert modular forms not of parallel weight
2 over any field other than $\Q$. It seems to be hard to extend the
arguments of either \cite{GeeSTwt3} or \cite{BLGHT} to the general
case; in the former case one has no way of establishing the existence
of infinitely many ordinary places (although it is conjectured that
the set of such places should be of density one), and in the latter
case one has no control over the mixture of ordinary and supersingular
places over any rational prime. In this paper, we adopt a new
approach: we combine the approach of \cite{GeeSTwt3}, which is based
on taking congruences to representations of $\GL_2(\A_F)$ of weight
$0$, with the twisting argument of \cite{harristrick} (or rather the
version of this argument used in \cite{BLGHT}). These techniques do
not in themselves suffice to prove the result, as one has to prove a
automorphy lifting theorem for non-ordinary representations over a
ramified base field. No such theorems are known for representations of
dimension greater than two. The chief innovation of this paper is a
new technique for proving such results.

Our new automorphy lifting theorem uses the usual Taylor-Wiles-Kisin
patching techniques, but rather than identifying an entire deformation
ring with a Hecke algebra, we prove that certain global Galois
representations, whose restrictions to decomposition groups lie on
certain components of the local lifting rings, are automorphic. That
this is the ``natural'' output of the Taylor-Wiles-Kisin method is at
least implicit in the work of Kisin, cf. section 2.3 of
\cite{MR2459302}. One has to be somewhat careful in making this
precise, because it is necessary to use fixed lattices in the global
Galois representations one considers, and to work with lifting rings
rather than deformation rings. In particular, it is not clear that the
set of irreducible components of a local lifting ring containing a
particular $\bigO_K$-valued point, $K$ a finite extension of $\Ql$, is
determined by the equivalence class of the corresponding
$K$-representation. This necessitates a good deal of care to work with
$\bigO_K$-liftings throughout the paper.

Effectively (modulo the remarks about lattices in the previous
paragraph) the automorphy lifting theorem that we prove tells us that
if we are given two congruent $n$-dimensional $l$-adic regular
crystalline essentially self dual representations of $G_F$ (the
absolute Galois group of a totally real field $F$) with the same
$l$-adic Hodge types, with ``the same ramification properties'',
and satisfying a standard assumption on the size of the mod $l$ image,
then if one of them is automorphic, so is the other. By ``the same
ramification properties'', we mean that they are ramified at the same
set of places, and that the points determined by the two
representations on the corresponding local lifting rings lie on the
same components. For example, we require that the two representations
have unipotent ramification at exactly the same set of places; we do
not know how to adapt Taylor's techniques for avoiding Ihara's lemma
(\cite{tay06}) to this more general setting.

The local deformation rings for places not dividing $l$ are reasonably
well-understood, so that it is possible to verify that this condition
holds at such places in concrete examples. On the other hand, the
components of the crystalline deformation rings of fixed weight are
not at all understood if $l$ ramifies in $F$, unless $n=2$ and the
representations are Barsotti-Tate, when there are at most two
components, corresponding to ordinary and non-ordinary
representations. This might appear to prevent us from being able to
apply our theorem to any representations at all. We get around this
problem by making use of the few cases where the components are
known. Specifically, we use the cases where $n=2$ and either the
representations are Barsotti-Tate; or $F$ is unramified in $l$, and
the representations are crystalline of low weight. To explain how we
are able to bootstrap from these two cases, we now explain the main
argument.

We begin with a regular algebraic cuspidal representation $\pi$ of
$\GL_2(\A_F)$, assumed not to be of CM type. By a standard analytic
argument, it suffices to prove that for each $n\ge 1$ the $(n-1)$-st
symmetric power of $\pi$ is potentially automorphic, in the sense that
there is a finite Galois extension $F''/F$ of totally real fields and an
automorphic representation $\pi_n$ of $\GL_n(\A_{F''})$ whose $L$-function
is equal to that of the base change to $F''$ of the $(n-1)$-st
symmetric power $L$-function of $\pi$. Equivalently, if we fix a prime
$l$, then it suffices to prove that the $(n-1)$-st symmetric power of
an $l$-adic Galois representation corresponding to $\pi$ is
potentially automorphic, i.e. that its restriction to $G_{F''}$ is
automorphic. This is what we prove. 

We choose $l$ to be large and to split completely in $F$, and such
that $\pi_v$ is unramified at all places $v$ of $F$ lying over $l$. We
begin by making a preliminary solvable base change to a totally real
field $F'/F$, such that the base change $\pi_{F'}$ of $\pi$ to $F'$ is
either unramified or an unramified twist of the Steinberg
representation at each finite place of $F'$. We then choose an
automorphic representation $\pi'$ of $\GL_2(\A_{F'})$ of weight 0
which is congruent to $\pi$, which for any place $v\nmid l$ is
unramified (respectively an unramified twist of the Steinberg
representation) if and only if $\pi$ is, and which is a principal
series representation (possibly ramified) or a supercuspidal
representation for all $v|l$. Furthermore we choose $\pi'$ so that
for places $v|l$, $\pi_v$ is ordinary if and only if $\pi'_v$ is
ordinary.

We now prove that the $(n-1)$-st symmetric power of $\pi'$ is
potentially automorphic over some finite Galois extension $F''$ of $F$. This is
straightforward, although it is not quite in the literature. This
is the only place that we need to make use of a potential automorphy
theorem for an $n$-dimensional Galois representation, and the theorems
of \cite{hsbt} (or rather the versions of them which are now available
thanks to improvements in our knowledge of the trace formula, which
remove the need for discrete series hypotheses) would suffice, but for the
convenience of the reader we use a theorem from \cite{BLGHT}
(which, for instance, already include the improvements made possible
by our enhanced understanding of the trace formula)
instead. This also allows us to avoid having to make an argument with
Rankin-Selberg convolutions as in \cite{hsbt}. We note that the
theorem we use from \cite{BLGHT} is for ordinary representations,
rather than the far more technical result for supersingular
representations that is also proved in \cite{BLGHT}.

We now wish to deduce the potential automorphy of the $(n-1)$-st
symmetric power of $\pi$, or rather the automorphy of the
corresponding $l$-adic Galois representation
$r:G_{F''}\to\GL_n(\Qlbar)$, from the automorphy of the $l$-adic
Galois representation $r':G_{F''}\to\GL_n(\Qlbar)$ corresponding to
the $(n-1)$-st symmetric power of $\pi'$. We cannot directly apply our
automorphy lifting theorem, because the Hodge-Tate weights of $r'$ and
$r$ are different. Instead, we employ an argument of Harris
(\cite{harristrick}), and tensor both $r$ and $r'$ with
representations obtained by the automorphic induction of algebraic
characters of a certain CM field. The choice of the field and the
characters is somewhat delicate, in order to preserve various
technical assumptions for the automorphy lifting theorem, in
particular the assumption of big residual image. The two characters
are chosen so that the resulting Galois representations $r''$ and
$r'''$ are potentially crystalline with the same Hodge-Tate
weights. The representation $r'''$ is automorphic, by standard results
on automorphic induction. We then apply our automorphy lifting theorem
to deduce the automorphy of $r''$. The automorphy of $r$ then follows
by an argument as in \cite{harristrick} (although we employ a version
of this which is very similar to that used in \cite{BLGHT}).

In order to apply our automorphy lifting theorem, we need to check the
local hypotheses. At places not dividing $l$, these essentially follow
from the construction of $\pi'$, together with a path-connectedness
argument, and a check (using the Ramanujan conjecture) that a certain
point on a local lifting ring is smooth. At the places dividing $l$
the argument is rather more involved. At the places where $\pi$ is
ordinary the hypothesis can be verified (after a suitable base change)
using the results of \cite{ger}. At the non-ordinary places we proceed
more indirectly. For each non-ordinary $v|l$ we choose two local
2-dimensional $l$-adic representations $\rho$ and $\rho'$ of $G_{F_v}$
which are induced from characters of quadratic extensions. The
representations $\rho$ and $\rho'$ are chosen to be congruent to the
local Galois representations attached to $\pi$, $\pi'$ respectively, with
$\rho$ crystalline of the same Hodge-Tate weights as the local
representation attached to $\pi$, and $\rho'$ non-ordinary and
potentially Barsotti-Tate. Then $\rho$ is on the same component of the
local crystalline lifting ring as the local representation attached to
$\pi$, and a similar statement is true for $\rho'$ and $\pi'$ after a base change
to make it crystalline (using the knowledge of the components of
Barsotti-Tate lifting rings mentioned above). Since the image of an
irreducible component under a continuous map is irreducible, a
straightforward argument shows that we need only check that the Galois
representations corresponding to the $(n-1)$-st symmetric powers of
$\rho$ and $\rho'$, when tensored with the Galois representations
obtained from the characters induced from the CM field, lie on a
common component of a crystalline deformation ring (possibly after a
base change). We ensure this by choosing our characters so that the
two Galois representations are both direct sums of unramified twists
of the same crystalline characters, and making a path-connectedness
argument.

We should note that we have suppressed some technical details in the
above outline of our argument; we need to take considerable care to
ensure that the hypotheses relating to residual Galois representations
having big image are satisfied. In addition, as mentioned above,
rather than working with Galois representations valued in fields it is
essential to work with fixed lattices throughout. We remark that in
the forthcoming paper \cite{blggt} we remove the need to consider lattices, and
generalise the arguments of this paper to prove potential automorphy
theorems for a broad class of Galois representations.

We now outline the structure of the paper. In section \ref{sec:
  notation} we recall some basic notation and definitions from
previous papers on automorphy lifting theorems. The automorphy lifting
theorem is proved in section \ref{sec: lifting theorem}, together with
some results on the behaviour of local lifting rings under conjugation
and functorialities. The most technical section of the paper is
section \ref{sec; character building}, where we construct the
characters of CM fields that we need in the main argument. In section
\ref{sec: twisting and untwisting} we recall various standard results
on base change and automorphic induction, and give an exposition of
Harris's trick in the level of generality we require. In section
\ref{sec: pot auto weight 0} we prove a potential automorphy theorem
in weight 0; the precise result we require is not in the literature,
and while it is presumably clear to the experts how to prove it, we
provide the details. Finally, in section \ref{sec: main thm} we carry
out the strategy described above, and deduce the Sato-Tate conjecture.

We would like to thank Richard Taylor for some helpful discussions
related to the content of this paper. We would also like to thank Florian
Herzig and Sug Woo Shin for their helpful comments on an earlier draft.

\section{Notation}\label{sec: notation}If $M$ is a field, we let $G_M$
denote its absolute Galois group.  We write all matrix transposes on
the left; so ${}^tA$ is the transpose of $A$. Let $\epsilon$ denote
the $l$-adic or mod $l$ cyclotomic character. If $M$ is a finite
extension of $\bb{Q}_p$ for some $p$, we write $I_M$ for the inertia
subgroup of $G_M$. If $R$ is a local ring we write $\mf{m}_{R}$ for
the maximal ideal of $R$.

We fix an algebraic closure $\Qbar$ of $\Q$, and regard all algebraic
extensions of $\Q$ as subfields of $\Qbar$. For each prime $p$ we fix
an algebraic closure $\Qpbar$ of $\Qp$, and we fix an embedding
$\Qbar\into\Qpbar$. In this way, if $v$ is a finite place of a number
field $F$, we have a homomorphism $G_{F_v}\into G_F$.

We will use some of the notation and definitions of \cite{cht} without
comment. In particular, we will use the notions of RACSDC (regular,
algebraic, conjugate self-dual, cuspidal) and RAESDC (regular,
algebraic, essentially self-dual, cuspidal)
automorphic representations, for which see sections 4.2 and 4.3 of
\cite{cht}. We will also use the notion of a RAECSDC (regular,
algebraic, essentially conjugate self-dual, cuspidal) automorphic
representation, for which see section 1 of \cite{BLGHT}. If $\pi$ is a RAESDC automorphic representation of $\GL_n(\A_F)$, $F$
a totally real field, and $\iota:\Qlbar\isoto\C$, then we let
$r_{l,\iota}(\pi):G_F\to\GL_n(\Qlbar)$ denote the corresponding Galois
representation. Similarly, if $\pi$ is a RAECSDC or RACSDC automorphic representation of $\GL_n(\A_F)$, $F$
a CM field (in this paper, all CM fields are totally imaginary), and $\iota:\Qlbar\isoto\C$, then we let
$r_{l,\iota}(\pi):G_F\to\GL_n(\Qlbar)$ denote the corresponding Galois
representation. For the properties of $r_{l,\iota}(\pi)$, see Theorems
1.1 and 1.2 of \cite{BLGHT}.  If $K$ is a finite extension of $\Qp$ for some
$p$, we will let $\rec_K$ be the local Langlands correspondence of
\cite{MR1876802}, so that if $\pi$ is an irreducible complex
admissible representation of $\GL_n(K)$, then $\rec_K(\pi)$ is a
Weil-Deligne representation of the Weil group $W_K$. If $K$ is an
archimedean local field, we write $\rec_K$ for the local Langlands
correspondence of \cite{MR1011897}. We will write $\rec$ for $\rec_K$
when the choice of $K$ is clear.

\section{An automorphy lifting theorem}\label{sec: lifting theorem}

\subsection{The group $\mathcal{G}_n$}
Let $n$ be a positive integer, and let $\Gn$ be the group scheme over
$\Z$ which is the semidirect product of $\GL_n\times\GL_1$ by the
group $\{1,j\}$, which acts on $\GL_n\times\GL_1$
by \[j(g,\mu)j^{-1}=(\mu{}^tg^{-1},\mu).\] There is a homomorphism
$\nu:\Gn\to\GL_1$ sending $(g,\mu)$ to $\mu$ and $j$ to $-1$. Write
$\mathfrak{g}_n^0$ for the trace zero subspace of the Lie algebra of
$\GL_n$, regarded as a Lie subalgebra of the Lie algebra of $\Gn$.

Suppose that $F$ is an imaginary CM field with totally real subfield $F^+$.
If $R$ is a ring and $r : G_{F^+} \rightarrow \mc{G}_n(R)$
is a homomorphism with $r^{-1}(\GL_n(R)\times \GL_1(R))=G_{F}$, we
will make a slight abuse of notation and write $r|_{G_{F}}$
(respectively $r|_{G_{F_w}}$ for $w$ a place of $F$) to mean
$r|_{G_{F}}$ (respectively $r|_{G_{F_w}}$) composed with the
projection $\GL_n(R)\times \GL_1(R) \rightarrow \GL_n(R)$.

\subsection{$l$-adic automorphic forms on unitary groups}
\label{subsec: l-adic aut forms on unitary groups}
Let $F^{+}$ denote a totally real number field and $n$ a positive
integer. Let $F/F^{+}$ be a totally imaginary quadratic extension of
$F^{+}$ and let $c$ denote the non-trivial element of
$\Gal(F/F^{+})$. 
Suppose that the extension $F/F^{+}$ is unramified at
all finite places. Assume that $n[F^{+}:\bb{Q}]$ is divisible by
4. Under this assumption, we can find a reductive algebraic group $G$
over $F^{+}$ with the following properties:
\begin{itemize}
\item $G$ is an outer form of $\GL_n$ with $G_{/F} \cong \GL_{n/F}$;
\item for every finite place $v$ of $F^{+}$, $G$ is quasi-split at
  $v$;
\item for every infinite place $v$ of $F^{+}$, $G(F^{+}_v) \cong
  U_{n}(\bb{R})$.
\end{itemize}
We can and do fix a model for $G$ over the ring of integers
$\mc{O}_{F^{+}}$ of $F^{+}$ as in section 2.1 of \cite{ger}.
For each place $v$ of $F^{+}$ which splits as $ww^{c}$ in $F$ there is
a natural isomorphism
\[ \iota_{w} : G(F^{+}_v) \isoto GL_n(F_w) \]
which restricts to an isomorphism between $G(\mc{O}_{F^{+}_v})$ and $GL_n(\mc{O}_{F_w})$.
If $v$ is a place of $F^{+}$ split over $F$ and $w$ is a place of
$F$ dividing $v$, then we let 
\begin{itemize}
\item $\Iw(w)$ denote the subgroup of $GL_n(\mc{O}_{F_w})$
  consisting of matrices which reduce to an upper triangular matrix
  modulo $w$;
\item $U_0(w)$ denote the subgroup of $GL_n(\mc{O}_{F_w})$ consisting
  of matrices whose last row is congruent to $(0,\ldots,0,*)$ modulo
  $w$;
\item $U_1(w)$ denote the subgroup of $U_0(w)$ consisting of matrices
  whose last row is congruent to $(0,\ldots,0,1)$ modulo $w$.
\end{itemize}

Let $l>n$ be a prime number with the property that every place of
$F^{+}$ dividing $l$ splits in $F$. Let $S_l$ denote the set of places
of $F^{+}$ dividing $l$. Let $K$ be an algebraic extension of $\bb{Q}_l$ in
$\Qlbar$ such that every embedding $F \hookrightarrow \Qlbar$ has
image contained in $K$. 
Let $\mc{O}$ denote the ring of integers in $K$ and $k$ the residue field. Let
$S_l$ denote the set of places of $F^{+}$ dividing $l$ and for each $v
\in S_l$, let $\wt{v}$ be a place of $F$ over $v$.

Let $W$ be an $\mc{O}$-module with an action of $G(\mc{O}_{F^{+},l})$. 
Let $V \subset G(\bb{A}_{F^+}^{\infty})$ be a compact open subgroup
with $v_l \in G(\mc{O}_{F^+,l})$ for all $v \in V$, where $v_l$
denotes the projection of $v$ to $G(F^{+}_l)$. We let $S(V,W)$ denote
the space of $l$-adic automorphic forms on $G$ of weight $W$ and level $V$, that is, the space of functions
\[  f : G(F^{+})\backslash G(\bb{A}_{F^+}^{\infty}) \rightarrow W \]
with $f(gv) = v_l^{-1} f(g)$ for all $v \in V$. 

Let $\wt{I}_{l}$ denote the set of embeddings $F \hookrightarrow
K$
giving rise to one of the places $\wt{v}$.
Let
$(\bb{Z}^{n}_{+})^{\wt{I}_l}$ denote the set of $\lambda \in
(\bb{Z}^{n})^{\wt{I}_l}$ with $\lambda_{\tau,1}\geq \lambda_{\tau,2}
\geq \ldots \geq \lambda_{\tau,n}$ for all embeddings $\tau \in \wt{I}_l$. To each $\lambda \in
(\bb{Z}^{n}_{+})^{\wt{I}_l}$ we associate a finite free
$\mc{O}$-module $M_{\lambda}$ with a continuous action of
$G(\mc{O}_{F^{+},l})$ as in Definition 2.2.3 of \cite{ger}. The
representation $M_{\lambda}$ is the tensor product over $\tau \in
\wt{I}_l$ of the irreducible algebraic representations of $\GL_n$ of
highest weights given by the $\lambda_\tau$. We write $S_{\lambda}(V,\mc{O})$
instead of $S(V,M_{\lambda})$ and similarly for any $\mc{O}$-module
$A$, we write $S_{\lambda}(V,A)$ for
$S(V,M_{\lambda}\otimes_{\mc{O}}A)$.

Assume from now on that $K$ is a finite extension of $\Ql$. Let
$\mf{l}$ denote the product of all places in $S_l$. Let $R$ and $S_a$
denote finite sets of finite places of $F^{+}$ disjoint from each
other and from $S_l$ and consisting only of places which split in
$F$. Assume that each $v \in S_a$ is unramified over a rational prime $p$ with
$[F(\zeta_p):F]>n$. Let $T=S_l \coprod R \coprod S_a$. For each $v \in
T$ fix a place $\wt{v}$ of $F$ dividing $v$, extending the choice of
$\wt{v}$ for $v \in S_l$. 
Let $U=\prod_v U_v$ be a compact open subgroup of
$G(\bb{A}_{F^{+}}^{\infty})$ such that
\begin{itemize}
\item $U_v = G(\mc{O}_{F^{+}_v})$ if $v \not \in R \cup S_a$ splits in
  $F$;
\item $U_v = \iota_{\wt{v}}^{-1}\ker( GL_n(\mc{O}_{F_{\wt{v}}})
  \rightarrow GL_n(k(\wt{v})))$ if $v \in S_a$;
\item $U_v$ is a hyperspecial maximal compact subgroup of $G(F^{+}_v)$
  if $v$ is inert in $F$.
\end{itemize}
(At this stage we impose no restrictions on $U_v$ for places $v\in
R$.) We note that if $S_a$ is non-empty then $U$ is sufficiently small (which means that its projection to $G(F^{+}_v)$ for some place $v \in F^{+}$ contains no finite order elements other than the identity).

For any $\mc{O}$-algebra $A$, the space $S_{\lambda}(U,A)$ is acted upon by the Hecke operators
\[ T_{w}^{(j)}:=  \iota_{w}^{-1}\left( \left[ GL_n(\mc{O}_{F_w}) \left( \begin{matrix}
      \varpi_{w}1_j & 0 \cr 0 & 1_{n-j} \end{matrix} \right)
GL_n(\mc{O}_{F_w}) \right] \right)
\]
for $w$ a place of $F$, split over $F^+$ and not lying over $T$,
$j=1,\ldots,n$ and $\varpi_{w}$ a uniformizer in $\mc{O}_{F_w}$. We
let $\bb{T}^{T}_{\lambda}(U,A)$ be the $A$-subalgebra of
$\End_A(S_{\lambda}(U,A))$ generated by these operators and the
operators $(T_w^{(n)})^{-1}$.

To any maximal ideal $\mf{m}$ of $\TT$ one can associate a continuous representation
\[ \rbar_{\mf{m}} : G_{F} \rightarrow
\GL_n(\TT/\mf{m}) \]
characterised by the following properties:
\begin{enumerate}
\item $\rbar_{\mf{m}}^{c} \cong \rbar_{\mf{m}}^{\vee} \overline{\epsilon}^{1-n}$;
\item $\rbar_{\mf{m}}$ is unramified outside $T$.  If $v \not \in T$ is a place of $F^{+}$ which splits as $ww^{c}$
  in $F$ and $\Frob_w$ is the geometric Frobenius element of $G_{F_w}/I_{F_w}$, then
$\rbar_{\mf{m}}(\Frob_{w})$ has characteristic polynomial
\[ X^{n}+\ldots+ (-1)^{j} (\mathbf{N}w)^{j(j-1)/2} T_{w}^{(j)}X^{n-j} + \ldots + (-1)^{n}(\mathbf{N}w)^{n(n-1)/2} T_{w}^{(n)}. \]
\end{enumerate}
The maximal ideal $\mf{m}$ is said to be \emph{non-Eisenstein} if
$\rbar_{\mf{m}}$ is absolutely irreducible. In this case, $\rbar_{\mf{m}}$ can be extended to a homomorphism $\rbar_{\mf{m}} : G_{F^+} \rightarrow \mc{G}_n(\TT/\mf{m})$ (in the sense that $\rbar_{\mf{m}}|_{G_{F}} = (\rbar_{\mf{m}},\overline{\epsilon}^{1-n}))$ with $\rbar_{\mf{m}}^{-1}((\GL_n \times \GL_1)(\TT/\mf{m})) = G_{F}$. 
Also, any such extension has a continuous lifting
\[ r_{\mf{m}} : G_{F^{+}} \rightarrow \mc{G}_{n}(\TT_{\mf{m}}) \]
with the following properties:
\begin{enumerate}
\item[(0)] $r_{\mf{m}}^{-1}((\GL_n\times \GL_1)(\TT_{\mf{m}})) = G_{F}$.
\item $\nu \circ r_{\mf{m}} = \epsilon^{1-n} \delta_{F/F^{+}}^{\mu_\mf{m}}$ where $\delta_{F/F^+}$ is the
non-trivial character of $\Gal(F/F^{+})$ and $\mu_{\mf{m}} \in
\bb{Z}/2\bb{Z}$.
\item $r_{\mf{m}}$ is unramified outside $T$. If $v \not \in T$ is a place of $F^{+}$ which splits as $ww^{c}$
  in $F$ and $\Frob_w$ is the geometric Frobenius element of $G_{F_w}/I_{F_w}$, then
$r_{\mf{m}}(\Frob_{w})$ has characteristic polynomial
\[ X^{n}+\ldots+ (-1)^{j} (\mathbf{N}w)^{j(j-1)/2} T_{w}^{(j)}X^{n-j} + \ldots + (-1)^{n}
(\mathbf{N}w)^{n(n-1)/2} T_{w}^{(n)} .\]
\item If $v \in S_l$, and $\zeta : \TT \rightarrow \Qlbar$ is a
  homomorphism of $\mc{O}$-algebras, then $\zeta \circ
  r_{\mf{m}}|_{G_{F_{\wt{v}}}}$ is crystalline of $l$-adic Hodge type
  $\mathbf{v}_{\lambda_{\wt{v}}}$ (in the sense of Definition \ref{defn: l-adic
    hodge type} below).
\end{enumerate}

\subsection{Local deformation rings}
\label{subsec: local deformation rings}

Let $l$ be a prime number and $K$ be a finite extension of
$\Ql$ with residue field $k$ and ring of integers $\bigO$, and write
$\mathfrak{m}_\mc{O}$ for the maximal ideal of $\bigO$. Let
$\mathcal{C}_\bigO$ be the category of complete local Noetherian $\bigO$-algebras
with residue field isomorphic to $k$ via the structural homomorphism. As in section 3 of \cite{BLGHT}, we consider an object $R$ of $\mc{C}_{\mc{O}}$ to be geometrically integral if for all finite extensions $K'/K$, the algebra $R\otimes_{\mc{O}}\mc{O}_{K'}$ is an integral domain.

Let $M$ be a finite extension of $\bb{Q}_p$ for some prime $p$
possibly equal to $l$ and let $\rhobar : G_M \rightarrow \GL_n(k)$ be
a continuous homomorphism.
Then the functor from $\mathcal{C}_\bigO$ to
$Sets$ which takes $A\in\mathcal{C}_\bigO$ to the set of continuous
liftings $\rho:G_M\to\GL_n(A)$ of $\rhobar$  is represented by
a complete local Noetherian $\bigO$-algebra $R^{\square}_{\rhobar}$. We call this
ring the universal $\mc{O}$-lifting ring of $\rhobar$.
 We write
$\rho^{\square}:G_M\to\GL_n(R^{\square}_{\rhobar})$ for the universal
lifting. 

The following definitions will prove to be useful later.
\begin{defn} \label{defn:inherited basis}
Suppose $\rho: G_M\to\GL_2(\bigO)$ is a representation.
We can think of this as putting a $G_M$-action on the vector space
$K^2$ (=$V$, say), in a way that stabilizes the lattice generated by the standard 
basis $\{e_0,e_1\}$, where $e_0=\langle1,0\rangle$, $e_1=\langle0,1\rangle$.
Considering
$\Sym^{n-1}\rho$ as a quotient of
$V^{\otimes(n-1)}$, we have an ordered basis
$\{g_0,\dots,g_{n-1}\}$ of $\Sym^{n-1}V$, where
$g_i$ is the image of $e_0^{\otimes(n-1-i)}\otimes e_1^{\otimes
  i}$. We call this the $\mc{O}$-basis of $\Sym^{n-1} V$ \emph{inherited}
  from our original basis in $\rho$.
\end{defn}

\begin{defn} \label{defn:inherited basis for otimes}
Suppose $\rho: G_M\to\GL_n(\bigO), \rho': G_M\to\GL_{n'}(\bigO)$
are representations, which
we think of as putting $G_M$-actions on the vector spaces $V_\rho=K^n$,
$V'_\rho=K^{n'}$ in a way that stabilizes the lattices generated by the standard bases of each.
In this situation we have an ordered
$\mc{O}$-basis on $ V_\rho \otimes_{\mc{O}} 
 V_\rho'$ given by the vectors $e_j\otimes f_k$, ordered lexicographically, where the $e_j$ are the standard 
 $\mc{O}$-basis in $V_\rho$ and the $f_k$ are the standard basis in
 $V_\rho'$. We call this the $\mc{O}$-basis of $ \rho \otimes_{\mc{O}} 
 \rho'$ \emph{inherited} from our original bases.
\end{defn}

\subsubsection{Local deformations ($p=l$ case)}

Suppose that $p=l$. In this section we will define an equivalence relation
on crystalline lifts of $\rhobar$.  For this, we need to consider
certain quotients of $R^\square_{\rhobar}$. Assume that $K$ contains the image of every embedding $M \into \overline{K}$.

\begin{defn} 
  Let $(\bb{Z}^{n}_{+})^{\Hom(M,K )}$ denote the subset of
  $(\bb{Z}^{n})^{\Hom(M,K )}$ consisting of elements $\lambda$ which
  satisfy
  \[ \lambda_{\tau,1} \geq \lambda_{\tau,2} \geq \ldots \geq
  \lambda_{\tau,n} \] for every embedding $\tau$.
\end{defn}

Let $\lambda$ be an element of $(\bb{Z}^{n}_{+})^{\Hom(M,K)}$.  We
associate to $\lambda$ an $l$-adic Hodge type $\mathbf{v}_{\lambda}$
in the sense of section 2.6 of \cite{kisinpst} as follows. Let $D_{K}$
denote the vector space $K^{n}$. Let $D_{K,M}=D_{K} \otimes_{\bb{Q}_l}
M$. For each embedding $\tau : M \hookrightarrow K$, we let
$D_{K,\tau}=D_{K,M} \otimes_{K\otimes M,1\otimes \tau}K $ so that
$D_{K,M} = \oplus_{\tau} D_{K,\tau}$. For each $\tau$ choose a
decreasing filtration $\Fil^{i}D_{K,\tau}$ of $D_{K,\tau}$ so that
$\dim_{K} \gr^{i}D_{K,\tau} = 0 $ unless $ i =
(j-1)+\lambda_{\tau,n-j+1}$ for some $j=1,\ldots,n$ in which case
$\dim_{K} \gr^{i}D_{K,\tau}=1$. We define a decreasing filtration of
$D_{K,M}$ by $K \otimes_{\bb{Q}_l} M$-submodules by setting
\[ \Fil^{i} D_{K,M} = \oplus_{\tau} \Fil^{i}D_{K,\tau}.\] Let
$\mathbf{v}_{\lambda}= \{ D_{K}, \Fil^{i}D_{K,M} \}$. 

We now recall some results of Kisin. Let $\lambda$ be an element of $(\bb{Z}^{n}_{+})^{\Hom(M,K)}$ and let
$\mathbf{v}_{\lambda}$ be the associated $l$-adic Hodge type.
\begin{defn}\label{defn: l-adic hodge type}
  If $B$ is a finite $K$-algebra and $V_{B}$ is a free $B$-module of
  rank $n$ with a continuous action of $G_{M}$ that makes $V_B$ into a
  de Rham representation, then we say that \emph{$V_B$ is of $l$-adic Hodge
  type $\mathbf{v}_{\lambda}$} if for each $i$ there is an isomorphism
  of $B \otimes_{\bb{Q}_l} M$-modules
\[ \gr^{i}(V_{B} \otimes_{\bb{Q}_l} B_{dR})^{G_{M}}
\tilde{\rightarrow} (\gr^{i}D_{K,M}) \otimes_{K} B. \] 
\end{defn}

Corollary 2.7.7 of \cite{kisinpst} implies that there is a unique
$l$-torsion free quotient $R^{\mathbf{v}_{\lambda},cr}_{\rhobar}$ of
$R^{\square}_{\rhobar}$ with the property that for any finite
$K$-algebra $B$, a homomorphism of $\mc{O}$-algebras $\zeta :
R^{\square}_{\rhobar}\rightarrow B$ factors through
$R^{\mathbf{v}_{\lambda},cr}_{\rhobar}$ if and only if $\zeta \circ
\rho^{\square}$ is crystalline of $l$-adic Hodge type
$\mathbf{v}_{\lambda}$. Moreover, Theorem 3.3.8 of \cite{kisinpst}
implies that $\Spec R^{\mathbf{v}_{\lambda},cr}_{\rhobar}[1/l]$ is
formally smooth over $K$ and equidimensional of dimension $n^2 +
\frac{1}{2}n(n-1)[M:\bb{Q}_l]$.

By Lemma 3.3.3 of \cite{ger}
there is a quotient $R^{\triangle_{\lambda},cr}_{\rhobar}$ of
  $R^{\mathbf{v}_{\lambda},cr}_{\rhobar}$ corresponding to a union of
  irreducible components such that for any finite extension $E$ of
  $K$, a homomorphism of $\mc{O}$-algebras $\zeta :
  R^{\mathbf{v}_{\lambda},cr}_{\rhobar} \rightarrow E$ factors through
  $R^{\triangle_{\lambda},cr}_{\rhobar}$ if and only if $\zeta \circ
  \rho^{\square}$ is crystalline and ordinary of weight $\lambda$.

We now introduce an equivalence relation on continuous representations
$G_M\to\GL_n(\bigO)$ lifting $\rhobar$.

\begin{defn}\label{defn:tilde}Suppose that $\rho_1,\rho_2 : G_M\to\GL_n(\bigO)$ are two continuous lifts
of $\rhobar$. Then we say that
$\rho_1\sim \rho_2$ if the following hold.
\begin{enumerate}
\item There is a $\lambda\in(\bb{Z}^{n}_{+})^{\Hom(M,K)}$
such that $\rho_1$ and $\rho_2$ both correspond to points of
$R^{\mathbf{v}_{\lambda},cr}_{\rhobar}$ (that is,
$\rho_1\otimes_\bigO K$ and $\rho_2\otimes_\bigO K$ are both crystalline of $l$-adic Hodge
  type $\mathbf{v}_{\lambda}$).
\item For every minimal prime ideal $\wp$ of
  $R^{\mathbf{v}_{\lambda},cr}_{\rhobar}$, the quotient
  $R^{\mathbf{v}_{\lambda},cr}_{\rhobar}/\wp$ is geometrically
  integral.
\item $\rho_1$ and $\rho_2$ give rise to closed points on a common irreducible
  component of  $\Spec R^{\mathbf{v}_{\lambda},cr}_{\rhobar}[1/l]$.
\end{enumerate}
\end{defn}

In (3) above, note that the irreducible component is uniquely
determined by either of $\rho_1$, $\rho_2$ because $\Spec
R^{\mathbf{v}_{\lambda},cr}_{\rhobar}[1/l]$ is formally smooth. Note
also that we can always ensure that (2) holds by replacing $\mc{O}$ with the
ring of integers in a finite extension of $K$.

Suppose that $\rho_1 \sim \rho_2$ as above and let $M'/M$ be a finite
extension. Assume that $K$ contains the image of every embedding $M'
\into \overline{K}$. Then we claim that $\rho_1|_{G_{M'}} \sim
\rho_2|_{G_{M'}}$. Indeed, let $\lambda$ be such that $\rho_1$ and
$\rho_2$ have $l$-adic Hodge type $\mathbf{v}_{\lambda}$. Define
$\lambda' \in (\bb{Z}^{n}_{+})^{\Hom(M',K)}$ by $\lambda'_{\tau} =
\lambda_{\tau|_{M}}$ for all $\tau : M' \into K$. Then restriction to
$G_{M'}$ gives rise to an $\mc{O}$-algebra homomorphism
$R^{\square}_{\rhobar|_{G_{M'}}} \to
R^{\mathbf{v}_{\lambda},cr}_{\rhobar}$ which factors through
$R^{\mathbf{v}_{\lambda'},cr}_{\rhobar|_{G_{M'}}}$ (using the fact
that $R^{\mathbf{v}_{\lambda},cr}_{\rhobar}$ is reduced and
$l$-torsion free). The result now follows from the formal smoothness
of $\Spec R^{\mathbf{v}_{\lambda'},cr}_{\rhobar|_{G_{M'}}}[1/l]$, which
implies that the image of any irreducible component of $\Spec
R^{\mathbf{v}_{\lambda},cr}_{\rhobar}[1/l]$ is contained in a unique
irreducible component of $\Spec
R^{\mathbf{v}_{\lambda'},cr}_{\rhobar|_{G_{M'}}}[1/l]$.

In a similar vein, it follows that if $n=2$ and $\rho_1 \sim \rho_2$,
then $\Sym^{k-1} \rho_1 \sim \Sym^{k-1} \rho_2$ for all $k \geq 1$,
where we take the $\mc{O}$-basis on the $\Sym^{k-1} \rho_i$ inherited
from the bases we have on the $\rho_i$, in the sense of Definition
\ref{defn:inherited basis}. [The same is true if $n>2$, with an
appropriate modification of Definition \ref{defn:inherited basis}.]

We will make one final variation on this theme. Suppose 
$\rho' : G_{M} \rightarrow \GL_m(\mc{O})$ is crystalline of $l$-adic 
Hodge type $\mathbf{v}_{\lambda'}$ for some $m$ and some $\lambda' \in 
(\bb{Z}^{m}_{+})^{\Hom(M,K)}$, and $\rho_1 \sim \rho_2$ are as above.
(Note $n$ need no longer be 2.) 
Then
$ \rho_1 \otimes_{\mc{O}} 
 \rho' \sim \rho_2 \otimes_{\mc{O}} 
 \rho'$, where we take as $\mc{O}$-basis on the $ \rho_i \otimes_{\mc{O}} 
 \rho'$ the inherited bases in the sense of Definition \ref{defn:inherited basis for otimes}.

 \begin{lem}
   \label{lem: sim preserved under extension of bigO}
   Let $\rhobar:G_M \rightarrow \GL_n(k)$ be a continuous
   homomorphism. Suppose $\rho_1,\rho_2:G_M \rightarrow \GL_n(\mc{O})$
   are two lifts of $\rhobar$ with $\rho_1 \sim \rho_2$. If $\mc{O}'$
   denotes the ring of integers in a finite extension of $K$ with
   residue field $k'$, then
   $\rho_1 \sim \rho_2$, regarded as lifts of $\rhobar \otimes_k k'$
   to $\GL_n(\mc{O}')$. 
 \end{lem}

 \begin{proof}
   Let $\lambda \in (\bb{Z}^n_+)^{\Hom(M,K)}$ be such that $\rho_1$
   and $\rho_2$ have $l$-adic Hodge type $\mathbf{v}_{\lambda}$. Let
   $R=R^{\mathbf{v}_{\lambda},cr}_{\rhobar}$ and $R'=
   R\otimes_{\mc{O}}\mc{O}'$. We need to show that $\rho_1$ and
   $\rho_2$ give rise to closed points of $\Spec R'[1/l]$ lying on a
   common component. Note that if $\mc{C}'$ is an irreducible
   component of $\Spec R'[1/l]$, then the image of $\mc{C}'$ in $\Spec R[1/l]$
   is an irreducible component. Indeed, the image of $\mc{C}'$ in
   $\Spec R[1/l]$ is irreducible and closed (as $R\rightarrow R'$ is
   finite). If $x'$ is a closed point of $\Spec R'[1/l]$ lying in
   $\mc{C}'$ with image $x$ in $\Spec R[1/l]$, then the completed
   local rings of $\Spec R'[1/l]$ and $\Spec R[1/l]$ at $x'$ and $x$ respectively
   are isomorphic. We deduce that the image of $\mc{C}'$ has the same
   dimension as $\mc{C}'$ and hence is an irreducible component.

   Now, let $x_1$ and $x_2$ denote the closed points of $\Spec R[1/l]$
   corresponding to $\rho_1$ and $\rho_2$ and let $\mc{C}$ denote the
   irreducible component of $\Spec R[1/l]$ containing $x_1$ and
   $x_2$. Then we claim that the preimage of $\mc{C}$ in $\Spec
   R'[1/l]$ is irreducible. Indeed, suppose there are two distinct
   irreducible components $\mc{C}'$ and $\mc{C}''$ of $\Spec R'[1/l]$
   mapping to $\mc{C}$. Then there are points $x_1'$ and $x_1''$ of
   $\mc{C}'$ and $\mc{C}''$ respectively mapping to $x_1$. However,
   the preimage of $x_1$ in $\Spec R'[1/l]$ consists of a single
   point (let $\mf{m}$ denote the maximal ideal of $R[1/l]$
   corresponding to $x_1$. Then the fibre over $x_1$ is given by the
   spectrum of
   $(R[1/l]/\mf{m})\otimes_{\mc{O}}\mc{O}' \cong K
   \otimes_{\mc{O}}\mc{O}' \cong K'$.) Thus $x_1'=x_1''$ lies in the
   intersection of $\mc{C}'$ and $\mc{C}''$, contradicting the formal
   smoothness of $\Spec R'[1/l]$. 
 \end{proof}

\subsubsection{Local deformations ($p\neq l$ case)}
Suppose now that $p \neq l$.
 By Theorem 2.1.6 of \cite{gee061}, $\Spec R^{\square}_{\rhobar}[1/l]$ is equidimensional of
dimension $n^2$.

\begin{defn}
\label{defn: leadsto}
  Let $\rho_1,\rho_2 : G_M \rightarrow
  \GL_n(\mc{O})$ be two lifts of $\rhobar$. We say that $\rho_1
  \leadsto_{\mc{O}} \rho_2$ if the following hold.
 \begin{enumerate}
 \item For each minimal prime ideal $\wp$ of
   $R^{\square}$, the quotient
   $R^{\square}/\wp$ is geometrically irreducible.
  \item $\rho_1$ corresponds to a closed point of $\Spec
    R^{\square}_{\rhobar}[1/l]$ which is contained in a unique
    irreducible component and this irreducible component also contains
    the closed point corresponding to $\rho_2$.
  \end{enumerate}
\end{defn}
We remark that, we can always replace
$\mc{O}$ by the ring of integers in a finite extension of $K$ so that
condition (1) above holds. Also, condition (1) ensures that if $\rho_1
\leadsto_{\mc{O}} \rho_2$ and if
$\mc{O}'$ is the ring of integers in a finite extension of $K$ then
$\rho_1 \leadsto_{\mc{O}'} \rho_2$.

\subsection{Properties of $\leadsto_{\bigO}$ and $\sim$}

\begin{lem} $\sim$ is an equivalence relation.
\end{lem}
\begin{proof}
This follows immediately from the definitions.
\end{proof}

\begin{lem}
\label{lem: unramified twists lie on same components}
  Let $M$ be a finite extension of $\bb{Q}_p$ for some prime $p$. Let
  $\rhobar : G_M \rightarrow \GL_n(k)$ be a continuous
  homomorphism. If $p\neq l$, let $R=R^{\square}_{\rhobar}$. If $p=l$,
  assume that $K$ contains the image of each embedding $M \into
  \Qlbar$ and
  let $R=R^{\mathbf{v}_{\lambda},cr}_{\rhobar}$ for some $\lambda \in
  (\bb{Z}^n_+)^{\Hom(M,K)}$. Let $\mc{O}'$ denote the ring of integers
  in a finite extension of $K$.
  Let $\rho$ and $\rho^{\prime}$ be two lifts of $\rhobar$
  to $\mc{O}'$ giving rise to closed points of $\Spec R[1/l]$.
 Suppose that after conjugation by an element of
  $\ker(GL_n(\mc{O'})\rightarrow \GL_n(\mc{O}'/\mf{m}_{\mc{O}'}))$ they differ by an
  unramified twist. Then an irreducible component of $\Spec
  R[1/l]$ contains $\rho$ if and only if it contains $\rho'$.
\end{lem}

\begin{proof}
  The universal unramified $\mc{O}$-lifting ring of the trivial
  character $G_M \rightarrow k^{\times}$ is given by $\mc{O}[[Y]]$
  where the universal lift $\chi^{\square}$ sends $\Frob_M$ to $1+Y$.
  Let $R[[Y, \underline{X}]] =
  R[[Y]][[X_{ij}:1\leq i,j \leq n]]$. Let
  $\rho^{\square}$ denote the universal lift of $\rhobar$ to $R$. Consider
  the lift $(1_n + (X_{ij}))\rho^{\square}(1_n+(X_{ij}))^{-1} \otimes
  \chi^{\square}$ of $\rhobar$ to $R[[Y,
  \underline{X}]]$. This lift gives rise to a homomorphism
  $R^{\square}_{\rhobar} \rightarrow R[[Y,
  \underline{X}]]$ which factors through $R$. Let $\alpha$ denote the resulting $\mc{O}$-algebra homomorphism $R \rightarrow R[[Y,\underline{X}]]$. Let $\iota : R \rightarrow R[[Y,\underline{X}]]$ be the standard $R$-algebra structure on $R[[Y,\underline{X}]]$.

  The minimal prime ideals of $R[[Y,\underline{X}]]$ and $R$ are in
  natural bijection (if $\wp$ is a minimal prime of $R$ then
  $\iota(\wp)$ generates a minimal prime of
  $R[[Y,\underline{X}]]$). Let $\wp$ be a minimal prime of $R$. We
  claim that the kernel of the map $\beta : R \rightarrow
  R[[Y,\underline{X}]]/\iota(\wp)=(R/\wp)[[Y,\underline{X}]]$ induced
  by $\alpha$ is
  $\wp$. Indeed, the $R$-algebra homomorphism (with
  $R[[Y,\underline{X}]]$ considered as an $R$-algebra via $\iota$)
  $\gamma : R[[Y,\underline{X}]] \rightarrow R$ which sends $Y$ and
  each $X_{ij}$ to $0$ is a section to the map $\beta$. The
  composition $ R \stackrel{\beta}{\rightarrow}
  (R/\wp)[[Y,\underline{X}]] \stackrel{\gamma}{\rightarrow} R/\wp$ is
  thus the natural reduction map. In particular its kernel is
  $\wp$. Since $\ker(\beta) \subset \ker(\gamma \circ \beta)=\wp$ and
  $\wp$ is minimal, we deduce $\ker(\beta)=\wp$. The lemma follows.
\end{proof}

\begin{lem}
  \label{lem: residually trivial crystalline, all you need is
    GLn-conjugacy}Let $M$ be a finite extension of $\Ql$.
    Let $\rhobar:
  G_M\to\GL_n(k)$ be the trivial representation, and let $\rho$ and
  $\rho': G_{M}\rightarrow \GL_n(\bigO)$ be two crystalline lifts of $\rhobar$ of
  $l$-adic Hodge type $\mathbf{v}_\lambda$ which are
  $\GL_n(\bigO)$-conjugate. Then $\rho\sim\rho'$.
\end{lem}
\begin{proof}
  Take $g\in\GL_n(\bigO)$ with $\rho'=g\rho g^{-1}$. Let
  $A=\bigO\langle X_{ij},Y\rangle/(Y\det(X_{ij})-1)$ where
  $\bigO\langle X_{ij},Y\rangle$ is the $\m_\bigO$-adic completion of
  $\bigO[X_{ij},Y]$. Let $\rho_A:G_M\to\GL_n(A)$ be given by $X\rho
  X^{-1}$, where $X$ is the matrix $(X_{ij})$. By Lemma 3.3.1 of
  \cite{ger}, there is a continuous homomorphism
  $R_{\rhobar}^\square\to A$ such that $\rho_A$ is the push-forward of
  the universal lifting
  $\rho^\square:G_M\to\GL_n(R_{\rhobar}^\square)$. Now, for any
  $\Qlbar$-point of $A$, the corresponding specialisation of $\rho_A$
  is a $\Qlbar$-conjugate of $\rho$, and is thus crystalline of
  $l$-adic Hodge type $\mathbf{v}_\lambda$, so corresponds to a
  $\Qlbar$-point of $R_{\rhobar}^{\mathbf{v}_\lambda,cr}$. Since the
  $\Qlbar$-points of $A$ are dense in $\Spec A$, we conclude that the
  homomorphism $R_{\rhobar}^\square\to A$ factors through
  $R_{\rhobar}^{\mathbf{v}_\lambda,cr}$.

Now, $\Spec A$ is irreducible, and the points $x$ and $x'$ of
$\Spec R_{\rhobar}^{\mathbf{v}_\lambda,cr}$ corresponding to $\rho$ and
$\rho'$ respectively are in the image of the map $\Spec A\to \Spec
R_{\rhobar}^{\mathbf{v}_\lambda,cr}$, because they correspond to
specialising the matrix $X$ to the matrices $1_n$ and $g$
respectively. The result follows.
\end{proof}
\begin{cor}\label{cor: sums of unramified twists of chars tilde}
  Let $M$ be a finite extension of $\Ql$. 
    Let $\rhobar: G_M\to\GL_n(k)$ be the trivial
  representation, and let $\rho, \rho' :G_{M} \rightarrow \GL_n(\mc{O})$ be two crystalline lifts
  of $\rhobar$ which are both $\GL_n(\mc{O})$-conjugate to direct sums of unramified twists of a
  common set of crystalline characters. Then $\rho\sim\rho'$.

\end{cor}
\begin{proof} After applying Lemma \ref{lem: residually trivial crystalline, all you need is
    GLn-conjugacy}, we may assume
  that \[\rho=\oplus_{i=1}^n\rho_i\]and \[\rho'=\oplus_{i=1}^n\rho'_i\]
where $\rho'_i$ and $\rho_i$ are crystalline characters $
G_{M}\rightarrow \mc{O}^{\times}$ which differ by an
unramified twist for each $i$ and reduce to the trivial character
modulo $\mf{m}_{\mc{O}}$.
  It suffices to check that the corresponding points $x$ and $x'$ of $R_{\rhobar}^{\mathbf{v}_\lambda,cr}$  are path-connected. 

As in the proof of Lemma \ref{lem: unramified twists lie on same
  components}, the universal unramified $\bigO$-lifting ring of the
trivial character $G_M\to k^\times$ is given by $\bigO[[Y]]$ with the
universal lifting $\chi^\square$ sending $\Frob_M$ to $1+Y$. Taking
$n$ copies of this character, we obtain a
lifting \[\oplus_{i=1}^n\rho_i\otimes\chi^\square_i\]of $\rhobar$ to
$\mc{O}[[Y_1,\dots,Y_n]]$, and thus a continuous map
$\Spec \bigO[[Y_1,\dots,Y_n]]\to\Spec R_{\rhobar}^{\mathbf{v}_\lambda,cr}$. Both $x$ and $x'$
are in the image of this map, so the result follows.
\end{proof}

The following is Lemma 3.4.3 of \cite{ger} (recall that the ring
$R^{\triangle_{\lambda},cr}_{\rhobar}$ is defined below Definition
\ref{defn: l-adic hodge type}, and is universal for crystalline
ordinary lifts of weight $\lambda$).

\begin{lem}
  \label{lem: crystalline ordinary lifting ring is irreducible when
    residually trivial}
    Suppose $M$ is a finite extension of $\Ql$ and $\rhobar: G_M
\rightarrow \GL_n(k)$ is the trivial representation. If the ring
$R^{\triangle_{\lambda},cr}_{\rhobar}$ is non-zero, then it is
irreducible. 
\end{lem}

\begin{lem}
\label{lem: nice twists preserve sim and leadsto}
Let $M$ be a finite extension of $\bb{Q}_p$ for some prime $p$ and let
$\rhobar : G_M \rightarrow \GL_n(k)$ be a continuous homomorphism. Let
$\rho_1,\rho_2 : G_M \rightarrow \GL_n(\mc{O})$ be two lifts of
$\rhobar$. If $p \neq l$, suppose that $\rho_1 \leadsto_{\mc{O}}
\rho_2$ and $\rho_2 \leadsto_{\mc{O}} \rho_1$. (Equivalently, assume that
for each minimal prime $\mf{p}$ of $R^\square$ the quotient
$R^\square/\mf{p}$ is geometrically irreducible, and $\rho_1$,
$\rho_2$ each correspond to closed points contained in a common
irreducible component of $\Spec R_{\rhobar}^\square[1/l]$, and neither
point is contained in any other irreducible component.) 
If $p=l$, assume that $\rho_1 \sim \rho_2$.
Let $\chi_1,\chi_2 : G_M
\rightarrow \mc{O}^{\times}$ be continuous characters with
$\overline{\chi}_1 = \overline{\chi}_2$ and $\chi_1|_{I_M} =
\chi_2|_{I_M}$. Suppose in addition that if $p=l$ then $\chi_1$ and
$\chi_2$ are crystalline. Then $\chi_1 \rho_1 \leadsto_{\mc{O}} \chi_2
\rho_2$ if $p \neq l$ and $\chi_1 \rho_1 \sim \chi_2 \rho_2$
if $p=l$.
\end{lem}

\begin{proof}
  We treat the case $p \neq l$, the other case being similar. Let
  $\overline{\chi} = \overline{\chi}_1 = \overline{\chi}_2$. Then the
  operation of twisting by $\chi_1$ defines an isomorphism of the
  lifting problems of $\rhobar$ and $\overline{\chi}\rhobar$. It
  therefore defines an isomorphism
  $R^{\square}_{\overline{\chi}\rhobar}\isoto R^{\square}_{\rhobar}$.
  It follows that $\chi_1 \rho_1 \leadsto_{\mc{O}} \chi_1 \rho_2$ and
  that $\chi_1 \rho_2$ gives rise to a closed point of $\Spec
  R^{\square}_{\overline{\chi}\rhobar}[1/l]$ lying on a unique
  irreducible component. Since $\chi_1$ and $\chi_2$ differ by a
  residually trivial unramified twist, an easy argument shows that
  this component also contains $\chi_2 \rho_2$ (c.f.\ the proof of
  Corollary \ref{cor: sums of unramified twists of chars tilde}). It
  follows that $\chi_1 \rho_1 \leadsto_{\mc{O}}\chi_2 \rho_2$.
\end{proof}

\subsection{Global deformation rings}
\label{subsec: global deformation rings}

Let $F/F^{+}$ be a totally imaginary quadratic extension of a totally
real field $F^+$. Let $c$ denote the non-trivial element of $\Gal(F/F^+)$. Let $k$ denote a finite field of characteristic $l$ and $K$ a
finite extension of $\bb{Q}_l$, inside our fixed algebraic closure $\Qlbar$, with ring of integers $\mc{O}$ and
residue field $k$. Assume that $K$ contains the image of every
embedding $F \into \Qlbar$ and that the prime $l$ is odd. Assume that
every place in $F^+$ dividing $l$ splits in $F$. Let $S$ denote a finite set of finite places of $F^+$ which split in
$F$, and assume that $S$ contains every place dividing $l$. Let $S_l$ denote the set of places of $F^+$ lying over $l$. Let $F(S)$
denote the maximal extension of $F$ unramified away from $S$. Let
$G_{F^{+},S}=\Gal(F(S)/F^{+})$ and $G_{F,S}=\Gal(F(S)/F)$. For each $v
\in S$ choose a place $\wt{v}$ of $F$ lying over $v$ and let $\wt{S}$
denote the set of $\wt{v}$ for $v \in S$. For each place $v|\infty$ of $F^+$ we let
$c_v$ denote a choice of a complex conjugation at $v$ in
$G_{F^+,S}$.  For each place $w$ of $F$
we have a $G_{F,S}$-conjugacy class of homomorphisms $G_{F_w}
\rightarrow G_{F,S}$. For $v \in S$ we fix a choice of homomorphism
$G_{F_{\wt{v}}} \rightarrow G_{F,S}$.

Fix a continuous homomorphism
\[ \rbar : G_{F^+,S} \rightarrow \mc{G}_n(k) \] such that $G_{F,S} =
\rbar^{-1}(\GL_n(k)\times \GL_1(k))$ and fix a continuous character
$\chi : G_{F^+,S}\rightarrow \mc{O}^{\times} $ such that $\nu \circ
\rbar = \overline{\chi}$. Assume that $\rbar|_{G_{F,S}}$ is absolutely
irreducible. 
As in Definition 1.2.1 of \cite{cht}, we define
\begin{itemize}
\item a \emph{lifting} of $\rbar$ to an object $A$ of
  $\mc{C}_{\mc{O}}$ to be a continuous homomorphism $r : G_{F^+,S}
  \rightarrow \mc{G}_n(A)$ lifting $\rbar$ and with $\nu \circ r =
  \chi$;
\item two liftings $r$, $r^{\prime}$ of $\rbar$ to $A$ to be
  \emph{equivalent} if they are conjugate by an element of
  $\ker(\GL_n(A)\rightarrow \GL_n(k))$;
\item a \emph{deformation} of $\rbar$ to an object $A$ of
  $\mc{C}_{\mc{O}}$ to be an equivalence class of liftings.
\end{itemize}
Similarly, if $T\subset S$, we define
\begin{itemize}
\item a \emph{$T$-framed lifting} of $\rbar$ to $A$ to be a tuple
 $(r,\{ \alpha_v\}_{v \in T})$ where $r$ is a lifting of $\rbar$ and
 $\alpha_v \in \ker(\GL_n(A)\rightarrow \GL_n(k))$ for $v \in T$;
\item two $T$-framed liftings  $(r,\{ \alpha_v\}_{v \in T})$,
 $(r^{\prime},\{ \alpha^{\prime}_v\}_{v \in T})$ to be
 \emph{equivalent} if there is an element $\beta \in
 \ker(\GL_n(A)\rightarrow \GL_n(k))$ with $r^{\prime}=\beta r
 \beta^{-1}$ and $\alpha_v^{\prime}=\beta \alpha_v$ for $v \in T$;
\item a \emph{$T$-framed deformation} of $\rbar$ to be an equivalence
 class of $T$-framed liftings.
\end{itemize}

For each place $v \in S$, let $R^{\square}_{\rbarwtv}$ denote
the universal $\mc{O}$-lifting ring of $\rbar|_{G_{F_{\wt{v}}}}$ and
let $R_{\wt{v}}$ denote a quotient of $R^{\square}_{\rbarwtv}$ which satisfies the
following property: 
\begin{itemize}
\item[(*)] let $A$ be an object of $\mc{C}_{\mc{O}}$ and let
$\zeta,\zeta^{\prime}: R^{\square}_{\rbarwtv}\rightarrow A$ be homomorphisms corresponding to
two lifts $r$ and $r^{\prime}$ of $\rbarwtv$ which are
conjugate by an element of $\ker(\GL_n(A)\rightarrow
\GL_n(k))$. Then $\zeta$ factors through $R_{\wt{v}}$ if and only if
$\zeta^{\prime}$ does.
\end{itemize}
We
consider the \emph{deformation problem}
\[ \mc{S} = (F/F^{+},S,\wt{S},\mc{O},\rbar,\chi,\{ R_{\wt{v}}\}_{v \in
  S}) \] 
(see sections 2.2 and 2.3 of \cite{cht} for this terminology).
We say that a lifting $r : G_{F^+,S}\rightarrow \mc{G}_n(A)$ is
\emph{of type $\mc{S}$} if for each place $v \in S$, the homomorphism
$R^{\square}_{\rbarwtv} \rightarrow A$ corresponding to $r|_{G_{F_{\wt{v}}}}$ factors
through $R_{\wt{v}}$. We also define deformations of type $\mc{S}$ in the same way.

Let $\Def_{\mc{S}}$ be the functor $ \mc{C}_{\mc{O}}\rightarrow Sets$
which sends an algebra $A$ to the set of deformations of $\rbar$ to
$A$ of type $\mc{S}$. Similarly, if $T\subset S$, let
$\Def_{\mc{S}}^{\square_T}$ be the functor $\mc{C}_{\mc{O}}\rightarrow
Sets$ which sends an algebra $A$ to the set of $T$-framed liftings of
$\rbar$ to $A$ which are of type $\mc{S}$. 
By Proposition 2.2.9 of
\cite{cht} these functors are represented by objects
$R_{\mc{S}}^{\univ}$ and $R^{\square_T}_{\mc{S}}$ respectively
of $\mc{C}_{\mc{O}}$.

\begin{lem}
  \label{lem: property (*)}
  Let $M$ be a finite extension of $\bb{Q}_p$ for some prime $p$. Let
  $\rhobar : G_M \rightarrow \GL_n(k)$ be a continuous
  homomorphism. If $p\neq l$, let $R$ be the maximal $l$-torsion free
  quotient of $R^{\square}_{\rhobar}$. If $p=l$,
  assume that $K$ contains the image of each embedding $M \into
  \Qlbar$ and
  let $R=R^{\mathbf{v}_{\lambda},cr}_{\rhobar}$ for some $\lambda \in
  (\bb{Z}^n_+)^{\Hom(M,K)}$. Then $R$ satisfies property $(*)$ above.
\end{lem}

\begin{proof}
  This can be proved in exactly the same way as Lemma \ref{lem: unramified twists lie on same components}.
\end{proof}

\subsection{Automorphy lifting}
\label{subsec: automorphy lifting}

\subsubsection{CM Fields}
\label{subsubsec: CM fields}

\begin{thm}
\label{thm: CM modularity lifting theorem, using tilde}
  Let $F$ be an imaginary CM field with totally real subfield $F^+$
  and let $c$ be the non-trivial element of $\Gal(F/F^+)$. Let $n\in
  \bb{Z}_{\geq 1}$ and let $l>n$ be a prime. Let $K\subset \Qlbar$ denote a finite extension of $\bb{Q}_l$ with ring of integers
$\mc{O}$ and residue field $k$. Assume that $K$ contains the image of
every embedding $F \into \Qlbar$.
  Let \[\rho:G_F \to\GL_n(\mc{O})\]be a continuous 
  representation and let $\rhobar = \rho \mod
  \mf{m}_{\mc{O}}$. Suppose that $\rho$ enjoys the following
  properties:
  \begin{enumerate}
  \item $\rho^c\cong \rho^\vee \epsilon^{1-n}$.
  \item The reduction $\rhobar$ is absolutely irreducible and 
 $\rhobar(G_{F(\zeta_l)}) \subset \GL_n(k)$ is big (see Definition
 \ref{defn: m-big}). 
  \item $(\overline{F})^{\ker\ad\rhobar}$ does not contain $\zeta_l$.
  \item There is a continuous representation
    $\rho':G_{F}\to\GL_n(\mc{O})$, a RACSDC automorphic representation
    $\pi$ of $\GL_n(\bb{A}_F)$ which is unramified above $l$ and $\iota: \Qlbar \isoto \bb{C}$
 such that
    \begin{enumerate}
    \item $\rho' \otimes_{\mc{O}}\Qlbar \cong r_{l,\iota}(\pi) : G_F
      \rightarrow \GL_n(\Qlbar)$.
    \item  $\rhobar=\rhobar'$.
    \item For all places $v\nmid l$ of $F$, either
      \begin{itemize}
      \item $\rho|_{G_{F_v}}$ and $\pi_v$ are both
        unramified, or
      \item  $\rho'|_{G_{F_v}}\leadsto_{\mc{O}} \rho|_{G_{F_v}}$.
      \end{itemize}
    \item For all places $v|l$, $\rho|_{G_{F_v}} \sim
      \rho'|_{G_{F_v}}$.
    \end{enumerate}
  \end{enumerate}
Then $\rho$ is automorphic.
\end{thm}

\begin{proof}
  Choose a place $v_1$ of $F$
  not dividing $l$ such that
  \begin{itemize}
  \item $v_1$ is unramified over a rational prime $p$ with
    $[F(\zeta_p): F]>n$;
  \item $v_1$ does not split completely in $F(\zeta_l)$;
  \item $\rho$ and $\pi$ are unramified at $v_1$;
  \item $\ad \rhobar (\Frob_{v_1})= 1$.
  \end{itemize}

  Extending $\mc{O}$ if necessary, choose an imaginary CM field $L/F$  such that:
  \begin{itemize}
  \item $L/F$ is solvable;
  \item $L$ is linearly disjoint from $\overline{F}^{\ker 
      \rbar}(\zeta_l)$ over $F$;
  \item $4 | [L^+:F^+]$ where $L^+$ denotes the maximal totally real
    subfield of $L$;
  \item $L/L^+$ is unramified at all finite places;
  \item Every prime of $L$ dividing $l$ is split over $L^+$ and every
    prime where $\rho|_{G_L}$ or $\pi_L$ ramifies is split over $L^+$ (here
    $\pi_L$ denotes the base change of $\pi$ to $L$);
  \item Every place of $L$ over $v_1$ or  $cv_1$ is split over
    $L^+$. Moreover, $v_1$ and $cv_1$ split completely in $L$;
  \item Every place $v|l$ of $F$ splits completely in $L$;
  \item If $v \nmid l$ is a place of $F$ and at least one of
    $\rho|_{G_{F_v}}$ or $\pi_v$ is ramified, then $v$
    splits completely in $L$.
  \end{itemize}

  Let $G_{/\mc{O}_{L^+}}$
  be an algebraic group as in section \ref{subsec: l-adic aut forms on
    unitary groups} (with $F/F^+$ replaced by $L/L^+$).
  By
  Th\'eor\`eme 5.4 and Corollaire 5.3 of \cite{labesse} there exists
  an automorphic representation $\Pi$ of $G(\bb{A}_{L^+})$ such that
  $\pi_L$ is a strong base change of $\Pi$. Let $S_l$ denote the set of
  places of $L^+$ dividing $l$ and let $R$ denote the set of places of
  $L^+$ not dividing $l$ and lying under a place of $L$ where
  $\rho$ or $\pi_L$ is ramified. Let $S_a$ denote
  the set of places of $L^+$ lying over the restriction of $v_1$ to $F^+$.
  Let $T= S_l \coprod R
  \coprod S_a$. For each place $v \in T$, choose a place $\wt{v}$ of
  $L$ lying over it and let $\wt{T}$ denote the set of $\wt{v}$ for $v
  \in T$. Let $U=\prod_v U_v \subset G(\bb{A}_{L^+}^{\infty})$ be a
  compact open subgroup such that
 \begin{itemize}
\item $U_v = G(\mc{O}_{L^+_v})$ for $v \in S_l$ and for $v \not \in
   T$ split in $L$;
 \item $U_v$ is a hyperspecial maximal compact subgroup of $G(L^+_v)$
   for each $v$ inert in $L$;
\item $U_v$ is such that $\Pi_v^{U_v}\neq \{0\}$ for $v \in R$;
 \item $U_v = \ker(G(\mc{O}_{L^+_v})\rightarrow G(k_{v}))$ for $v \in S_a$.
 \end{itemize}

Extend $K$ if necessary so that it contains the image of
every embedding $L\into \Qlbar$. For each $v \in S_l$, let
$\lambda_{\wt{v}}$ be the element of $(\bb{Z}^{n}_+)^{\Hom(L_{\wt{v}},K)}$ with the property that 
$\rho|_{G_{L_{\wt{v}}}}$ and $\rho'|_{G_{L_{\wt{v}}}}$ have $l$-adic Hodge type
$\mathbf{v}_{\lambda_{\wt{v}}}$. Let $\wt{I}_l$ denote the set of
embeddings $L\into K$ giving rise to one of the places $\wt{v}$.
Let $\lambda = (\lambda_{\wt{v}})_{v
  \in S_l}$ regarded as an element of $(\bb{Z}^n_+)^{\wt{I}_l}$ in the
evident way and let $S_{\lambda}(U,\mc{O})$ be the space of $l$-adic
automorphic forms on $G$ of weight $\lambda$ introduced above. Let $\bb{T}^{T}_{\lambda}(U,\mc{O})$ be the $\mc{O}$-subalgebra
of $\End_{\mc{O}}(S_{\lambda}(U,\mc{O}))$ generated by the Hecke operators
$T_{w}^{(j)}$, $(T_{w}^{(n)})^{-1}$ for $w$ a place of $L$ split over
$L^+$, not lying over $T$ and $j=1,\ldots,n$. The eigenvalues of
the operators $T_{w}^{(j)}$ on the space $(\iota^{-1}\Pi^{\infty})^{U}$ give
rise to a homomorphism of $\mc{O}$-algebras
$\bb{T}^{T}_{\lambda}(U,\mc{O}) \rightarrow \Qlbar$. Extending $K$ if necessary, we can
and do assume that this homomorphism takes values in $\mc{O}$. Let
$\mf{m}$ denote the unique maximal ideal of
$\bb{T}^{T}_{\lambda}(U,\mc{O})$ containing the kernel of this homomorphism.

Let $\delta_{L/L^+}$ be the quadratic character of $G_{L^+}$
corresponding to $L$. By Lemma 2.1.4 of \cite{cht} we can and do
extend $\rho$ and $\rho'$ to homomorphisms $r,r' : G_{L^+} \rightarrow
\mc{G}_n(\mc{O})$ with $r \otimes k =r' \otimes k : G_{L^+}
\rightarrow \mc{G}_n(k)$, $r|_{G_L}=(\rho|_{G_L},\epsilon^{1-n})$,
$r'|_{G_L}=(\rho'|_{G_L},\epsilon^{1-n})$ and $\nu \circ r = \nu \circ
r' = \epsilon^{1-n}\delta_{L/L^+}^{\mu}$ for some $\mu \in
(\bb{Z}/2\bb{Z})$. Let $\rbar = r \otimes k : G_{L^+} \rightarrow
\mc{G}_n(k)$.

For $v\in R\cup S_a$, let $\overline{R}^{\square}_{\rbarwtvL}$ denote
the maximal $l$-torsion free quotient of
$R^{\square}_{\rbarwtvL}$. Note that $R^{\square}_{\rbarwtvL}$ is
formally smooth over $\mc{O}$ for $v \in S_a$ by Lemma 2.4.9 of
\cite{cht} so $\overline{R}^{\square}_{\rbarwtvL} =
R^{\square}_{\rbarwtvL}$. Consider the deformation problem  
\[ \mc{S} :=
\left(L/L^+,T,\wt{T},\mc{O},\rbar,\epsilon^{1-n}\delta^{\mu}_{L/L^+},\{
R^{\mathbf{v}_{\lambda_{\wt{v}}},cr}_{\rbarwtvL}\}_{v \in S_l} \cup \{
\overline{R}^{\square}_{\rbarwtvL}\}_{v \in R\cup S_a} \right).\]
By Lemma \ref{lem: property (*)}, the rings
$R^{\mathbf{v}_{\lambda_{\wt{v}}},cr}_{\rbarwtvL}$ for $v \in S_l$ and $
\overline{R}^{\square}_{\rbarwtvL}$ for $v \in R$ satisfy the
property (*) of section \ref{subsec: global deformation rings}.
Let $R^{\univ}_{\mc{S}}$ be the object representing the corresponding
deformation functor. Note that $\rbar|_{G_L}$ is $\GL_n(k)$-conjugate to $\rbar_{\mf{m}}$
where $\rbar_{\mf{m}}$ is the representation associated to the maximal
ideal $\mf{m}$ of $\bb{T}^T_{\lambda}(U,\mc{O})$ in section \ref{subsec: l-adic aut forms
  on unitary groups}. After conjugating we
can and do assume that $\rbar|_{G_L} = \rbar_{\mf{m}}$.
Since $\mf{m}$ is non-Eisenstein, we have as above a
continuous lift 
\[ r_{\mf{m}} : G_{L^+} \rightarrow
\mc{G}_n(\bb{T}^T_{\lambda}(U,\mc{O})_{\mf{m}}) \]
of $\rbar$. Properties (0)-(3) of $r_{\mf{m}}$ and the fact that $\TT_{\mf{m}}$ is $l$-torsion free and reduced imply that $r_{\mf{m}}$ 
 is of type $\mc{S}$.
Hence $r_{\mf{m}}$ gives rise
to an $\mc{O}$-algebra homomorphism
\[  R^{\univ}_{\mc{S}} \onto \bb{T}^T_{\lambda}(U,\mc{O})_{\mf{m}} \]
(which is surjective by property (2) of $r_{\mf{m}}$). To prove the
theorem it suffices to show that the homomorphism
$R^{\univ}_{\mc{S}}\rightarrow \mc{O}$ corresponding to $r$ factors
through $\bb{T}^{T}_{\lambda}(U,\mc{O})_{\mf{m}}$.

We define
\[ R^{\loc} := \left(\widehat{\otimes}_{v \in
  S_l}R^{\mathbf{v}_{\lambda_{\wt{v}}},cr}_{\rbarwtvL}\right)\widehat{\otimes} \left(\widehat{\otimes}_{v \in
  R \cup S_a}\overline{R}^{\square}_{\rbarwtvL}\right) \]
where all completed tensor products are taken over $\mc{O}$. Note that $R^{\loc}$ is equidimensional of dimension $1 + n^2 \# T + [L^+:\bb{Q}]n(n-1)/2$ by Lemma 3.3 of \cite{BLGHT}.

\begin{sublem}
  There are integers $q,g \in \bb{Z}_{\geq 0}$ with
\[ 1+q+n^2 \# T = \dim R^{\loc} + g \]
and a module $M_{\infty}$ for both $R_{\infty}:=R^{\loc}[[x_1,\ldots,x_g]]$ and $S_{\infty}:=\mc{O}[[z_1,\ldots,z_{n^2 \#T},y_1,\ldots,y_q]]$ such that:
\begin{enumerate}
\item $M_{\infty}$ is finite and free over $S_{\infty}$.
\item $M_{\infty}/(z_i,y_j) \cong S_{\lambda}(U,\mc{O})_{\mf{m}}$.
\item The action of $S_{\infty}$ on $M_{\infty}$ can be factored through an $\mc{O}$-algebra homomorphism $S_{\infty} \rightarrow R_{\infty}$.
\item There is a surjection $R_{\infty} \onto R^{\univ}_{\mc{S}}$
  whose kernel contains all the $z_i$ and $y_j$ which is compatible
  with the actions of $R_{\infty}/(z_i,y_j)$ and $R^{\univ}_{\mc{S}}$
  on $M_{\infty}/(z_i,y_j)\cong
  S_{\lambda}(U,\mc{O})_{\mf{m}}$. Moreover, there is a lift $r^{\univ}_{\mc{S}} : G_{L^+} \rightarrow \mc{G}_n(R^{\univ}_{\mc{S}})$ of $\rbar$ representing the universal deformation
 so that for each $v \in T$, the composite
  $R^{\square}_{\rbarwtvL} \rightarrow R^{\loc} \rightarrow R_{\infty}
  \onto R^{\univ}_{\mc{S}}$ arises from the lift $r^{\univ}_{\mc{S}}|_{G_{L_{\wt{v}}}}$.
\end{enumerate}
\end{sublem}

Assuming the sublemma for now, let us finish the proof of the
theorem. 
Since $R_{\infty}$ is equidimensional of dimension $\dim
R^{\loc} + g$, it follows from (1) and (3) that the support of
$M_{\infty}$ in $R_{\infty}$ is a union of irreducible
components.  (Indeed by Lemma 2.3 of \cite{tay06} it is enough to check that the
$\m_{R_\infty}$-depth of $M_\infty$ is equal to $\dim R_\infty=\dim
S_\infty$. By (3) it is enough to check the same statement for the
$\m_{S_\infty}$-depth, and this is immediate from (1)).
The conjugacy class of $r'$ determines a homomorphism $\zeta' :
R^{\univ}_{\mc{S}} \rightarrow \mc{O}$ so that $r'$ is
$\ker(\GL_n(\mc{O})\rightarrow \GL_n(k))$-conjugate to $\zeta' \circ
r^{\univ}_{\mc{S}}$.
By the choice of $L$, for each $v \in R$, $r'|_{G_{L_{\wt{v}}}}$ lies
on a unique irreducible component $\mc{C}_{\wt{v}}$ of $\Spec
\overline{R}^{\square}_{\rbarwtvL}[1/l]$.
By Lemma \ref{lem: unramified twists lie on same components},
$\mc{C}_{\wt{v}}$ is also the unique irreducible component of $\Spec
\overline{R}^{\square}_{\rbarwtvL}$ containing $\zeta' \circ
r^{\univ}_{\mc{S}}|_{G_{L_{\wt{v}}}}$. For $v|l$, a similar argument
shows that $\zeta' \circ r^{\univ}_{\mc{S}} |_{G_{L_{\wt{v}}}}$ and
$r'|_{G_{L_{\wt{v}}}}$ lie on the same irreducible component
$\mc{C}_{\wt{v}}$ of $\Spec R^{\mathbf{v}_{\lambda_{\wt{v}}},cr}_{\rbarwtvL}[1/l]$.

The conjugacy class of $r$ determines a homomorphism $\zeta :
R^{\univ}_{\mc{S}}\rightarrow \mc{O}$ so that $r$ is
$\ker(\GL_n(\mc{O})\rightarrow \GL_n(k))$-conjugate to $\zeta \circ
r^{\univ}_{\mc{S}}$. By Lemma \ref{lem: unramified twists lie on same
  components}, the set of irreducible components of $\Spec
\overline{R}^{\square}_{\rbarwtvL}[1/l]$ containing $\zeta \circ
r^{\univ}_{\mc{S}}|_{G_{L_{\wt{v}}}}$ is equal to the set of components containing
$r|_{G_{L_{\wt{v}}}}$. By the choice of $L$ it follows that
$\mc{C}_{\wt{v}}$ contains $\zeta \circ r^{\univ}_{\mc{S}}|_{G_{L_{\wt{v}}}}$. A
similar argument, using part (d) of assumption (4) of the theorem, shows that $\mc{C}_{\wt{v}}$ contains $\zeta \circ
r^{\univ}_{\mc{S}}|_{G_{L_{\wt{v}}}}$ for $v|l$.

By part 5 of Lemma 3.3 of \cite{BLGHT} the irreducible components $\mc{C}_{\wt{v}}$ for $v \in S_l \cup R$
determine an irreducible component $\mc{C}_{r'}$ of $\Spec R_{\infty}$
(as mentioned above, $R^{\square}_{\rbarwtvL}$ is formally smooth over $\mc{O}$
for $v \in S_a$). Moreover, $\zeta'$ composed with the surjection $R_{\infty} \onto R^{\univ}_{\mc{S}}$ of part (4) of the sublemma gives rise to
a closed point of $\Spec R_{\infty}[1/l]$ which is in the support of
$M_{\infty}$ and which lies in $\mc{C}_{r'}$ but does not lie in any other irreducible component of $\Spec
R_{\infty}$. We deduce that $\mc{C}_{r'}$ is in
the support of $M_{\infty}$. Since the closed point of $\Spec R_{\infty}[1/l]$
corresponding to $\zeta$ lies in $\mc{C}_{r'}$ it is also in the
support of $M_{\infty}$ and we are done by assertion (4) of the sublemma.

\begin{proof}[Proof of sublemma]
We apply the Taylor-Wiles-Kisin patching method. Let
\[ q_0 = [L^+:\bb{Q}]n(n-1)/2 +
[L^+:\bb{Q}]n(1-(-1)^{\mu-n})/2. \]
If $(Q,\wt{Q},\{ \psi_{\wt{v}} \}_{v\in Q})$ is a triple where
\begin{itemize}
\item  $Q$ is
a finite set of places of $L^+$ disjoint from $T$ and consisting of places
which split in $L$;
\item $\wt{Q}$ consists of one place $\wt{v}$ of $L$
over each place $v \in Q$;
\item for each $v \in Q$, $\rbarwtvL \cong
  \overline{\psi}_{\wt{v}}\oplus \overline{s}_{\wt{v}}$ where $\dim
  \overline{\psi}_{\wt{v}}=1$ and $\overline{\psi}_{\wt{v}}$ is not
  isomorphic to any subquotient of $\overline{s}_{\wt{v}}$;
\end{itemize}
then for each $v \in Q$, let $R_{\rbarwtvL}^{\overline{\psi}_{\wt{v}}}$
denote the quotient of $R^{\square}_{\rbarwtvL}$ corresponding to lifts
$r: G_{L_{\wt{v}}} \rightarrow \GL_n(A)$ which are
$\ker(\GL_n(A)\rightarrow\GL_n(k))$-conjugate to a lift of the form
$\psi \oplus s$ where $\psi$  lifts
$\overline{\psi}_{\wt{v}}$ and  $s$ is an unramified lift of
$\overline{s}_{\wt{v}}$. We then introduce the deformation problem
\[ \mc{S}_{Q} = 
\left(L/L^+,T\cup Q,\wt{T}\cup \wt{Q},\mc{O},\rbar,\epsilon^{1-n}\delta^{\mu}_{L/L^+},\{
R^{\mathbf{v}_{\lambda_{\wt{v}}},cr}_{\rbarwtvL}\}_{v \in S_l} \cup \{
\overline{R}^{\square}_{\rbarwtvL}\}_{v \in R\cup S_a} \cup \{ R^{\overline{\psi}_{\wt{v}}}_{\rbarwtvL}\}_{v\in Q} \right).\]
We define deformations (resp.\ $T$-framed deformations) of $\rbar$ of type $\mc{S}_{Q}$ in
the evident manner and let $R^{\univ}_{\mc{S}_Q}$ (resp.\  $R^{\square_T}_{\mc{S}_{Q}}$) denote the
universal deformation ring  (resp.\ $T$-framed deformation ring) of type $\mc{S}_Q$.

By Proposition 2.5.9 of \cite{cht} we can and do choose an integer
$q\geq q_0$ and for each $N \in \bb{Z}_{\geq 1}$ a tuple
$(Q_N,\wt{Q}_N,\{\psi_{\wt{v}}\}_{v \in Q_N})$ as above with the
following additional properties:
\begin{itemize}
\item $\# Q_N = q$ for all $N$;
\item $\mathbf{N}v \equiv 1 \mod l^N$ for $v \in Q_N$;
\item the ring $R^{\square_T}_{\mc{S}_{Q_N}}$ can be topologically
  generated over $R^{\loc}$ by
\[ q-q_0 = q -  [L^+:\bb{Q}]n(n-1)/2 -
[L^+:\bb{Q}]n(1-(-1)^{\mu-n})/2 \]
elements.
\end{itemize}

For each $N \geq 1$, let $U_1(Q_N)=\prod_v U_1(Q_N)_v$ and
$U_0(Q_N)=\prod_v U_0(Q_N)_v$ be the compact open subgroups of
$G(\bb{A}_{L^+}^{\infty})$ with $U_i(Q_N)_v = U_v$ for $v \not \in
Q_N$, $i=0,1$ and $U_i(Q_N)_v = \iota_{\wt{v}}^{-1}U_i(\wt{v})$ for $v
\in Q_N$, $i=0,1$. Note that we have natural maps
\[ \bb{T}_{\lambda}^{T\cup Q_N}(U_{1}(Q_N),\mc{O}) \onto \bb{T}_{\lambda}^{T\cup Q_N}(U_{0}(Q_N),\mc{O}) \onto \bb{T}_{\lambda}^{T\cup Q_N}(U,\mc{O}) \into \bb{T}_{\lambda}^T(U,\mc{O}).\]
Thus $\mf{m}$ determines maximal ideals of the first three algebras in this sequence which we denote by $\mf{m}_{Q_N}$ for the first two and $\mf{m}$ for the third. Note also that $\bb{T}_{\lambda}^{T \cup Q_N}(U,\mc{O})_{\mf{m}} = \bb{T}_{\lambda}^T(U,\mc{O})_{\mf{m}}$ by the proof of Corollary 3.4.5 of \cite{cht}.

For each $v \in Q_N$ choose an element
$\phi_{\wt{v}} \in G_{L_{\wt{v}}}$ lifting geometric Frobenius and
let $\varpi_{\wt{v}} \in \mc{O}_{L_{\wt{v}}}$ be the uniformiser with
$\Art_{L_{\wt{v}}}\varpi_{\wt{v}}=\phi_{\wt{v}}|_{L_{\wt{v}}^{\ab}}$. Let
$P_{\wt{v}}(X) \in \bb{T}^{T\cup
  Q_N}_{\lambda}(U_1(Q_N),\mc{O})_{\mf{m}_{Q_N}} [X]$ denote the characteristic
polynomial of $r_{\mf{m}_{Q_N}}(\phi_{\wt{v}})$.
By Hensel's lemma, we can factor
$P_{\wt{v}}(X)=(X-A_{\wt{v}})Q_{\wt{v}}(X)$ where $A_{\wt{v}}$ lifts
$\overline{\psi}_{\wt{v}}(\phi_{\wt{v}})$ and $Q_{\wt{v}}(A_{\wt{v}})$
is a unit in $\bb{T}^{T\cup Q_N}_{\lambda}(U_1(Q_N),\mc{O})_{\mf{m}_{Q_N}}$. For
$i=0,1$ and $\alpha \in L_{\wt{v}}$ of non-negative valuation, consider the Hecke operator
\[V_{\alpha}:= \iota_{\wt{v}}^{-1}\left( \left[ U_{i}(\wt{v})  \left( \begin{matrix}
      1_{n-1} & 0 \cr 0 & \alpha \end{matrix} \right)
  U_{i}(\wt{v}) \right] \right) \]
on $S_{\lambda}(U_i(Q_N),\mc{O})_{\mf{m}_{Q_N}}$. Let $G_{Q_N} =
U_0(Q_N)/U_1(Q_N)$ and let $\Delta_{Q_N}$ denote the maximal $l$-power
order quotient of $G_{Q_N}$. Let $\mf{a}_{Q_N}$ denote the kernel of
the augmentation map $\mc{O}[\Delta_{Q_N}] \rightarrow \mc{O}$.
For $i=0,1$, let
\[ H_{i,Q_N} = \prod_{v \in Q_N}
Q_{\wt{v}}(V_{\varpi_{\wt{v}}})S_{\lambda}(U_{i}(Q_N),\mc{O})_{\mf{m}_{Q_N}} .\]
and let $\bb{T}_{i,Q_N}$ denote the image of $\bb{T}^{T\cup
  Q_N}_{\lambda}(U_i(Q_N),\mc{O})$ in
$\End_{\mc{O}}(H_{i,Q_N})$. Let $H = S_{\lambda}(U,\mc{O})_{\mf{m}}$.
 We
claim that the following hold:
\begin{enumerate}
\item For each $N$, the map
\[ \prod_{v\in Q_N} Q_{\wt{v}}(V_{\varpi_{\wt{v}}}): H \rightarrow
H_{0,Q_N} \]
is an isomorphism.
\item For each $N$, $H_{1,Q_N}$ is free over $\mc{O}[\Delta_{Q_N}]$ with
\[ H_{1,Q_N}/{\mf{a}_{Q_N}} \isoto H_{0,Q_N}. \]
\item For each $N$ and each $v \in Q_N$, there is a character with
  open kernel $V_{\wt{v}} : L_{\wt{v}}^{\times} \rightarrow
  \bb{T}_{1,Q_N}^{\times}$ so that
  \begin{enumerate}
  \item for each $\alpha \in L_{\wt{v}}$ of non-negative valuation,
    $V_{\alpha}= V_{\wt{v}}(\alpha)$ on $H_{1,Q_N}$;
  \item $(r_{\mf{m}_{Q_N}}\otimes \bb{T}_{1,Q_N})|_{W_{L_{\wt{v}}}}
    \cong s \oplus (V_{\wt{v}}\circ \Art_{L_{\wt{v}}}^{-1})$ with $s$ unramified, lifting $\overline{s}_{\wt{v}}$ and $(V_{\wt{v}}\circ \Art_{L_{\wt{v}}}^{-1})$ lifting $\overline{\psi}_{\wt{v}}$.
  \end{enumerate}
\end{enumerate}
To see this, note that Lemmas 3.1.3 and 3.1.5 of \cite{cht} imply that
$P_{\wt{v}}(V_{\varpi_{\wt{v}}})=0$ on
$S_{\lambda}(U_{1}(Q_N),\mc{O})_{\mf{m}_{Q_N}}$.  Property (1) now
follows from Lemma 3.2.2 of \cite{cht} together with Lemma 3.1.5 of
\cite{cht} and the fact that $\bb{T}^{T\cup
  Q_N}_{\lambda}(U,\mc{O})_{\mf{m}} =
\bb{T}^{T}_{\lambda}(U,\mc{O})_{\mf{m}}$. Property (3) follows exactly
as in the proof of part 8 of Proposition 3.4.4 of \cite{cht}. Note
that $H_{1,Q_N}$ is a
$\bb{T}_{\lambda}^{T}(U_1(Q_N),\mc{O})_{\mf{m}_{Q_N}}[G_{Q_N}]$-direct
summand of $S_{\lambda}(U_1(Q_N),\mc{O})_{\mf{m}_{Q_N}}$. Moreover, it follows from the fact that $U$ is sufficiently
small (see Lemma 3.3.1 of \cite{cht}) that
$S_{\lambda}(U_1(Q_N),\mc{O})_{\mf{m}_{Q_N}}$ is finite free over
$\mc{O}[G_{Q_N}]$ with $G_{Q_N}$-coinvariants isomorphic to
$S_{\lambda}(U_0(Q_N),\mc{O})_{\mf{m}_{Q_N}}$ via the trace map
$\tr_{G_{Q_N}}$. It follows that $H_{1,Q_N}$
has $G_{Q_N}$-coinvariants isomorphic to $H_{0,Q_N}$ via
$\tr_{G_{Q_N}}$. Finally, note that by (3) the action of $\alpha =
(\alpha_{\wt{v}})_{v \in Q_N} \in G_{Q_N}$ on $H_{1,Q_N}$ is given by
$\prod_{v \in Q_N} V_{\wt{v}}(\alpha_{\wt{v}})$. Since each
$\overline{\psi}_{\wt{v}}$ is unramified, the action of $G_{Q_N}$ on
$H_{1,Q_N}$ must factor through $\Delta_{Q_N}$ and (2) follows.

For each $N$, the lift $r_{\mf{m}_{Q_N}} \otimes \bb{T}_{1,Q_N}$ of $\rbar$ is of
type $\mc{S}_{Q_N}$ and gives rise to a surjection
$R^{\univ}_{\mc{S}_{Q_N}} \onto \bb{T}_{1,Q_N}$. 
Thinking of $\Delta_{Q_N}$ as the maximal $l$-power quotient of
$\prod_{v \in Q_N} I_{L_{\wt{v}}}$, the determinant of any choice of
universal deformation $r^{\univ}_{\mc{S}_{Q_N}}$ gives rise to a
homomorphism $\Delta_{Q_N} \rightarrow (R^{\univ}_{\mc{S}_{Q_N}})^{\times}$. We
thus have homomorphisms $\mc{O}[\Delta_{Q_N}] \rightarrow
R^{\univ}_{\mc{S}_{Q_N}} \rightarrow R^{\square_T}_{\mc{S}_{Q_N}}$ and
natural isomorphisms $R^{\univ}_{\mc{S}_{Q_N}}/\mf{a}_{Q_N} \cong
R^{\univ}_{\mc{S}}$ and $R^{\square_T}_{\mc{S}_{Q_N}}/\mf{a}_{Q_N}
\cong R^{\square_T}_{\mc{S}}$.

Let
\[ \mc{T} = \mc{O}[[X_{v,i,j}:v \in T, i,j = 1,\ldots,n]].\]
Choose a lift $r^{\univ}_{\mc{S}} : G_{L^+} \rightarrow
\mc{G}_n(R^{\univ}_{\mc{S}})$ representing the universal
deformation. The tuple $(r^{\univ}_{\mc{S}},(1_n+X_{v,i,j})_{v \in T})$ gives rise to an isomorphism
\[ R^{\square_T}_{\mc{S}} \isoto
R^{\univ}_{\mc{S}}\widehat{\otimes}_{\mc{O}} \mc{T}.\]
For each $N$, choose a lift $r^{\univ}_{\mc{S}_{Q_N}} : G_{L^+}
\rightarrow \mc{G}_n(R^{\univ}_{\mc{S}_{Q_N}})$ representing the
universal deformation with $r^{\univ}_{S_{Q_N}} \mod
\mf{a}_{Q_N}=r^{\univ}_{\mc{S}}$. This gives rise to an isomorphism
$R^{\square_T}_{\mc{S}_{Q_N}} \isoto
R^{\univ}_{\mc{S}_{Q_N}}\widehat{\otimes}_{\mc{O}}\mc{T}$ which
reduces modulo $\mf{a}_{Q_N}$ to the isomorphism
$R^{\square_T}_{\mc{S}}\isoto
R^{\univ}_{\mc{S}}\widehat{\otimes}_{\mc{O}} \mc{T}$. We let
\begin{eqnarray*}
  H^{\square_T} & = & H \otimes_{R^{\univ}_{\mc{S}}}R^{\square_T}_{\mc{S}}
  \\
  H^{\square_T}_{1,Q_N} & = & H_{1,Q_N}
  \otimes_{R^{\univ}_{\mc{S}_{Q_N}}} R^{\square_T}_{\mc{S}_{Q_N}} \\
 \bb{T}^{\square_T}_{1,Q_N} & = & \bb{T}_{1,Q_N} \otimes_{R^{\univ}_{\mc{S}_{Q_N}}} R^{\square_T}_{\mc{S}_{Q_N}}.
\end{eqnarray*}
Then $H^{\square_T}_{1,Q_N}$ is a finite free
$\mc{T}[\Delta_{Q_N}]$-module with $H^{\square_T}_{1,Q_N}/\mf{a}_{Q_N}
\cong H^{\square_T}$, compatible with the isomorphism
$R^{\square_T}_{\mc{S}_{Q_N}}/\mf{a}_{Q_N} \cong R^{\square_T}_{\mc{S}}$.

Let $g=q-q_0$ and let
\begin{eqnarray*}
  \Delta_{\infty} & = & \bb{Z}_l^q \\
  R_{\infty} & = & R^{\loc}[[x_1,\ldots,x_g]] \\
  S_{\infty} & = & \mc{T}[[\Delta_{\infty}]] 
\end{eqnarray*}
and let $\mf{a}$ denote the kernel of the $\mc{O}$-algebra homomorphism $S_{\infty}
\rightarrow \mc{O}$ which sends each $X_{v,i,j}$ to 0 and each element
of $\Delta_{\infty}$ to 1. Note that $S_\infty$ is a formally smooth
over $\bigO$ of relative dimension $q+n^2\# T$. For each $N$, choose a surjection
$\Delta_{\infty} \onto \Delta_{Q_N}$ and let $\mf{c}_N$ denote the
kernel of the corresponding homomorphism $S_{\infty} \onto \mc{T}[\Delta_{Q_N}]$.
For each $N \geq 1$, choose a surjection of $R^{\loc}$-algebras
\[ R_{\infty} \onto R^{\square_T}_{\mc{S}_{Q_N}}.\]
We regard each $R^{\square_T}_{\mc{S}_{Q_N}}$ as an
$S_{\infty}$-algebra via $S_{\infty}\onto \mc{T}[\Delta_{Q_N}]
\rightarrow R^{\square_T}_{\mc{S}_{Q_N}}$. In particular,
$R^{\square_T}_{\mc{S}_{Q_N}}/\mf{a} \cong R^{\univ}_{\mc{S}}$.

Choose a sequence of open ideals $(\mf{b}_N)_{N\geq 1}$ of $S_{\infty}$ with
\begin{itemize}
  \item $\mf{b}_N \supset \mf{c}_N$
  \item $\mf{b}_{N} \supset \mf{b}_{N+1}$
  \item $\cap_N \mf{b}_N = (0)$.
\end{itemize}
Let $\bb{T}=\bb{T}_{\lambda}(U,\mc{O})_{\mf{m}}$. Choose a sequence of open ideals $(\mf{d}_N)_{N \geq 1}$ of
$R^{\univ}_{\mc{S}}$ with the following properties:
\begin{itemize}
  \item $ \mf{b}_NR^{\univ}_{\mc{S}} + \ker(R^{\univ}_{\mc{S}} \onto
    \bb{T}) \supset \mf{d}_N \supset
    \mf{b}_N R^{\univ}_{\mc{S}}$
  \item $\mf{d}_{N} \supset \mf{d}_{N+1}$
  \item $\cap_N \mf{d}_N = (0)$.
\end{itemize}
In the first bullet point 
$S_{\infty}$ acts on $R^{\univ}_{\mc{S}}$ via the quotient
$S_{\infty}/\mf{a}=\mc{O}$. In what follows, we also consider $S_{\infty}$ acting on
$\bb{T}$ and $H$ via the quotient $S_{\infty}/\mf{a}$.

For each $N \geq 1$, define a `patching datum of level $N$' to consist
of a tuple $(\phi,\mc{M},\psi)$ where
\begin{itemize}
  \item $\phi$ is a surjective homomorphism of $\mc{O}$-algebras
\[ \phi : R_{\infty} \onto R^{\univ}_{\mc{S}}/\mf{d}_N.\]
  \item $\mc{M}$ is a module over $R_{\infty} \widehat{\otimes}_{\mc{O}}
    S_{\infty}$ which is finite free over $S_{\infty}/\mf{b}_N$.
  \item $\psi$ is an isomorphism of $\mc{O}$-modules
\[ \psi : \mc{M}/\mf{a} \isoto H/\mf{b}_N \]
compatible with the action of $R_{\infty}$ on $\mc{M}/\mf{a}$, the action
of $R^{\univ}_{\mc{S}}/\mf{d}_N$ on $H/\mf{b}_N$ (via
$R^{\univ}_{\mc{S}}/\mf{d}_N \onto \bb{T}/\mf{b}_N$)  and the homomorphism $\phi$.
\end{itemize}
We consider two such patching data of level $N$ $(\phi,\mc{M},\psi)$ and
$(\phi',\mc{M}',\psi')$ to be equivalent if $\phi = \phi'$ and there is an
isomorphism $\mc{M} \cong \mc{M}'$ of $R_{\infty}\widehat{\otimes}_{\mc{O}}
S_{\infty}$-modules which is compatible with $\psi$ and $\psi'$ when reduced
modulo $\mf{a}$. Note that there are only finitely many patching data
of level $N$ up to equivalence. Note also that given $N' \geq N \geq
1$ and a patching datum
of level $N'$, we can obtain a patching datum of level $N$ in an
obvious fashion.

For each $M\geq N \geq 1$, let $D(M,N)$ be the patching datum of level
$N$ consisting of
\begin{itemize}
  \item the surjective homomorphism
\[ R_{\infty} \onto R^{\square_T}_{\mc{S}_{Q_M}} \onto
R^{\square_T}_{\mc{S}_{Q_M}}/\mf{a}=R^{\univ}_{\mc{S}} \onto
R^{\univ}_{\mc{S}}/\mf{d}_N \]
  \item the module $H^{\square_T}_{1,Q_M}/\mf{b}_N$ which is finite
    free over $S_{\infty}/\mf{b}_N$ and acted upon by $R_{\infty}$ via
    $R_{\infty} \onto R^{\square_T}_{\mc{S}_{Q_M}} \onto \bb{T}_{1,Q_M}^{\square}$. 
\item the isomorphism
\[ H^{\square_T}_{1,Q_M} /(\mf{a} + \mf{b}_N) \cong H/\mf{b}_N \]
which is compatible with the homomorphism $R_{\infty} \onto
R^{\square_T}_{\mc{S}_{Q_M}} \onto R^{\univ}_{\mc{S}}/\mf{d}_N$.
\end{itemize}
Since there are only finitely many patching data of each level $N$, we
can and do choose a sequence of pairs of integers $(M_i,N_i)_{i \geq
  1}$ with $M_{i+1} > M_i$, $N_{i+1}> N_{i}$ and $M_i \geq N_i$ for
all $i$ such that $D(M_{i+1},N_{i+1})$ reduced to level $N_i$ is
equivalent to $D(M_i,N_i)$. In other words
\begin{itemize}
\item for each $i$, the homomorphism
\[ \phi_{i+1}: R_{\infty} \onto R^{\square_T}_{\mc{S}_{Q_{M_{i+1}}}} \onto
R^{\square_T}_{\mc{S}_{Q_{M_{i+1}}}}/\mf{a}=R^{\univ}_{\mc{S}} \onto
R^{\univ}_{\mc{S}}/\mf{d}_{N_{i+1}}, \]
when reduced modulo $\mf{d}_{N_{i}}$ is equal to the homomorphism 
\[ \phi_i: R_{\infty} \onto R^{\square_T}_{\mc{S}_{Q_{M_i}}} \onto
R^{\square_T}_{\mc{S}_{Q_{M_i}}}/\mf{a}=R^{\univ}_{\mc{S}} \onto
R^{\univ}_{\mc{S}}/\mf{d}_{N_i}. \]
\item for each $i$, we can and do choose an isomorphism of $R_{\infty}
  \widehat{\otimes}_{\mc{O}} S_{\infty}$-modules
\[ \gamma_i : H^{\square_T}_{1,Q_{M_{i+1}}}/\mf{b}_{N_{i}} \cong
H^{\square_T}_{1,Q_{M_i}}/\mf{b}_{N_i}. \]
\end{itemize}
Taking the inverse limit of the $\phi_i$ gives rise to a surjection
\[  R_{\infty} \onto R^{\univ}_{\mc{S}}. \]
Define
\[ H^{\square_T}_{\infty} := \varprojlim_{i}
H^{\square_T}_{1,Q_{M_{i}}}/\mf{b}_{N_i}\] where the limit is taken
with respect to the $\gamma_i$. Then $H^{\square_T}_{\infty}$ is a
module for $R_{\infty}\widehat{\otimes}_{\mc{O}}S_{\infty}$ which is
finite free over $S_{\infty}$. Note that the image of $S_{\infty}$ in
$\End_{S_{\infty}}(H^{\square_T}_{\infty})$ is contained in the image
of $R_{\infty}$ (indeed, the image of $R_{\infty}$ is closed and the
corresponding statement is true for each
$H^{\square_T}_{1,Q_{M_i}}/b_{N_i}$).  Since $S_{\infty}$ is formally
smooth over $\mc{O}$, we can and do factor the action of $S_{\infty}$
on $H^{\square_T}_{\infty}$ through $R_{\infty}$ (note that it
suffices to define the factorisation on a set of topological
generators of $S_\infty$ over $\bigO$).
Note that we have
\[ H^{\square_T}_{\infty}/\mf{a} \cong H \]
compatible with the surjection $R_{\infty} \onto
R^{\univ}_{\mc{S}}$. Since $R_{\infty}$ is equidimensional of
dimension 
\[ 1 + n^2 \#T + [L^+:\bb{Q}]n(n-1)/2 + q-q_0 = 1 + q + n^2 \#T  -
[L^+:\bb{Q}]n(1-(-1)^{\mu-n})/2 \]
and $H^{\square_T}_{\infty}$ has $\m_{R_{\infty}}$-depth at least
\[  1 + q + n^2 \#T \]
(the dimension of $S_{\infty}$) we deduce from Lemma 2.3 of \cite{tay06} that $\mu \equiv n \mod
2$. Hence
\[ 1 + q + n^2 \#T = \dim R^{\loc} + g\]
and taking $M_{\infty} = H^{\square_T}_{\infty}$, the sublemma is proved.
\end{proof}
Since proving the sublemma was our only remaining task, the automorphy lifting 
theorem is proven.
\end{proof}

\subsubsection{Totally real fields}
\label{subsubsec: totally real fields}

\begin{thm}\label{thm: the main modularity lifting theorem, using tilde}
  Let $F^+$ be a totally real field. Let $l>n$ be a prime and let $K\subset \Qlbar$ denote a finite extension of $\bb{Q}_l$ with ring of integers
$\mc{O}$ and residue field $k$. Assume that $K$ contains the image of
every embedding $F^+ \into \Qlbar$.
  Let \[\rho:G_{F^+}\to\GL_n(\mc{O})\]be a continuous 
  representation and let $\rhobar = \rho \mod
  \mf{m}_{\mc{O}}$. Suppose that $\rho$ enjoys the following properties:
  \begin{enumerate}
  \item $\rho^\vee\cong \rho\epsilon^{n-1}\chi$  for some character
    $\chi:G_{F^+}\to\mc{O}^\times$ with $\chi(c_v)$ independent of
    $v|\infty$ (where $c_v$ denotes a complex conjugation at $v$).
   \item The reduction $\rhobar$ is absolutely irreducible and 
 $\rhobar(G_{F^+(\zeta_l)}) \subset \GL_n(k)$ is big (see Definition
 \ref{defn: m-big}). 
  \item $(\overline{F^+})^{\ker\ad\rhobar}$ does not contain $\zeta_l$.
  \item There is a continuous representation
    $\rho':G_{F^+}\to\GL_n(\mc{O})$, a RAESDC automorphic representation
    $\pi$ of $\GL_n(\bb{A}_{F^+})$ which is unramified above $l$ and $\iota: \Qlbar \isoto \bb{C}$
 such that
    \begin{enumerate}
    \item $\rho' \otimes_{\mc{O}}\Qlbar \cong r_{l,\iota}(\pi) : G_{F^+}
      \rightarrow \GL_n(\Qlbar)$ and $(\rho')^\vee\cong
  \rho'\epsilon^{n-1}\chi'$ for some character
  $\chi':G_{F^+}\to\mc{O}^\times$ with $\overline{\chi}'= \overline{\chi}$.
    \item  $\rhobar=\rhobar'$.
    \item For all places $v\nmid l$ of $F^+$, either
     $\rho|_{G_{F^+_v}}$ and $\pi_v$ are both
        unramified, or the following both hold:
      \begin{itemize}
      \item $\rho'|_{G_{F^+_v}}\leadsto_{\mc{O}} \rho|_{G_{F^+_v}}$
        and $\rho|_{G_{F^+_v}}\leadsto_{\mc{O}} \rho'|_{G_{F^+_v}}$.
      \item $\chi|_{I_{F^+_v}}=\chi'|_{I_{F^+_v}}$.
      \end{itemize}
    \item For all places $v|l$, $\rho|_{G_{F_v}} \sim
      \rho'|_{G_{F_v}}$.
    \end{enumerate}
  \end{enumerate}
Then $\rho$ is automorphic.
\end{thm}

\begin{proof}
  Extending $\mc{O}$ if necessary, choose a quadratic CM extension $F$
  of $F^+$ and algebraic characters $\psi,\psi' : G_{F} \rightarrow
  \mc{O}^{\times}$ such that the following hold.
  \begin{enumerate}[(i)]
  \item $F$ is linearly disjoint from $\overline{F^+}^{\ker 
      \rbar}(\zeta_l)$ over $F^{+}$.
  \item Each place of $F^+$ lying over $l$ and each place at which $\rho$ or $\pi$ is
    ramified splits completely in $F$.
  \item $\psi \psi^{c} = \chi|_{G_{F}}$.
  \item $\psi$ and $\psi'$ are crystalline above $l$.
  \item $\psi' (\psi')^{c}=\chi'|_{G_F}$.
  \item Let $S$ denote the set of places of $F$ which divide $l$ or
    which lie over a place of $F^+$ where $\rho$ or $\pi$
    ramifies. Then for all $w \in S$, we have $\psi'|_{I_{L_w}} =
    \psi|_{I_{L_w}}$.
  \item $\overline{\psi}=\overline{\psi}'$.
  \end{enumerate}
  (Take $F$ to be a quadratic CM extension of $F^+$ satisfying (i) and
  (ii).  Use Lemma 4.1.5 of \cite{cht} to construct a $\psi$ which
  satisfies (iii) and (iv). Note that $\chi$ and $\chi'$ are
  crystalline characters of $G_{F^+}$ with the same Hodge-Tate
  weights. In particular, for each place $v|l$ of $F^+$ we have
  $\chi|_{I_{F^+_v}} = \chi' |_{I_{F^+_v}}$. We can therefore apply
  Lemma 4.1.6 of \cite{cht} to find a $\psi'$ satisfying (v), (vi) and
  (vii).)
  
  Let $r = \psi \rho|_{G_F}$ and $r' = \psi' \rho'|_{G_F}$. Then
  $\rbar = \rbar'$, $r^c \cong r^{\vee}\epsilon^{1-n}$ and $(r')^c
  \cong (r')^{\vee}\epsilon^{1-n}$. For $w \in S$, Lemma \ref{lem:
    nice twists preserve sim and leadsto} implies that $r|_{G_{F_w}}
  \sim r'|_{G_{F_w}}$ if $v|l$ and both $r|_{G_{F_w}}
  \leadsto_{\mc{O}} r'|_{G_{F_w}}$ and $r'|_{G_{F_w}}
  \leadsto_{\mc{O}} r|_{G_{F_w}}$ otherwise.
    
    The theorem now follows from Theorem \ref{thm: CM modularity lifting
    theorem, using tilde} applied to $r|_{G_F}$ and $r'|_{G_F}$,
  together with Lemma 1.5 of \cite{BLGHT}.
\end{proof}

\subsubsection{Finiteness of a deformation ring}
\label{subsubsec: finiteness of def ring}

We now deduce a result on the finiteness of a universal deformation
ring. This result is not needed for the main theorem, so this section
may be skipped by readers interested only in the Sato-Tate
conjecture. However, we believe that it is of independent interest,
and it will prove useful to us in future work.

Let $F$ be an imaginary CM field with totally real subfield $F^+$ and
let $c$ be the non-trivial element of $\Gal(F/F^+)$. Let $n\in
\bb{Z}_{\geq 1}$ and let $l>n$ be a prime. Let $K\subset \Qlbar$
denote a finite extension of $\bb{Q}_l$ with ring of integers $\mc{O}$
and residue field $k$. Assume that $K$ contains the image of every
embedding $F \into \Qlbar$.  Suppose in addition that each place of
$F^+$ dividing $l$ splits in $F$.  Let
\[ \rhobar : G_F \rightarrow \GL_n(k) \]
be a continuous homomorphism and suppose
\[ \rho' : G_F \rightarrow \GL_n(\mc{O}) \]
is a continuous lift of $\rhobar$ which is automorphic of
level prime to $l$. In particular, $(\rho')^c \cong (\rho')^{\vee} \epsilon^{1-n}$.
By definition, there is an RACSDC automorphic
representation $\pi$ of $\GL_n(\bb{A}_{F})$ (which is unramified above $l$)
and an isomorphism $\iota : \Qlbar \isoto
\bb{C}$ such that $\rho^{\prime}\otimes\Qlbar \cong
r_{l,\iota}(\pi) : G_{F} \rightarrow \GL_n(\Qlbar)$. Suppose that
every finite place of $F$ at which $\pi$ is ramified is split over
$F^+$. Let $S_l$ denote the set of places of $F^+$ lying above
$l$. Let $R$ denote a finite set of finite places of $F^+$ disjoint
from $S_l$ and containing the restriction to $F^+$ of every finite
place of $F$ where $\pi$ is ramified.

Let $\delta_{F/F^+}$ be the quadratic character of $G_{F^+}$
corresponding to $F$. By Lemma 2.1.4 of \cite{cht} we can and do
extend $\rhobar$ and $\rho'$ to homomorphisms $\rbar : G_{F^+}
\rightarrow \mc{G}_n(k)$ and $r' : G_{F^+} \rightarrow
\mc{G}_n(\mc{O})$ with $r' \otimes k =\rbar $,
$\rbar|_{G_F}=(\rhobar,\epsilon^{1-n})$,
$r'|_{G_F}=(\rho'|_{G_F},\epsilon^{1-n})$ and $ \nu \circ r' =
\epsilon^{1-n}\delta_{F/F^+}^{\mu}$ for some $\mu \in
(\bb{Z}/2\bb{Z})$ (which is independent of the choice of $r'$).

For each $v \in R\cup S_l$ choose once and for all a place $\wt{v}$ of
$F$ lying above $v$. Let $\wt{R}$ and $\wt{S}_l$ denote the set of
$\wt{v}$ for $v$ in $R$ and $S$ respectively.  For each $v \in R$, let
$R^{\square}_{\rbarwtv}$ denote the universal $\mc{O}$-lifting ring of
$\rbarwtv$.
Suppose that for each $v \in R$ and each minimal prime ideal $\wp$ of $R^{\square}_{\rbarwtv}$, the quotient $R^{\square}_{\rbarwtv}/\wp$ is geometrically integral.
Note that this can always be achieved by replacing $\mc{O}$ with the ring of integers in a finite extension of $K$.
Suppose that the lift $r'|_{G_{F_{\wt{v}}}}$
corresponds to a closed point of $\Spec R^{\square}_{\rbarwtv}[1/l]$
which lies on only one irreducible component. 
Let $R_{\wt{v}}$ denote
the quotient of $R^{\square}_{\rbarwtv}$ by the minimal prime ideal
corresponding to this irreducible component.

For $v \in S_l$, let
$\lambda_{\wt{v}}$ be the element of
$(\bb{Z}^{n}_+)^{\Hom(F_{\wt{v}},K)}$ with the property that
$r'|_{G_{F_{\wt{v}}}}$ has $l$-adic Hodge type
$\mathbf{v}_{\lambda_{\wt{v}}}$. Suppose that for each minimal prime ideal $\wp$ of $R^{\mathbf{v}_{\lambda_{\wt{v}}},cr}_{\rbarwtv}$, the quotient $R^{\mathbf{v}_{\lambda_{\wt{v}}},cr}_{\rbarwtv}/\wp$ is geometrically integral.
 Let $R_{\wt{v}}$ be the quotient of
$R^{\mathbf{v}_{\lambda_{\wt{v}}},cr}_{\rbarwtv}$ by the minimal prime
ideal corresponding to the (necessarily unique) irreducible component
of $\Spec R^{\mathbf{v}_{\lambda_{\wt{v}}},cr}_{\rbarwtv}[1/l]$
containing $r'|_{G_{F_{\wt{v}}}}$.

Consider the deformation problem
\begin{align*} 
  \mc{S}_{r'} & := \left(F/F^+,R\cup S_l,\wt{R}\cup
    \wt{S}_l,\mc{O},\rbar,\epsilon^{1-n}\delta^{\mu}_{F/F^+},\{
    R_{\wt{v}}\}_{v \in R\cup S_l} \right) .
\end{align*}
By Lemma \ref{lem: property (*)}, the rings $R_{\wt{v}}$ for $v \in
S_l \cup R$ satisfy the property (*) of section \ref{subsec:
  global deformation rings}. Let $R^{\univ}_{\mc{S}_{r'}}$ be the
object of $\mc{C}_{\mc{O}}$ representing the corresponding deformation
functor.

\begin{prop}
\label{prop: finiteness of global deformation ring} Maintain the
assumptions made above. Suppose in addition that
\begin{itemize}
\item[(i)] $\rhobar(G_{F(\zeta_l)})\subset \GL_n(k)$ is big, and
\item[(ii)] $\overline{F}^{\ker \ad \rhobar}$ does not contain
  $\zeta_l$.
\end{itemize}
Then
\begin{enumerate}
\item $R^{\univ}_{\mc{S}_{r'}}$ is a finite $\mc{O}$-algebra.
\item  Any
  $\Qlbar$-point of $R^{\univ}_{\mc{S}_{r'}}$ gives rise to a
  representation $G_F \rightarrow \GL_n(\Qlbar)$ which is automorphic
  of level prime to $l$.
\item $\mu \equiv n \mod 2$.
\end{enumerate}
\end{prop}

\begin{proof}
Choose a place $v_1$ of $F$ not lying over $l$ and such that 
 \begin{itemize}
  \item $v_1$ is unramified over a rational prime $p$ with
    $[F(\zeta_p): F]>n$;
  \item $v_1$ does not split completely in $F(\zeta_l)$;
  \item $\rho$ and $\pi$ are unramified at $v_1$;
  \item $\ad \rhobar (\Frob_{v_1})= 1$.
  \end{itemize}
Choose an imaginary CM field $L/F$  such that:
  \begin{itemize}
  \item $L/F$ is solvable;
  \item $L$ is linearly disjoint from $\overline{F}^{\ker 
      \rbar}(\zeta_l)$ over $F$;
  \item $4 | [L^+:F^+]$ where $L^+$ denotes the maximal totally real
    subfield of $L$;
  \item $L/L^+$ is unramified at all finite places;
  \item Every place of $L$ over $v_1$ or  $cv_1$ is split over
    $L^+$. Moreover, $v_1$ and $cv_1$ split completely in $L$;  
  \item The places $v \in R \cup S_l$ split completely in $L^+$.
  \end{itemize}
Let $S_{L,l}$ denote the set of
  places of $L^+$ dividing $l$ and let $R_L$ denote the set of places of
  $L^+$ lying over $R$. Similarly, let $\wt{S}_{L,l}$ and $\wt{R}_L$
  denote the sets of places of $L$ lying over $\wt{S}_l$ and $\wt{R}$
  respectively.  Let $S_{L,a}$ denote
  the set of places of $L^+$ lying over the restriction of $v_1$ to
  $F^+$. Let $\wt{S}_{L,a}$ denote the set of places of $L$ lying over
  $v_1$. Let $T = S_{L,l} \coprod R_{L} \coprod S_{L,a}$ and let
  $\wt{T} = \wt{S}_{L,l} \coprod \wt{R}_{L} \coprod \wt{S}_{L,a}$ so
  that $\wt{T}$ consists of one place $\wt{v}$ for each place $v \in T$.

  For $v \in R_{L}\cup S_{L,l}$, let $R_{\wt{v}} = R_{\wt{v}|_F}$ regarded as a
  quotient of $R^{\square}_{\rbarwtvL}$ (note that $F_{\wt{v}|_F}
  \isoto L_{\wt{v}}$ by the choice of $L$).
For $v \in R_L \cup S_{L,a}$, let
  $\overline{R}^{\square}_{\rbarwtvL}$ be the maximal $l$-torsion free
  quotient of $R^{\square}_{\rbarwtvL}$. For $v \in S_{L,a}$,
  $R^{\square}_{\rbarwtvL}$ is formally smooth over $\mc{O}$ and we
  let $R_{\wt{v}} = R^{\square}_{\rbarwtvL} = \overline{R}^{\square}_{\rbarwtvL}$.
 For each $v \in S_{L,l}$, let
$\lambda_{\wt{v}} = \lambda_{\wt{v}|_F} \in
(\bb{Z}^{n}_+)^{\Hom(L_{\wt{v}},K)}$.
Consider the deformation problems
\begin{align*} \mc{S}_{L} & :=
\left(L/L^+,T,\wt{T},\mc{O},\rbar|_{G_{L^+}},\epsilon^{1-n}\delta^{\mu}_{L/L^+},\{
R^{\mathbf{v}_{\lambda_{\wt{v}}},cr}_{\rbarwtvL}\}_{v \in S_{L,l}} \cup \{
\overline{R}^{\square}_{\rbarwtvL}\}_{v \in R_L\cup S_{L,a}} \right) \\
\mc{S}_{L,r'} & := \left(L/L^+,T,\wt{T},\mc{O},\rbar|_{G_{L^+}},\epsilon^{1-n}\delta^{\mu}_{L/L^+},\{
R_{\wt{v}}\}_{v \in T}  \right) .
\end{align*}
By Lemma \ref{lem: property (*)}, the rings
$R^{\mathbf{v}_{\lambda_{\wt{v}}},cr}_{\rbarwtvL}$, $R_{\wt{v}}$ for
$v \in S_{L,l}$, $ \overline{R}^{\square}_{\rbarwtvL}$, $R_{\wt{v}}$
for $v \in R_L$ and $ \overline{R}^{\square}_{\rbarwtvL}$ for $v\in
S_{L,a}$ satisfy the property (*) of section \ref{subsec: global
  deformation rings}.  Let $R^{\univ}_{\mc{S}_L}$ and
$R^{\univ}_{\mc{S}_{L,r'}}$ be the objects of $\mc{C}_{\mc{O}}$
representing the corresponding deformation functors. Note that there
is a natural surjection $R^{\univ}_{\mc{S}_L} \onto
R^{\univ}_{\mc{S}_{L,r'}}$. Note also that there is a natural
homomorphism of $\mc{O}$-algebras $R^{\univ}_{\mc{S}_{L,r'}}
\rightarrow R^{\univ}_{\mc{S}_{r'}}$ which is obtained by restricting
the universal deformation of type $\mc{S}_{r'}$ to $G_{L^+}$. This map
is finite by an argument of
Khare and Wintenberger (cf. Lemma 3.2.5 of \cite{GG}) and hence it suffices to show that
$R^{\univ}_{\mc{S}_{L,r'}}$ is finite over $\mc{O}$.

  Let $G_{/\mc{O}_{L^+}}$
  be an algebraic group as in section \ref{subsec: l-adic aut forms on
    unitary groups} (with $F/F^+$ replaced by $L/L^+$).
  By
  Th\'eor\`eme 5.4 and Corollaire 5.3 of \cite{labesse} there exists
  an automorphic representation $\Pi$ of $G(\bb{A}_{L^+})$ such that
  $\pi_L$ is a strong base change of $\Pi$.
Let $U=\prod_v U_v \subset G(\bb{A}_{L^+}^{\infty})$ be a
  compact open subgroup such that
 \begin{itemize}
\item $U_v = G(\mc{O}_{L^+_v})$ for $v \in S_{L,l}$ and for $v \not \in
   T$ split in $L$;
 \item $U_v$ is a hyperspecial maximal compact subgroup of $G(L^+_v)$
   for each $v$ inert in $L$;
\item $U_v$ is such that $\Pi_v^{U_v}\neq \{0\}$ for $v \in R_L$;
 \item $U_v = \ker(G(\mc{O}_{L^+_v})\rightarrow G(k_{v}))$ for $v \in S_{L,a}$.
 \end{itemize}

Let $\wt{I}_l$ denote the set of
embeddings $L\into K$ giving rise to one of the places $\wt{v} \in \wt{S}_{L,l}$.
Let $\lambda = (\lambda_{\wt{v}})_{v
  \in S_{L,l}}$ regarded as an element of $(\bb{Z}^n_+)^{\wt{I}_l}$ in the
evident way and let $S_{\lambda}(U,\mc{O})$ be the space of $l$-adic
automorphic forms on $G$ of weight $\lambda$ introduced above. Let $\bb{T}^{T}_{\lambda}(U,\mc{O})$ be the $\mc{O}$-subalgebra
of $\End_{\mc{O}}(S_{\lambda}(U,\mc{O}))$ generated by the Hecke operators
$T_{w}^{(j)}$, $(T_{w}^{(n)})^{-1}$ for $w$ a place of $L$ split over
$L^+$, not lying over $T$ and $j=1,\ldots,n$. The eigenvalues of
the operators $T_{w}^{(j)}$ on the space $(\iota^{-1}\Pi^{\infty})^{U}$ give
rise to a homomorphism of $\mc{O}$-algebras
$\bb{T}^{T}_{\lambda}(U,\mc{O}) \rightarrow \Qlbar$. Extending $K$ if necessary, we can
and do assume that this homomorphism takes values in $\mc{O}$. Let
$\mf{m}$ denote the unique maximal ideal of
$\bb{T}^{T}_{\lambda}(U,\mc{O})$ containing the kernel of this homomorphism.

Note that $\rbar|_{G_L}$ is $\GL_n(k)$-conjugate to $\rbar_{\mf{m}}$
where $\rbar_{\mf{m}}$ is the representation associated to the maximal
ideal $\mf{m}$ of $\bb{T}^T_{\lambda}(U,\mc{O})$ in section \ref{subsec: l-adic aut forms
  on unitary groups}. After conjugating we
can and do assume that $\rbar|_{G_L} = \rbar_{\mf{m}}$.
Since $\mf{m}$ is non-Eisenstein we have a
continuous lift 
\[ r_{\mf{m}} : G_{L^+} \rightarrow
\mc{G}_n(\bb{T}^T_{\lambda}(U,\mc{O})_{\mf{m}}) \]
of $\rbar$. Properties (0)-(3) of $r_{\mf{m}}$ and the fact that $\TT_{\mf{m}}$ is $l$-torsion free and reduced imply that $r_{\mf{m}}$ 
 is of type $\mc{S}_{L}$.
Hence $r_{\mf{m}}$ gives rise
to an $\mc{O}$-algebra homomorphism
\[  R^{\univ}_{\mc{S}_{L}} \onto \bb{T}^T_{\lambda}(U,\mc{O})_{\mf{m}} \]
(which is surjective by property (2) of $r_{\mf{m}}$). 

We define
\[ R^{\loc} := \left(\widehat{\otimes}_{v \in
    S_{L,l}}R^{\mathbf{v}_{\lambda_{\wt{v}}},cr}_{\rbarwtvL}\right)\widehat{\otimes}
\left(\widehat{\otimes}_{v \in R_L \cup
    S_{L,a}}\overline{R}^{\square}_{\rbarwtvL}\right) \] where all
completed tensor products are taken over $\mc{O}$. Note that
$R^{\loc}$ is equidimensional of dimension $1 + n^2 \# T +
[L^+:\bb{Q}]n(n-1)/2$ by Lemma 3.3 of \cite{BLGHT}.

\begin{sublem}
  There are integers $q,g \in \bb{Z}_{\geq 0}$ with
\[ 1+q+n^2 \# T = \dim R^{\loc} + g \]
and a module $M_{\infty}$ for both $R_{\infty}:=R^{\loc}[[x_1,\ldots,x_g]]$ and $S_{\infty}:=\mc{O}[[z_1,\ldots,z_{n^2 \#T},y_1,\ldots,y_q]]$ such that:
\begin{enumerate}
\item $M_{\infty}$ is finite and free over $S_{\infty}$.
\item $M_{\infty}/(z_i,y_j) \cong S_{\lambda}(U,\mc{O})_{\mf{m}}$.
\item The action of $S_{\infty}$ on $M_{\infty}$ can be factored through an $\mc{O}$-algebra homomorphism $S_{\infty} \rightarrow R_{\infty}$.
\item There is a surjection $R_{\infty} \onto R^{\univ}_{\mc{S}_L}$
  whose kernel contains all the $z_i$ and $y_j$ which is compatible
  with the actions of $R_{\infty}/(z_i,y_j)$ and
  $R^{\univ}_{\mc{S}_L}$ on $M_{\infty}/(z_i,y_j)\cong
  S_{\lambda}(U,\mc{O})_{\mf{m}}$. Moreover, there is a lift
  $r^{\univ}_{\mc{S}_L} : G_{L^+} \rightarrow
  \mc{G}_n(R^{\univ}_{\mc{S}_L})$ of $\rbar$ representing the universal
  deformation so that for each $v \in T$, the composite
  $R^{\square}_{\rbarwtvL} \rightarrow R^{\loc} \rightarrow R_{\infty}
  \onto R^{\univ}_{\mc{S}_L}$ arises from the lift
  $r^{\univ}_{\mc{S}_L}|_{G_{L_{\wt{v}}}}$.
\end{enumerate}
Moreover, $\mu \equiv n \mod 2$.
\end{sublem}

This is proved in exactly the same way as the sublemma in the proof of
Theorem \ref{thm: CM modularity lifting theorem, using tilde} (for the
fact that $\mu\equiv n\mod 2$, see the penultimate sentence of
\emph{loc. cit.}). Points
(1) and (3) tell us that the support of $M_{\infty}$ in $\Spec
R_{\infty}$ is a union of irreducible components.
Let 
\[ R^{\loc}_{r'} = \widehat{\otimes}_{v \in T}R_{\wt{v}}. \] Then
$R^{\loc}_{r'}$ is a quotient of $R^{\loc}$ corresponding to an
irreducible component (see part 5 of Lemma 3.3 of \cite{BLGHT}). Let $R_{\infty,r'} = R_{\infty}
\otimes_{R^{\loc}} R^{\loc}_{r'}$. Again $R_{\infty,r'}$ is a quotient
of $R_{\infty}$ corresponding to an irreducible component. The lift
$r'|_{G_{L^+}}$ of $\rbar|_{G_{L^+}}$ gives rise to a closed point of
$\Spec R_{\infty}[1/l]$ which lies in the support of $M_{\infty}$ and
which lies in $\Spec R_{\infty,r'}[1/l]$ but in no other irreducible
component of $\Spec R_{\infty}[1/l]$. We deduce that $\Spec
R_{\infty,r'}$ is contained in the support of $M_{\infty}$. In other
words, in the terminology of section 2 of \cite{tay06},
$M_{\infty}\otimes_{R_{\infty}} R_{\infty,r'}$ is a nearly faithful
$R_{\infty,r'}$-module. It follows (by Lemma 2.2 of \cite{tay06}) that
$M_{\infty} \otimes_{R_{\infty}} R^{\univ}_{\mc{S}_{L,r'}}$ is a
nearly faithful $R^{\univ}_{\mc{S}_{L,r'}}$-module. Note that
$M_{\infty}\otimes_{R_{\infty}} R^{\univ}_{\mc{S}_{L,r'}}$ is a finite
$\mc{O}$-module, being a quotient of
$S_{\lambda}(U,\mc{O})_{\mf{m}}$. Let $I$ denote the annihilator of
$M_{\infty}\otimes_{R_{\infty}} R^{\univ}_{\mc{S}_{L,r'}}$ in
$R^{\univ}_{\mc{S}_{L,r'}}$. Then $R^{\univ}_{\mc{S}_{L,r'}}/I$ is
finite over $\mc{O}$. The same is true of $(R^{\univ}_{\mc{S}_{L,r'}})^{\red}$ since
$I$ is nilpotent. It follows that
$R^{\univ}_{\mc{S}_{L,r'}}/\m_{\bigO}$ is Artinian (being Noetherian
of dimension 0) and hence
$R^{\univ}_{\mc{S}_{L,r'}}$ is finite over $\mc{O}$ by the topological
version of Nakayama's lemma.

We have established parts (1) and (3) of the proposition. For part
(2), it is clear from the above that any $\Qlbar$-point of
$R^{\univ}_{\mc{S}}$ gives rise to a representation $G_{F} \rightarrow
\GL_n(\Qlbar)$ which becomes automorphic upon restriction to $G_L$ and
hence is automorphic by Lemma 1.4 of \cite{BLGHT} (since $L/F$ is
solvable).
\end{proof}

\section{A character building exercise}\label{sec; character building}

\subsection{} The main purpose of this section is to prove a lemma 
allowing us to construct Galois characters with certain properties.
Before we do so, we must discuss the notion of bigness and prove
a result which will allow us to deduce that representations are `big' in certain
circumstances.

We begin by recalling the definition of $m$-big from \cite{BLGHT}.
\begin{defn}
\label{defn: m-big}
  Let $k/\F_l$ be algebraic and $m$ a positive integer. We say that a
  subgroup $H\subset\GL_n(k)$ of $\GL_n(k)$ is $m$-\emph{big} if the
  following conditions are satisfied.
  \begin{itemize}
  \item $H$ has no $l$-power order quotient.
  \item $H^0(H,\mathfrak{sl}_n(k))=(0)$.
  \item $H^1(H,\mathfrak{sl}_n(k))=(0)$.
  \item For all irreducible $k[H]$-submodules $W$ of
    $\mathfrak{gl}_n(k)$ we can find $h\in H$ and $\alpha\in k$ such
    that:
    \begin{itemize}
    \item $\alpha$ is a simple root of the characteristic polynomial
      of $h$, and if $\beta$ is any other root then
      $\alpha^m\ne\beta^m$.
    \item Let $\pi_{h,\alpha}$ (respectively $i_{h,\alpha}$) denote
      the $h$-equivariant projection from $k^n$ to the
      $\alpha$-eigenspace of $h$ (respectively the $h$-equivariant
      injection from the $\alpha$-eigenspace of $h$ to $k^n$). Then
      $\pi_{h,\alpha}\circ W\circ i_{h,\alpha}\ne 0$.
    \end{itemize}

  \end{itemize}
We simply write ``big'' for 1-big. If $\rbar$ is a representation of
some group valued in $\GL_n(k)$, then we say that $\rbar$ has $m$-big
image if the image of $\rbar$ is $m$-big. If $K$ is an algebraic
extension of $\Ql$ with residue field $k$ and $r$ is a representation of
some group valued in $\GL_n(K)$, then we say that $r$ has $m$-big
image if $\rbar$ has $m$-big image, where $\rbar$ is the
semisimplification of the reduction mod $l$ of $r$.
\end{defn}

The following lemma is essentially implicit in the proof of Theorem 7.6 of \cite{BLGHT},
but it is hard to extract by reference the material we need from the proof there,
so we will give a self-contained statement and proof.
\begin{lem} \label{lem:big image etc} Suppose that 
$F$ is a totally real field,
$l$ is a rational prime,
$m^*$ is a positive even integer not divisible by $l$,
$n$ is a positive integer with $l>2n-2$, 
$r:G_F\to\GL_n(\overline{\Z}_l)$ is a continuous $l$-adic Galois representation, and
$M$ is a cyclic CM extension of $F$ of degree $m^*$ such that:
\begin{itemize}
\item $M$ is linearly disjoint from $\bar{F}^{\ker\bar{r}}(\zeta_l)$ over $F$, and
\item every prime $v$ of $F$ above $l$ is unramified in $M$.
\end{itemize}
Suppose also that $\theta':G_M\to \overline{\Z}_l^\times$ is a continuous character.

\begin{enumerate}
\item Suppose $[\overline{F}^{\ker\ad \bar{r}}(\zeta_l):\overline{F}^{\ker\ad \bar{r}}]>m^*$.
  Then the 
    fixed field of the kernel of the representation
    $\ad(\bar{r}\tensor\Ind_{G_M}^{G_F}\bar{\theta}')$ does not contain
    $\zeta_l$.
  \item Suppose that $r|_{G_{F(\zeta_l)}}$ has $m^*$-big image and
    that $(\overline{\theta}')(\overline{\theta}')^c$ can be extended
    to $G_F$.  Suppose further that there is a prime $\fQ$ of $M$
    lying above a prime $\fq$ of $F$ lying in turn above a rational
    prime $q$, such that:
\begin{itemize}
\item $r$ is unramified at all primes above $q$,
\item $q\neq l$,
\item $q$ splits completely in $M$,
\item $q$ is unramified in $\bar{F}^{\ker\bar{r}}(\zeta_l)$,
\item $q-1>2n$,
\item $(\theta')(\theta')^c$ is unramified
  at primes above $q$, and
\item $q|\#\theta'(I_\fQ)$ (and so $q|\#\theta'(I_{\fQ^c})$), and $\theta'$ is unramified at all primes above $\fq$ except $\fQ$ and $\fQ^c$.
\end{itemize}
Then $(r\otimes\Ind_{G_M}^{G_F}\theta')|_{G_{F(\zeta_l)}}$ has big image.
\end{enumerate}
\end{lem}
\begin{proof}
  We will begin by proving the first part; so suppose that $r$ has the
  property assumed there.  By the given assumption, there is an
  element of $\Gal(\overline{F}^{\ker\ad
    \bar{r}}(\zeta_l)/\overline{F}^{\ker\ad \bar{r}})$ of order larger
  than $m^*$; using the assumption that $M$ and $\overline{F}^{\ker
    \bar{r}}(\zeta_l) \supset \overline{F}^{\ker\ad \bar{r}}(\zeta_l)$
  are linearly disjoint over $F$, we can consider this as an element
  of $\Gal(M\overline{F}^{\ker\ad
    \bar{r}}(\zeta_l)/\overline{F}^{\ker\ad \bar{r}}M)$ with the same
  property, and lift this to an element $\sigma$ of
  $\Gal(\overline{F}/M\overline{F}^{\ker\ad \bar{r}})$. (This element
  will have the property that $\overline{\eps}_l(\sigma)$ has order
  $>m^*$.)
Let $\tau$ be a generator of $\Gal(M/F)$. Then lift $\tau$ to an element $\tilde{\tau}$ of $G_F$;
let\[\sigma'=\prod_{i=0}^{m^*-1} \tilde{\tau}^{-i} \sigma
\tilde{\tau}^{i}\]and notice that $\ad(\bar{r}\otimes\Ind^{G_F}_{G_M}
\bar{\theta}')$ is trivial on $\sigma'$, while
$\overline{\eps}_l(\sigma')=\overline{\eps}_l(\sigma)^{m^*}\neq
1$. This proves (1).

\smallskip 
We now turn to proving the second part; so let us assume
that we have $k, q,\fq,\fQ$ with the properties stipulated there. We
will follow the proof of Theorem 7.6 of \cite{BLGHT} very closely.
Since $l>2n-2$, the main result of \cite{MR1253203} shows that
$\ad\rbar|_{G_{F(\zeta_l)}}$ is semisimple, and we may write
\begin{equation}\label{eq:decomposition of rbar}
  \ad \bar{r}|_{G_{F(\zeta_l)}}= V_0^{\oplus m_0} \oplus V_1^{\oplus m_1} \oplus \dots\oplus V_s^{\oplus m_s}
\end{equation}
where the $V_i$ are pairwise non-isomorphic, irreducible
$\Fbar_l[G_{F(\zeta_l)}]$-modules and the $m_i$ are positive
integers. Let $V_0=1$. By the assumption that $r|_{G_{F(\zeta_l)}}$
has $m^*$-big image, $m_0=1$. Adopting $r'$ as an abbreviated name for
$\Ind_{G_M}^{G_F}\theta'$, let us choose $e_0,\dots, e_{m^*-1}$ a
basis for $r'$ as follows. First, choose $\tilde{\tau}\in G_F$ a
lifting of $\tau\in \Gal(M/F)$. 
Then, choose a non-zero primitive vector $e_0$ in $r'$ such
that $r'(\sigma) e_0 = \theta'(\sigma) e_0$ for all $\sigma\in G_M$,
and set $e_i=r'(\tilde{\tau}^i)e_0$ for $i=1,\dots,m^*-1$.  Note that
this means that $G_M$ acts on $e_i$ via the character
${\theta'}^{\tilde{\tau}^{-i}}$. Moreover, 
\begin{align*}
r'(\tilde{\tau})(e_{m^*-1}) &= r'(\tilde{\tau}^{m^*})e_0  =
\theta'(\tilde{\tau}^{m^*})e_0
\end{align*}

 Now let $f_0,\dots, f_{m^*-1}$ be the basis of
 $\Hom(r',\overline{\Z}_l)$ dual to $e_0,\dots, e_{m^*-1}$. Let us
 quickly establish a sublemma:
 \begin{sublem} Suppose there is some character $\psi:G_M\to\Fbar_l^\times$
 which is unramified at primes above $q$ such that
   $\thbarpr/\thbarpr^{\tilde{\tau}^k} =
   \thbarpr^{\tilde{\tau}^l}/\thbarpr^{\tilde{\tau}^{m}} \psi$ (or
   equivalently $\thbarpr\thbarpr^{\tilde{\tau}^{m}} =
   {\thbarpr^{\tilde{\tau}^l}\thbarpr^{\tilde{\tau}^k}}\psi$); then either
   $k=m$ and $l=0$, $k=0$ and $l=m$, or $l=(m^*/2)+k$ and $m=m^*/2$. 
   Conversely, if $k=m$ and $l=0$, then $\thbarpr/\thbarpr^{\tilde{\tau}^k} =
   \thbarpr^{\tilde{\tau}^l}/\thbarpr^{\tilde{\tau}^{m}}$
   and the same is true if $k=0$ and $l=m$, or if $l=(m^*/2)+k$ and $m=m^*/2$.
 \end{sublem}
 \begin{proof} For the first part, we consider the action of
   inertia above $\fq$ on each side of
   $\thbarpr\thbarpr^{\tilde{\tau}^{m}} =
   {\thbarpr^{\tilde{\tau}^l}\thbarpr^{\tilde{\tau}^k}}$.  
   
   We first set up some notation. Let $\iota_{I_\fQ}$ denote the inclusion of $I_\fQ$
   into $G_M$, and let $\delta:=\thbarpr\circ\iota_{I_\fQ}$. (Note that $\delta$ is 
   not trivial, by our assumption that $q|\#\theta'(I_\fQ)$.) For $i$ an integer,
   let $c_{{\tilde{\tau}^i}}:G_M\to G_M$
   denote the `conjugation by ${\tilde{\tau}^i}$' map, and let
   $\delta_i:=\thbarpr\circ c_{\tilde{\tau}^i}\circ\iota_{I_\fQ}$. Observe that 
   $\delta_0=\delta$, that $\delta_i$ is trivial for $i\neq 0, m^*/2$
   (by the assumption that $\theta'$ is unramified at all primes above $\fq$
   except $\fQ$ and $\fQ^c$), and that
   $\delta_{m^*/2}=\delta^{-1}$ (since $\delta_{m^*/2} = 
   (\thbarpr\circ c_{\tilde{\tau}^{m^*/2}}\circ\iota_{I_\fQ}) = 
   (\thbarpr^{\tilde{\tau}^{m^*/2}}\circ\iota_{I_\fQ}) = 
   (\thbarpr^c)\circ\iota_{I_\fQ}) = 
   (1/\thbarpr)\circ\iota_{I_\fQ}) = 1/\delta$, where the penultimate equality
   is by our assumption that $(\theta')(\theta')^c$ is unramified at primes above $q$).
   
   Now let us consider the maps
   $\beta_i:=(\thbarpr\thbarpr^{\tilde{\tau}^{m}})\circ c_{\tilde{\tau}^i}\circ\iota_{I_\fQ} $
   for each $i$. These maps are clearly trivial unless $i=0$, $m^*/2$, $m$ or $m+(m^*/2)$
   (considering $i$ mod $m^*$). Let us
   split into cases:
   \begin{itemize}
   \item {\sl We have $m\neq (m^*/2),0$.} In this case, it is easy to
     see that the values $i=0$, $m^*/2$, $m$ or $m+(m^*/2)$ mod $m^*$ are
     distinct. It is also easy to see that $\beta_i$ for these $i$ values are
     (respectively)
     $\delta,\delta^{-1},\delta,\delta^{-1}$, where $\delta$ is some
     character. 
     
     On the other hand, consider 
     $\beta'_i:=(\thbarpr^{\tilde{\tau}^{l}}\thbarpr^{\tilde{\tau}^{k}}\psi)\circ
      c_{\tilde{\tau}^i}\circ\iota_{I_\fQ} $.
     For these maps to be nontrivial for 4 distinct values of $i$ mod $m^*$,
     as they must, we need
     $l\neq k,k+(m^*/2)$; then it will be nontrivial for
     $i=l$, $l+(m^*/2)$, $k$, $k+(m^*/2)$, for which values of $i$ we
     get $\delta,\delta^{-1},\delta,\delta^{-1}$ respectively; and
     comparing, we conclude that $l=0$ and $k=m$ or $l=m$ and $k=0$
     (since $\delta \neq \delta^{-1}$).
   \item {\sl We have $m=0$.} In this case, $\beta_i$ is nontrivial for
   $i=0$ or $i=m^*/2$ (mod $m^*$), which are distinct, and the
   corresponding $\beta_i$ are (respectively)
     $\delta^2,\delta^{-2}$. Now consider 
     $\beta'_i:=(\thbarpr^{\tilde{\tau}^{l}}\thbarpr^{\tilde{\tau}^{k}}\psi)\circ
      c_{\tilde{\tau}^i}\circ\iota_{I_\fQ} $. For these maps to be nontrivial for
      at most two values of $i$ we need $l=k$ or $l=k+(m^*/2)$, and in the
     latter case we see that \emph{all} $\beta'_i$ are in fact trivial. Thus 
     we must have
     $l=k$. Comparing the $i$ for which $\beta_i$ and $\beta'_i$ are
     nontrivial, we see that $l=0$.
   \item {\sl We have $m=m^*/2$.} In this case, it is easy to see that all $\beta_i$
   are trivial. Now, considering $\beta'_i:=(\thbarpr^{\tilde{\tau}^{l}}\thbarpr^{\tilde{\tau}^{k}}\psi)\circ
      c_{\tilde{\tau}^i}\circ\iota_{I_\fQ} $, we see that for these maps to be nontrivial for
      at most two values of $i$ (as they must be) we need $l=k$ or $l=k+(m^*/2)$, and we see that
      only in the latter case is it in
     fact trivial for all values of $i$.
   \end{itemize}
   For the converse, if $k=m$ and $l=0$, the fact that $\thbarpr/\thbarpr^{\tilde{\tau}^k}
   = \thbarpr/\thbarpr^{\tilde{\tau}^{k}}$ is trivial. Similarly, if
   $k=0$ and $l=m$, we need to prove that  $\thbarpr/\thbarpr
   = \thbarpr^{\tilde{\tau}^m}/\thbarpr^{\tilde{\tau}^{m}}$, which is
 also trivial. Finally, if  $l=(m^*/2)+k$ and $m=m^*/2$, it remains
   to show that
   $(\thbarpr)(\thbarpr)^c=((\thbarpr)(\thbarpr)^c)^{\tilde{\tau}^k}$
   (where we write $c$ for complex conjugation), which is true since
   one of our hypotheses is that
   $(\thbarpr)(\thbarpr)^c$ can be extended to $G_F$.
  \end{proof}
We can now decompose:
 \begin{equation} \label{eq:decomposition of rbarprime}
 \ad\bar{r}'|_{G_{F(\zeta_l)}} =\left(
   \bigoplus_{\chi\in\Hom(\Gal({M(\zeta_l)}/{F(\zeta_l)}),\Fbar_l^\times)} W_\chi
 \right)\oplus\left(
   \bigoplus_{i=1}^{m^*-1} W_i
 \right)
 \end{equation}
where
\begin{itemize}
\item $W_\chi$ is the span of $\sum_{i=0}^{m^*-1} \chi^{-1}(\tau^i)
  e_i \otimes f_i$, so $W_\chi\cong \Fbar_{l}(\chi)$.
\item $W_i$ is the span of $\{e_j\otimes f_{i+j}\}_{j=0,\dots,m^*-1}$,
  so $W_i\cong
  \Ind^{G_{F(\zeta_l)}}_{G_{M(\zeta_l)}}(\thbarpr/\thbarpr^{\tilde{\tau}^{-i}})$.
\end{itemize}

From this we turn to study $\rbar\otimes\rbar'$, which we will
abbreviate $\rbar''$.  We can decompose:
\begin{equation} \label{eq:decomposition of rbardoubleprime}
\ad\rbar''|_{G_{F(\zeta_l)}}=\left(
   \bigoplus_{j=0}^s\bigoplus_{\chi\in\Hom(\Gal({M(\zeta_l)}/{F(\zeta_l)}),\Fbar_l^\times)} V_j(\chi)^{m_j}
 \right)\oplus\left(\bigoplus_{j=0}^s
   \bigoplus_{i=1}^{m^*-1} (V_j\otimes W_i)^{m_j}\right)
\end{equation}

We then have the following straightforward lemma:
\begin{lemunnum} The decomposition \ref{eq:decomposition of rbardoubleprime} enjoys the following properties:
\begin{enumerate}
\item Each $V_j(\chi)$ is irreducible.
\item We have $V_j(\chi)\not\equiv V_{j'}(\chi')$ unless
  $\chi=\chi'$ and $j=j'$.
\item Each $V_j\otimes
  W_i\cong\Ind^{G_{F(\zeta_l)}}_{G_{M(\zeta_l)}}(V_j\otimes\thbarpr/\thbarpr^{\tilde{\tau}^i})$
  is irreducible; moreover, we have
  $\Ind^{G_{F(\zeta_l)}}_{G_{M(\zeta_l)}}(V_j\otimes\thbarpr/\thbarpr^{\tilde{\tau}^i})
  \not\cong
  \Ind^{G_{F(\zeta_l)}}_{G_{M(\zeta_l)}}(V_{j'}\otimes\thbarpr/\thbarpr^{\tilde{\tau}^{i'}})
  $ unless $j=j'$ and $i\in\{i',m^*-i'\}$.
\item $V_j(\chi) \not\cong
  \Ind^{G_{F(\zeta_l)}}_{G_{M(\zeta_l)}}(V_{j'}\otimes\thbarpr/\thbarpr^{\tilde{\tau}^{i'}})
  $ for all $\chi,i',j,j'$.
\end{enumerate}
\end{lemunnum}
\begin{proof}
For part 1, see equation \ref{eq:decomposition of rbar} and the remarks immediately following it.
Part 2 is clear, because $M$ is linearly disjoint from $\bar{F}^{\ker\rbar}(\zeta_l)$ over $F$.
Both of the assertions in part 3 follow from \emph{(a)} the
fact that the only cases where $G_{M(\zeta_l)}$ acts via the same
character on $e_0\otimes f_k$ and $e_l\otimes f_m$ are when $k=m$
and $l=0$, $k=0$ and $l=m$, or $l=(m^*/2)+k$ and $m=m^*/2$ (this fact is
our sublemma) and \emph{(b)} from examining the action of inertia above $q$.
Finally, part 4 again follows from a consideration of the action of inertia above $q$.
\end{proof}

It is immediate (since there are no terms in the
decomposition of $\ad\rbar''|_{G_{F(\zeta_l)}}$ isomorphic to $1$
other than $V_0(1)^{m_0}$) that
$$H^0((\ad \rbar'')(G_{F(\zeta_l)}),\ad^0\rbar'') = (0).$$

Let $H$ (resp.\ $H'$, resp.\ $H''$) denote the image $(\ad
\rbar)(G_{F(\zeta_l)})$ (resp.\ $(\ad \rbar')(G_{F(\zeta_l)})$, resp.\
$(\ad \rbar'')(G_{F(\zeta_l)})$). We note that the only element of $H$
which is a scalar when thought of as an element of $\Aut\ad\rbar$ is
the identity\footnote{In case it will help avoid confusion, let us
  spell this out a little more. The representation $\rbar$ can
  equivalently be thought of as a vector space $V_{\rbar}$ on which
  the Galois group acts; similarly, we can think of the representation
  $\ad\rbar$ as a vector space $V_{\ad\rbar}$ with an action of the
  Galois group; the underlying vector space of $V_{\ad\rbar}$ is just
  the vector space $\End_{\mathrm{VectSpc}} V_{\rbar}$ of
endomorphisms of the underlying vector space of $V_{\rbar}$.  Each
element $h$ of $H$ determines an element of $\Aut_{\mathrm{VectSpc}}
V_{\ad\rbar}$; this is what we mean by $h$ `thought of as an element
of $\Aut\ad\rbar$'.}; and the same is true for $H'$, so that
$$H \times H' \hookrightarrow\Aut(\ad \rbar'')\,\text{ and
}\!\ker\ad\rbar''|_{G_{F(\zeta_l)}}=\ker\ad\rbar|_{G_{F(\zeta_l)}}\cap\ker\ad\rbar'|_{G_{F(\zeta_l)}}\,\,\text{(=$K_\cap$,
  say)}$$ Thus, if we define $K_\cup=\langle(\ker \ad \rbar|_{G_{F(\zeta_l)}}),(\ker \ad
\rbar'|_{G_{F(\zeta_l)}})\rangle\triangleleft G_{F(\zeta_l)}$, and
$\overline{H}=G_{F(\zeta_l)}/K_\cup$, we have maps
$$
\pi:H=\frac{G_{F(\zeta_l)}}{\ker \ad \rbar|_{G_{F(\zeta_l)}}} \twoheadrightarrow\frac{G_{F(\zeta_l)}}{K_\cup}=\overline{H}
\quad\text{and}\quad
\pi':H'=\frac{G_{F(\zeta_l)}}{\ker \ad \rbar'|_{G_{F(\zeta_l)}}} \twoheadrightarrow\frac{G_{F(\zeta_l)}}{K_\cup}=\overline{H}
$$ 
such that $G_{F(\zeta_l)} \twoheadrightarrow \{(h,h')\in H\times
H'|\pi(h)=\pi'(h')\}$. (This is surjective; it suffices to show a)
that for each $h \in H$, the image contains $(h,x)$ for some $x$, and
b) that the image contains $(e,h')$ for each $h'$ in the kernel of
$\pi'$, both of which are clear.)

Finally,
\begin{align*}
\ker (G_{F(\zeta_l)}&\onto \{(h,h')\in H\times H'|\pi(h)=\pi'(h')\}) =\ker (G_{F(\zeta_l)}\to H\times H')\\
&=\ker (G_{F(\zeta_l)}\to H) \cap\ker (G_{F(\zeta_l)}\to H')=K_\cap=\ker \ad\rbar''|_{G_{F(\zeta_l)}}
\end{align*}
and hence there is an isomorphism $H''\cong G_{F(\zeta_l)}/\ker\ad\rbar''|_{G_{F(\zeta_l)}}\cong\{(h,h')\in H\times H'|\pi(h)=\pi'(h')\}$.

Let $K'$ denote the kernel of $\pi'$. It is the case that:
\begin{itemize}
\item The image of the inertia group at any prime above $q$ in $H'$ is
  contained in $K'$ (as $\rbar$ is unramified above $q$).
\item $K'\hookrightarrow H'\onto \Gal(M(\zeta_l)/F(\zeta_l))$ is
  surjective (since $M$ is linearly disjoint from
  $\overline{F}^{\ker \rbar}(\zeta_l)$ over
  $F$).
\end{itemize}
From this we easily see that $(\ad \rbar'')^{K'}=\ad \rbar$ (use the
decomposition (\ref{eq:decomposition of rbardoubleprime})), and hence
(by inflation-restriction, using the fact that $l\notdiv\#K'$, so that
the group $H^1(K',\ad^0\rbar'')$ vanishes, and the assumption that
$\rbar|_{G_{F(\zeta_l)}}$ has big image) that:
$$(0)=H^1(H,\ad^0\rbar)=H^1(H''/K',(\ad^0\rbar'')^{K'}) \overset{\sim}{\rightarrow} H^1(H'',\ad^0\rbar'').$$

All that still remains is to show the non-group-cohomology-related
part of the definition of `big image'.  To this end, let us fix a copy
$V_j\subset \ad \bar{r}$. (Note that every copy of
$V_j(\chi)\subset\ad\bar{r}''$ will be of the form $V_j\otimes W_\chi$
for some copy of $V_j\subset\ad\bar{r}$; this uses our analysis,
above, of the conditions under which terms in the direct sum
(\ref{eq:decomposition of rbardoubleprime}) are isomorphic.) Since
$r|_{G_{F(\zeta_l)}}$ has $m^*$-big image, hence big image, and $M$ is linearly
disjoint from $\overline{F}^{\ker\rbar}(\zeta_l)$ over $F$, we see that
$r|_{G_{M(\zeta_l)}}$ has big image. Thus we can find a $\sigma\in
G_{M(\zeta_l)}$ and a simple root $\alpha$ of the characteristic
polynomial $\det(X-\bar{r}(\sigma))$ such that
$\pi_{\bar{r}(\sigma),\alpha} V_j i_{\bar{r}(\sigma),\alpha}\neq
(0)$. Altering $\sigma$ by elements of inertia subgroups at primes
above $q$ (which does not affect $\rbar(\sigma)$), we can
assume\footnote{In particular, we note that we can alter the
  quantities $\thbarpr^{\tilde{\tau}^i}(\sigma)$ for $1\leq i\leq
  (m^*/2)-1$ independently of each other, and of $\thbarpr(\sigma)$,
  since altering $\sigma$ by inertia at a prime $\fQ^{\tau^i}$ will
  only affect $\thbarpr^{\tilde{\tau}^i}(\sigma)$ and
  $\thbarpr^{\tilde{\tau}^{i+(m^*/2)}}(\sigma)$. Moreover,
  since $(\thbarpr/\thbarpr^{\tilde{\tau}^i})(\sigma)$ and
  $(\thbarpr/\thbarpr^{\tilde{\tau}^{i+(m^*/2)}})(\sigma)$ are related
  via the fact that 
  $(\thbarpr)(\thbarpr)^c$ is unramified at all places of $M$ above
  $q$, we may ensure that  $(\thbarpr/\thbarpr^{\tilde{\tau}^{i+(m^*/2)}})(\sigma)$ avoids
  taking the value $\alpha'/\alpha$ by ensuring that
  $\thbarpr^{\tilde{\tau}^i}(\sigma)$ avoids the value   $(\alpha'(\thbarpr\thbarpr^c)(\sigma))/(\thbarpr(\sigma)\alpha)$.
}
that for $i=1,\dots, m^*-1$ the ratio
$(\thbarpr/\thbarpr^{\tilde{\tau}^i})(\sigma)$ does not equal
$\alpha'/\alpha$ for any root $\alpha'$ (including $\alpha$) of the
characteristic polynomial $\det(X-\bar{r}(\sigma))$. (That this is
possible relies on the fact that $q>2n+1$.)  Thus
$\alpha\thbarpr(\sigma)$ is a simple root of the characteristic
polynomial of $\rbar''(\sigma)$ and, for each $\chi$,
\begin{align*}
\pi&_{\rbar''(\sigma),\alpha\thbarpr(\sigma)}\circ V_j(\chi)\circ i_{\rbar''(\sigma),\alpha\thbarpr(\sigma)} \\
&=
(\pi_{\rbar(\sigma),\alpha}\circ V_j\circ i_{\rbar(\sigma),\alpha}) 
\left(\pi_{\rbar'(\sigma),\thbarpr(\sigma)} \circ\left(\sum_{i=0}^{m^*-1} \chi^{-1}(\tau^i) e_i \otimes f_i\right)\circ i_{\rbar'(\sigma),\thbarpr(\sigma)}\right)\\
&=\pi_{\rbar(\sigma),\alpha}\circ V_j\circ i_{\rbar(\sigma),\alpha} \neq (0)
\end{align*}

Next, let us fix $j\in\{0,\dots ,s\}$ and $i\in\{1,\dots,m^*/2\}$, and let 
$\gamma:W_i\overset{\sim}{\to}W_{m^*-i}$ be the isomorphism such 
that $\gamma(e_0\otimes f_i)=e_{(m^*/2)+i}\otimes f_{(m^*/2)}$. (In the 
special case $2i=m^*$, $\gamma$ will  happen to be the identity; we will
in fact not make use of $\gamma$ in this case.)  We can write any submodule of 
$\ad\rbar''|_{G_{F(\zeta_l)}}$ which is isomorphic to $V_j\otimes W_i$ as:
$$\{\eta_1(v)\otimes w 
                      + \eta_2(v)\otimes \gamma(w):v\in V_j, w\in W_i\}$$
where $\eta_1,\eta_2$ are embeddings 
$V_j\hookrightarrow \ad\rbar$, and where we suppress 
the second term in the sum if $2i=m^*$. (This uses our analysis, 
above, of the conditions under which terms in the direct sum 
\ref{eq:decomposition of rbardoubleprime} are isomorphic.) Using the 
fact that $r|_{G_{F(\zeta_l)}}$ is $m^*$-big, we can find a 
$\sigma\in G_{F(\zeta_l)}$ and a root $\alpha$ of $\det(X-\bar{r}(\sigma))$ such that:
\begin{itemize}
\item $\pi_{\bar{r}(\sigma),\alpha} \circ V_j \circ i_{\bar{r}(\sigma),\alpha}\neq (0)$.
\item No other root of $\det(X-\bar{r}(\sigma))$ has $m^*$th power equal to $\alpha^{m^*}$.
\end{itemize}
Since $M$ is linearly disjoint from $\overline{F}^{\ker \rbar}(\zeta_l)$ over $F$, we may additionally assume:
\begin{itemize}
\item $\sigma$ maps to the generator $\tau$ of $\Gal(M(\zeta_l)/F(\zeta_l))$.
\end{itemize}
Define $\beta_0,\dots\beta_{m^*-1}$ by:
$$\rbar'(\sigma) e_i = \beta_i e_{i+1}$$
(where we take subscripts modulo $m^*$). The roots of the
characteristic polynomial of $\rbar'(\sigma)$ are exactly the $m^*$th
roots of $\beta_0\beta_1\dots\beta_{m^*-1}$. If $\beta$ is such a root,
so $\beta^{m^*}=\beta_0\beta_1\dots\beta_{m^*-1}$, then a corresponding
eigenvector is:
$$ v_{\beta}:=e_0 + \frac{\beta_0}{\beta} e_1 +  \frac{\beta_0\beta_1}{\beta^2} e_2 + \dots + 
 \frac{\beta_0\dots\beta_{m^*-2}}{\beta^{m^*-1}} e_{m^*-1}
$$
and the corresponding equivariant projection is
$$\pi_{\rbar'(\sigma),\beta} e_j = \frac{\beta^j}{m^*\beta_0\beta_1\dots\beta_{j-1}} v_\beta
$$ 
We see that $\alpha\beta$ is a simple root of the characteristic polynomial $\det(X-\bar{r}''(\sigma))$, and that for $v\in V_j$
\begin{align*}
\pi&_{\rbar''(\sigma),\alpha\beta}\circ (\eta_1(v)\otimes e_0\otimes
f_i + \eta_2(v)\otimes e_{(m^*/2)+i}\otimes f_{(m^*/2)}) \circ i_{\rbar''(\sigma),\alpha\beta} \\
&=
(\pi_{\rbar(\sigma),\alpha}\circ \eta_1(v)\circ i_{\rbar(\sigma),\alpha}) 
(\pi_{\rbar'(\sigma),\beta}\circ (e_{0}\otimes f_i)\circ i_{\rbar'(\sigma),\beta}) \\
&\hspace{4cm}+
(\pi_{\rbar(\sigma),\alpha}\circ \eta_2(v)\circ i_{\rbar(\sigma),\alpha}) 
(\pi_{\rbar'(\sigma),\beta}\circ (e_{(m^*/2)+i}\otimes f_{(m^*/2)})
\circ i_{\rbar'(\sigma),\beta}) \\
&=
(\pi_{\rbar(\sigma),\alpha}\circ \eta_1(v)\circ i_{\rbar(\sigma),\alpha})
\frac{1}{m^*}\frac{\beta_0\dots\beta_{i-1}}{\beta^i} \\
&\hspace{4cm}+
(\pi_{\rbar(\sigma),\alpha}\circ \eta_2(v)\circ i_{\rbar(\sigma),\alpha}) 
\frac{\beta^{{m^*/2}+i}}{m^* \beta_0\dots \beta_{{m^*/2}+i-1}}\frac{\beta_0\dots\beta_{{m^*/2}-1}}{\beta^{m^*/2}}\\
&=\frac{1}{m^*}\Big(
(\pi_{\rbar(\sigma),\alpha}\circ \eta_1(v)\circ i_{\rbar(\sigma),\alpha})
\frac{\beta_0\dots\beta_{i-1}}{\beta^i} +
(\pi_{\rbar(\sigma),\alpha}\circ \eta_2(v)\circ i_{\rbar(\sigma),\alpha})
\frac{\beta^{i}}{\beta_{m^*/2}\dots \beta_{{m^*/2}+i-1}}\Big).
\end{align*}
This will be nonzero for some choice of $\beta$ and $v$.

Since all terms in the sum (\ref{eq:decomposition of rbardoubleprime})
are isomorphic either to a $V_j(\chi)$ or a $W_i\otimes V_j$, the only
remaining point is to check that $\rbar''(G_{F(\zeta_l)}) $ has no
quotients of $l$-power order. It suffices to prove that
$H''=(\ad\rbar'')(G_{F(\zeta_l)})$ has no quotients of $l$-power order
(because the group of scalar matrices in $\GL_{nm^*}(\Flbar)$ has no
elements of order divisible by $l$). Since $m^*$ is not divisible by
$l$, $H'$ has order prime to $l$, and we see that any quotient of
$H''$ of $l$-power order would also be a quotient of $H$. Since
$r|_{G_{F(\zeta_l)}}$ has $m^*$-big image, $H$ has no quotients of
$l$-power order, and we are done.
\end{proof}

\subsection{} We now return to the main business of section: constructing characters.
For the entire remainder of this section, we will be working with the following
combinatorial data in the background:

\begin{situation} \label{character-lemma-sitn} Suppose that $F$ is a
  totally real field, $l$ is a rational prime which splits completely
  in $F$, and that we are given the following data:
  \begin{itemize}
  \item A partition of the set of primes above $l$ into two subsets
    $S_{\ord}$ and $S_{\ssg}$,
  \item For each prime $v$ above $l$, integers $a_v$ and $b_v$, and
  \item A set $T$ of places of $F$, not containing places above $l$.
  \end{itemize}
  such that the sum $-2a_v+b_v$ takes some fixed value, $w$ say,
  independent of $v$.
\end{situation}

(It may be helpful to the reader if we remark that this combinatorial
data is intended to be related to the automorphic representation $\pi$
of $\GL_2(\A_F)$ with which we will eventually be working in the
following manner: the sets $S_{\ord}$, $S_{\ssg}$ reflect the places
above $l$ where $\pi$ is ordinary and where it is supersingular; $\pi$
is thought of as being associated to a Galois representation having
Hodge-Tate numbers $\{-a_{v},b_{v}-a_{v}\}$ at the place $v$; and the
set $T$ contains the places away from $l$ where $\pi$ is ramified.)

We define a certain integer $m^*$, dependent on the set of $b_v$'s of
Situation \ref{character-lemma-sitn}, but not on the prime $l$ itself.
\begin{defn}
Let $B=\{b_v|\text{$v$ a prime of $F$ above $l$}\}$, considered as a
set \emph{without} multiplicity. We define the integer $m^*$ to be the
least common multiple of the integers in the set $\{2\}\cup B$.
\end{defn}

We have seen in the previous section that our lifting
theorems require us to maintain careful control of the lattices with which
we work. We therefore single out certain lattices which will be important
in the sequel.
\begin{defn} \label{standard bases}
Suppose we are in the situation of Situation \ref{character-lemma-sitn}. 
We make the following definitions:
\begin{enumerate}
\item
Suppose that $v\in S_{\ssg}$ is a place of $F$ above $l$, and
    let $L$ be the quadratic unramified extension of $F_v$ in $\overline{F}_v$ (so that $L$
    is isomorphic to $\Q_{l^2}$). 
     Let $K$ be a finite extension of $\Q_l$ with ring of integers $\bigO$,
    and suppose that
    $\chi:G_{L}\to\bigO^\times$ is a de Rham character.
    Finally suppose we have chosen $\sigma\in G_{F_v}$ 
    mapping to a generator of $\Gal(L/F_v)$. Then we can consider the ordered 
    $\bigO$-basis $\{f_0,f_1\}$ of $\Ind_{G_L}^{G_{F_v}}\chi$ where
    $f_i:G_{F_v}\to\bigO$ is the function supported on $\sigma^iG_L$ and
    taking the value $1$ on $\sigma^i$. We call this the \emph{$\sigma$-standard}
    basis for $\Ind_{G_L}^{G_{F_v}}\chi$.
\item Continue the assumptions of the previous point. We get an ordered basis 
$\{g_0,\dots,g_{n-1}\}$ of $\Sym^{n-1}\Ind_{G_L}^{G_{F_v}}\chi$
inherited (in the sense of 
  Definition \ref{defn:inherited basis}) from the $\sigma$-standard
  basis on $\Ind_{G_L}^{G_{F_v}}\chi$ defined in the previous part. We
  call this the \emph{$\sigma$-standard} basis of 
  $\Sym^{n-1}\Ind_{G_L}^{G_{F_v}}\chi$. (Concretely, considering
$\Sym^{n-1}\Ind_{G_L}^{G_{F_v}}\chi$ as a quotient of
$(\Ind_{G_L}^{G_{F_v}}\chi)^{\otimes(n-1)}$, 
$g_i$ is the image of $f_0^{\otimes(n-1-i)}\otimes f_1^{\otimes
  i}$.)
\item Suppose $M/F$ is a cyclic degree $m^*$  extension with $F$ totally real and $M$ CM,
$K$ is a finite extension of $\Q_l$ with ring of integers $\bigO$,
$\wt{\tau} \in G_{F}$ an element mapping to a generator $\tau \in \Gal(M/F)$, and
$\theta:G_M\to \bigO^\times$ is a character. We consider the ordered
 $\mc{O}$-basis $\{ e_0,\ldots,e_{m^* -1}\}$ 
 of $\Ind_{G_{M}}^{G_{F}}\theta$, where $e_{i} : G_{F} \rightarrow \mc{O}$
  is the function supported on $\wt\tau^iG_{M}$ with value 1 on
  $\wt\tau^i$, and call this the \emph{$\wt{\tau}$-standard basis} for 
  $\Ind_{G_{M}}^{G_{F}}\theta$.

\item If $v,L,K,M,\chi,\theta,\sigma,\wt{\tau}$ are as above, then we consider 
$\Sym^{n-1}\Ind_{G_L}^{G_{F_v}}\chi\otimes\Ind_{G_{M}}^{G_{F}}\theta$,
which is a representation of $G_{F_v}$, and has an ordered basis inherited (in the sense
of Definition \ref{defn:inherited basis for otimes}) from the $\sigma$-standard and 
$\wt{\tau}$-standard bases on $\Sym^{n-1}\Ind_{G_L}^{G_{F_v}}\chi$ and 
$\Ind_{G_{M}}^{G_{F}}\theta$ already defined.
We call this the \emph{$(\sigma,\wt{\tau})$-standard basis}
  of  $\Sym^{n-1}\Ind_{G_L}^{G_{F_v}}\chi\otimes\Ind_{G_{M}}^{G_{F}}\theta$.
\end{enumerate}
\end{defn}

We are now in a position to construct the characters we will need.
\begin{lem}\label{lem: existence of chars with long list of properties} 
Suppose we are in the situation described in Situation
\ref{character-lemma-sitn}, and we have fixed an integer $n$ and an
extension $F^{(\mathrm{bad})}$ of $F$. Assume that $l\nmid m^*$ and
$l>2n-2$. Then we can find a degree $m^*$ cyclic CM extension $M$ of $F$,
linearly disjoint from $F^{(\mathrm{bad})}$ over $F$, and continuous
characters $$\theta,\theta':G_M \to \overline{\Z}_{l}^\times,$$ which
are de Rham at all primes above $l$, and
which enjoy the following further properties:
  \begin{enumerate}
  \item $\theta, \theta'$ are congruent (mod $l$).
\item Suppose that $v\in S_{\ssg}$ is a place of $F$ above $l$, and
    let $L$ be the quadratic unramified extension of $F_v$ in $\overline{F}_v$ (so that $L$
    is isomorphic to $\Q_{l^2}$). Suppose that
    $\chi$, $\chi':G_{L}\to\Qlbar^\times$ are de Rham characters with
    $\overline{\chi}=\overline{\chi}'$. Suppose furthermore that the
    Hodge-Tate weights of $\chi$ are $-a_v$ and $b_v-a_v$, while those
    of $\chi'$ are $0$ and $1$. 

Let $K\subset \Qlbar$ be a finite extension of $\Q_l$ with ring of integers $\bigO$,
and suppose that $K$ is large enough that $\theta$, $\theta'$, $\chi$
and $\chi'$ are all valued in $\bigO$. Let $\sigma\in
G_{F_v}$ be an element mapping to a generator of $\Gal(L/F_v)$, and 
$\wt{\tau} \in G_{F}$ an element mapping to a generator $\tau \in
\Gal(M/F)$.

  Let $L'$ be a finite extension of $L$ in $\overline{F}_v$ such that
  $\chi|_{G_{L'}}$, $\chi'_{G_{L'}}$, $\theta|_{G_{L'}}$, and
  $\theta'|_{G_{L'}}$ are all
  crystalline. Let 
  \begin{align*}
  \rho_\chi&=(\Sym^{n-1}\Ind_{G_L}^{G_{F_v}}\chi)|_{G_{L'}}\otimes
  (\Ind_{G_{M}}^{G_{F}}\theta)|_{G_{L'}},\\ \text{and }
  \rho_{\chi'}&=(\Sym^{n-1}\Ind_{G_L}^{G_{F_v}}\chi')|_{G_{L'}}\otimes
  (\Ind_{G_{M}}^{G_{F}}\theta')|_{G_{L'}},
  \end{align*} regarded as
  representations $G_{L'}\to\GL_{nm^*}(\bigO)$ with respect to their
  $(\sigma,\wt{\tau})$-standard bases. Note $\rho_\chi$ and
  $\rho'_{\chi'}$ become equal after composition with the homomorphism
  $\GL_{nm^*}(\bigO)\to\GL_{nm^*}(k)$. Assume in fact that $L'$ has
  been chosen so that this common composite is the trivial
  representation.

Then $\rho_\chi\sim\rho_{\chi'}$ (in the sense of Definition \ref{defn:tilde}).

    \item For any $r:\Gal(\overline{F}/F)\to \GL_n(\overline{\Z}_l)$, a
    continuous Galois representation ramified only at primes in
    $T$ and above $l$, which satisfies $\overline{F}^{\ker \bar{r}}(\zeta_l)\subset
    F^{(\mathrm{bad})}$:
    \begin{itemize}
    \item  If $r|_{G_{F(\zeta_l)}}$ has $m^*$-big image, then
    $(r\tensor\Ind_{G_M}^{G_F} \theta)|_{G_{F(\zeta_l)}}$ and $(r\tensor\Ind_{G_M}^{G_F}
    \theta')|_{G_{F(\zeta_l)}}$ have big image.
    \item If $[\overline{F}^{\ker\ad \bar{r}}(\zeta_l):\overline{F}^{\ker\ad \bar{r}}]>m^*$ then neither the 
    fixed field of the kernel of  $\ad(\bar{r}\tensor\Ind_{G_M}^{G_F} \bar{\theta})$ 
    nor that of 
    $\ad(\bar{r}\tensor\Ind_{G_M}^{G_F}\bar{\theta}')$ will contain
    $\zeta_l$.
    \end{itemize}
  \item 
   We can put a perfect pairing on $\Ind_{G_M}^{G_F} \theta$ satisfying
     \begin{enumerate}
     \item $\langle v_1,v_2\rangle=(-1)^n\langle v_2,v_1\rangle$.
     \item For $\sigma \in G_F$, we have 
        $$\langle \sigma v_1, \sigma v_2\rangle 
             = \eps_l(\sigma)^{-m^*n+1-(1-n)w}
             \wtilde(\sigma)^{-(w-1)(n-1)} \langle  v_1, v_2\rangle$$
             where $\wtilde$ is the Teichm\"{u}ller lift of the mod
             $l$ cyclotomic character.
     \end{enumerate}
     Thus, in particular, 
     $$(\Ind_{G_M}^{G_F} \theta)^\vee\cong (\Ind_{G_M}^{G_F} \theta)
     \tensor \eps_l^{m^*n-1+(1-n)w}\wtilde^{(w-1)(n-1)}.$$ (Note that
     the character on the right hand side takes the value $(-1)^n$ on
     complex conjugations.)  \item Similarly, we can put a perfect
     pairing on $\Ind_{G_M}^{G_F} \theta'$ satisfying
     \begin{enumerate}
     \item $\langle v_1,v_2\rangle=(-1)^n\langle v_2,v_1\rangle$.
     \item For $\sigma \in G_F$, we have 
        $\langle \sigma v_1, \sigma v_2\rangle 
           = \eps_l(\sigma)^{-(m^*-1)n} \langle  v_1, v_2\rangle.$
     \end{enumerate}
     Thus, in particular, 
     $(\Ind_{G_M}^{G_F} \theta')^\vee\cong (\Ind_{G_M}^{G_F} \theta') 
                                                                               \tensor \eps_l^{(m^*-1)n}$.

     \noindent(Again the character on the right hand side is $(-1)^n$ on complex conjugations.)
  \item Suppose $r:\Gal(\overline{F}/F)\to \GL_2(\overline{\Z}_l)$ is
    a continuous representation with Hodge-Tate weights
    $\{-a_v,b_v-a_v\}$ at $v$ for each place $v$ of $F$ above $l$;
    then $\Sym^{n-1} r \tensor \Ind_{G_M}^{G_F} \theta$ has the
    following Hodge-Tate weights at $v$:
    \begin{align*}
     \{0,1,2,\dots,m^*n-2,m^*n-1\}
    \end{align*}
    (for each $v$). In particular, $\Sym^n r \tensor \Ind_{G_M}^{G_F}
    \theta$ is regular.
  \item Suppose $r':\Gal(\overline{F}/F)\to \GL_2(\overline{\Z}_l)$ is
    a continuous representation with Hodge-Tate weights $\{0,1\}$ at
    $v$ for each place $v$ of $F$ above $l$; then $\Sym^{n-1} r'
    \tensor \Ind_{G_M}^{G_F} \theta'$ also has the following
    Hodge-Tate weights at $v$:
    \begin{align*}
     \{0,1,2,\dots,m^*n-2,m^*n-1\}
    \end{align*}
    (for each $v$). In particular, $\Sym^n r' \tensor \Ind_{G_M}^{G_F}
    \theta'$ is regular.
  
  \end{enumerate}
\end{lem}
\begin{proof}
 {\sl Step 1: Finding a suitable field $M$.} We claim that there
 exists a surjective character $\chi: \Gal(\overline{F}/F)\to
 \mu_{m^*}$ (where $\mu_{m^*}$ is the group of $m^*$-th
 roots of unity in $\Qbar^{\times}$) such that
 \begin{itemize}
 \item $\chi$ is unramified at all places of $F$ above $l$.
 \item $\chi(\Frob_v)=1$ for all $v\in S_{\ord}$.
 \item $\chi(\Frob_v)=-1$ for all $v\in S_{\ssg}$.
 \item $\chi(c_v)=-1$ for each infinite place $v$ (where $c_v$ denotes
   a complex conjugation at $v$).
 \item $\overline{F}^{\ker \chi}$ is linearly disjoint from $F^{(\mathrm{bad})}$ over $F$.
 \end{itemize}
 We construct the character $\chi$ as follows. First, we find using
 weak approximation a totally negative element $\alpha\in F^\times$
 which is a $v$-adic unit for each $v|l$, and which is congruent to a
 quadratic residue mod each $v\in S_{\ord}$ and a quadratic
 non-residue mod each $v\in S_{\ssg}$. Let $\chi_0$ be the quadratic 
 character associated to the extension we get by adjoining the square
 root of this element. Then:
\begin{itemize}
 \item $\chi_0$ is unramified at all places of $F$ above $l$.
 \item $\chi_0(\Frob_v)=1$ for all $v\in S_{\ord}$.
 \item $\chi_0(\Frob_v)=-1$ for all $v\in S_{\ssg}$.
 \item $\chi_0(c_v)=-1$ for each infinite place $v$ (where $c_v$ denotes
   a complex conjugation at $v$).
\end{itemize}
Now choose (for example by Theorem 6 of chapter 10 of \cite{MR2467155}) a cyclic totally real extension $M_1/\Q$ of degree $m^*$
such that:
\begin{itemize}
\item $M_1/\Q$ is unramified at all the rational primes where
  $\overline{F}^{\ker \chi_0}F^{(\mathrm{bad})}/\Q$ is ramified.
\item $l$ splits completely in $M_1$.
\end{itemize}
Since $\overline{F}^{\ker \chi_0}F^{(\mathrm{bad})}/\Q$ and $M_1/\Q$ ramify at disjoint sets of
primes, they are linearly disjoint, and we can find a rational prime
$p$ which splits completely in $F^{(\mathrm{bad})}\overline{F}^{\ker
  \chi_0}$ but such that $\Frob_p$ generates $\Gal(M_1/\Q)$. Since
$M_1/\Q$ is cyclic, we may pick an isomorphism between $\Gal(M_1/\Q)$
and $\mu_{m^*}$, and we can think of $M_1$ as determining a character
$\chi_1:G_\Q \to \mu_{m^*}$ such that:
\begin{itemize}
\item $\chi_1$ is trivial on $G_{\Ql}$.
\item $\chi_1$ is trivial on complex conjugation.
\item $\chi_1(\Frob_p) = \zeta_{m^*}$, a primitive $m^*$th root of
  unity.
\end{itemize}
Then, set $\chi=(\chi_1|_{G_F})\chi_0$. Note that this maps onto
$\mu_{m^*}$, even when we restrict to $G_{F^{(\mathrm{bad})}}$ (since
$p$ splits completely in $F^{(\mathrm{bad})}$ and if $\wp$ is a place
of $F^{(\mathrm{bad})}$ over $p$, we have $\chi_0(\Frob_\wp)=1$ while
$\chi_1(\Frob_\wp)=\zeta_{m^*}$).
The remaining properties
are clear.

 Having shown $\chi$ exists, we set $M=\overline{F}^{\ker \chi}$; note
 that this is a CM field, and a cyclic extension of $F$ of degree $m^*$. Write $\tau$
 for a generator of $\Gal(M/F)$. Write $M^+$ for the maximal totally
 real subfield of $M$.

 \medskip{\sl Step 2: Defining certain sequences of numbers}. We now will now define
certain sequences of numbers. The reason for introducing them is that the
characters $\theta$ and $\theta'$ will be engineered such that their Hodge-Tate
numbers at the primes above $l$ in $M$ will be given according to the sequences
we are about to construct, and many of the formulae we use in defining 
$\theta$ and $\theta'$ will require explicit mention of these numbers. Thus it will
be helpful to have introduced notation for them.

In particular, for each prime $v$ of $F$ above $l$, we
define $m^*$-tuples of integers $(h_{v,0},\dots,h_{v,m^*-1})$ and
$(h'_{v,0},\dots,h'_{v,m^*-1})$ by putting:
 \begin{align*}
    (h_{v,0},\dots,h_{v,m^*-1})
        =(&a_{v}(n-1),1+a_{v}(n-1),2+a_{v}(n-1),\dots,b_{v}-1+a_{v}(n-1), \\&\quad
             b_{v}n+a_{v}(n-1),b_{v}n+1+a_{v}(n-1),\dots,b_{v}n+(b_v-1)+a_{v}(n-1), \\&\quad\quad
              \dots, \\&\quad\quad\quad
                (m^*-b_{v})n+a_{v}(n-1),(m^*-b_{v})n+1+a_{v}(n-1),\\
                &\hspace{4cm}\dots,(m^*-b_{v})n+b_{v}-1+a_{v}(n-1))
 \\
    (h'_{v,0},\dots,h'_{v,m^*-1})=(&0,n,
                                        2n,\dots
                                             ,(m^*-1)n).
 \end{align*}
 We note that, so defined, $h$ and $h'$ satisfy, for each $i$:
 \begin{align}
   h_{v,i}+h_{v,m^*-i-1} &= (m^*-b_{v})n+b_{v}-1+2a_{v}(n-1)=m^*n-1+(1-n)w    \label{eq:hs-work-out}\\
   h'_{v,i}+h'_{v,m^*-i-1} &= (m^*-1)n    \label{eq:hprimess-work-out}
 \end{align}
 (The characters $\theta$ and $\theta'$ will be engineered to have these
 Hodge-Tate numbers at the primes above $l$ in $M$.)

 \medskip{\sl Step 3: An auxiliary prime $q$}. Choose a rational prime $q$ such that
\begin{itemize}
\item no prime of $T$ lies above $q$,
\item $q \ne l$,\item $q$ splits completely in $M$, 
\item $q$ is unramified in $F^{(\mathrm{bad})}$, and
\item $q-1>2n$.
\end{itemize}
Also choose a prime $\fq$ of $F$ above $q$, and a prime $\fQ$ of M
above $\fq$. 

\medskip{\sl Step 4: Defining certain algebraic characters
  $\phi,\phi': \A_{M}^\times\to (M')^\times$}.  For each prime $v$ of $F$
   above $l$, let us choose a
prime $w_v$ of $M$ above $v$. We now have a convenient notation for
all the primes above $v$; if $v\in S_{\ord}$ there are $m^*$ of them,
$\tau^j w_v$ for $j=0,\dots, m^*-1$; and if $v\in S_{\ssg}$ there are
$m^*/2$ of them, $\tau^j w_v$ for $j=0,\dots, (m^*/2)-1$.  Also,
choose $\iota_v$ to be an embedding $M\to \Qbar_l$ attached to the
prime $w_v$ (in case $v\in S_{\ord}$ there is only one choice; in case
$v\in S_{\ssg}$ there are two). 

We are now forced into a slight notational ugliness. Write $\tilde{M}$
for the Galois closure of $M$ over $\Q$.  (Thus $\Gal(\tilde{M}/\Q)$
is in bijection with embeddings $\tilde{M}\to \Qbar$.) Let us fix
$\iota^*$, an embedding of $\tilde{M}$ into $\Qbar_l$, and write $v^*$ for the prime of $M$ below this.\footnote{The choice
  of this $\iota^*$ will affect the choice of the algebraic characters
  $\phi,\phi'$ below, but will be cancelled out---at least concerning
  the properties we care about---when we pass to the $l$-adic
  characters $\theta,\theta'$ below.}  Given any embedding $\iota'$ of $M$
into $\Qbar_l$, we can choose an element $\sigma_{\iota'}$ in
$\Gal(\tilde{M}/\Q)$ such that $\iota'= \iota^*\circ\sigma_{\iota'}$.

 We claim that there exists an extension $M'$ of $M$, and a character
 $\phi:\A_{M}^\times\to (M')^\times$ with open kernel such that:
 \begin{itemize}
 \item For $\alpha\in M^\times$,
   \[
    \phi(\alpha) = \prod_{v\in S_{\ord}\cup S_{\ssg}} 
    \prod^{(m^*/2)-1}_{j=0} (\sigma_{\iota_v\circ\tau^{-j}}(\alpha))^{h_{v,j}}
         (\sigma_{\iota_v\circ\tau^{-j-(m^*/2)}}(\alpha))^{h_{v,m^*-1-j}}
   \]
\item For $\alpha\in(\A_{M^+})^\times$, we have 
  $$\phi(\alpha)=(\prod_{v\!\not\,|\,\infty}
  |\alpha_v|\prod_{v|\infty}\sgn_v(\alpha_v)\delta_{M/M^+}(
  \Art_{M^+}(\alpha)))^{-m^*n+1-(1-n)w},$$
  where $\delta_{M/M^+}$ is the quadratic character of $G_{M^+}$
  associated to $M$.
 (Note that, in the right hand side, we
  really think of $\alpha$ as an element of $\A_{M^+}$, not
  just as an element of $\A_{M}$ which happens to lie in
  $\A_{M^+}$; so for instance $v$ runs over places of $M^+$,
  and the local norms are appropriately normalized to reflect us
  thinking of them as places of $M^+$.)
\item $\phi$ is unramified at $l$. \end{itemize}
This is an immediate consequence of Lemma 2.2 of \cite{hsbt}; we must
simply verify that the conditions in the bullet points are compatible;
the only difficult part is comparing the first and second, where
equation \ref{eq:hs-work-out} gives us what we need\footnote{We also
  use the fact that, if we fix a complex embedding $\iota_\C$ of $M$,
  $\iota_\C\circ\sigma_{\iota_v\circ\tau^j}$ will run through all
  other complex embeddings as $v$ runs through primes above $l$ and
  $j$ runs from $0$ to $m^*-1$, as may be seen by taking a field
  isomorphism $\C\cong \Qbar_l$}.

Similarly, we construct a character $\phi':(\A_{M})^\times\to
(M')^\times$ (enlarging $M'$ if necessary) with open kernel such that:
 \begin{itemize}
 \item For $\alpha\in M^\times$,
   \[
    \phi'(\alpha) = \prod_{v\in S_{\ord}\cup S_{\ssg}} 
    \prod^{(m^*/2)-1}_{j=0} (\sigma_{\iota_v\circ\tau^{-j}}(\alpha))^{h'_{v,j}}
         (\sigma_{\iota_v\circ\tau^{-j-(m^*/2)}}(\alpha))^{h'_{v,m^*-1-j}}
   \]
\item For $\alpha\in(\A_{M^+})^\times$, we have 
  $$\phi'(\alpha)=(\prod_{v\!\not\,|\,\infty}
  |\alpha_v| \prod_{v | \infty} \sgn_v(\alpha_v) \delta_{M/M^+}(\Art_{M^+}(\alpha)))
^{-(m^*-1)n}.$$ (Again, we think of $\alpha$ in the right
  hand side as a bona fide member of $\A_{M^+}$.)
\item $\phi'$ is unramified at $l$.                                     \item $q|\#\phi'(\mathcal{O}^\times_{M,\fQ})$, but $\phi'$ is
  unramified at primes above $\fq$ other than $\fQ$ and $\fQ^c$
\end{itemize}
Again, this follows from Lemma 2.2 of \cite{hsbt}, now using equation
\ref{eq:hprimess-work-out}.

 \medskip{\sl Step 5: Defining the characters
   $\theta,\theta':\Gal(\overline{M}/M)\to
 \overline{\Z}_l^\times$}. Write $\tilde{M}'$ for the Galois
 closure of $M'$ over $\Q$, and extend $\iota^*:\tilde{M}\to\Qbar_l$
 to an embedding $\iota^{**}:\tilde{M}'\to\Qbar_l$.  Define $l$-adic
 characters $\theta_0,\theta':\Gal(\overline{M}/M)\to
 \overline{\Z}_l^\times$ by:
 \begin{align*}
   \theta_0(\Art \alpha) &= \iota^{**}(\phi(\alpha)) \prod_{v|l}
   \prod^{(m^*/2)-1}_{j=0} (\iota_v\circ \tau^{-j})(\alpha_{\tau^j
     w_v})^{-h_{v,j}} (\iota_v\circ
   \tau^{-j-(m^*/2)})(\alpha_{\tau^{(m^*/2)+j} w_v})^{-h_{v,m^*-1-j}}
   \\ \theta'(\Art \alpha) &= \iota^{**}(\phi'(\alpha)) \prod_{v|l}
   \prod^{(m^*/2)-1}_{j=0} (\iota_v\circ \tau^{-j})(\alpha_{\tau^j w_v})^{-h'_{v,j}}
   (\iota_v\circ \tau^{-j-(m^*/2)})(\alpha_{\tau^{(m^*/2)+j} w_v})^{-h'_{v,m^*-1-j}}
 \end{align*}
where $v$ runs over places of $F$ dividing $l$. (It is easy to check that the expressions on the right hand sides are
 unaffected when $\alpha$ is multiplied by an element of $M^\times$.)
 Observe then that they enjoy the following properties:
 \begin{itemize}
 \item $\theta' \circ V_{M/M^+} = (\eps_l\delta_{M/M^+})^{-(m^*-1)n}$ where $V_{M/M^+}$
     is the transfer map $G_{M^+}^{\ab} \rightarrow G_{M}^{\ab}$. In
     particular, $\theta'\theta'^c=\eps_l^{-(m^*-1)n}$.
 \item $\theta_0 \circ V_{M/M^+} = (\eps_l\delta_{M/M^+})^{-m^*n+1-(1-n)w}$
   and hence , $\theta_0\theta_0^c=\eps_l^{-m^*n+1-(1-n)w}$.
                    \item For $v\in S_{\ord}$ and $0\le j\le (m^*/2)-1$, the Hodge-Tate weight of
  $\theta_0|_{G_{M_{\tau^jw_v}}}$ is $h_{v,j}$, and the  Hodge-Tate weight of
  $\theta_0|_{G_{M_{\tau^{j+m^*/2}w_v}}}$ is $h_{v,m^*-1-j}$.
\item For $v\in S_{\ssg}$ and $0\le j\le (m^*/2)-1$, the Hodge-Tate
  weights of   $\theta_0|_{G_{M_{\tau^jw_v}}}$ 
    are $h_{v,j}$ and  $h_{v,m^*-1-j}$.
\item For $v\in S_{\ord}$ and $0\le j\le (m^*/2)-1$, the Hodge-Tate weight of
  $\theta'|_{G_{M_{\tau^jw_v}}}$ is $h'_{v,j}$, and the  Hodge-Tate weight of
  $\theta'|_{G_{M_{\tau^{j+m^*/2}w_v}}}$ is $h'_{v,m^*-1-j}$.
\item For $v\in S_{\ssg}$ and $0\le j\le (m^*/2)-1$, the Hodge-Tate
  weights of   $\theta'|_{G_{M_{\tau^jw_v}}}$
    are $h'_{v,j}$ and $h'_{v,m^*-1-j}$.
 \item $q|\#\theta'(I_\fQ)$, but $\theta'$ is unramified at all primes
   above $\fq$ except $\fQ,\fQ^c$.
 \end{itemize}
We now define $\theta=\theta_0(\tilde{\theta'}/
\tilde{\theta_0})$---where $\tilde{\theta_0}$ (resp $\tilde{\theta'}$)
denotes the Teichmuller lift of the reduction mod $l$ of $\theta_0$
(resp $\theta'$)---and observe that:
 \begin{itemize}
 \item $\theta$ (mod $l$) = $\theta'$ (mod $l$).
 \item $\theta\theta^c=\eps_l^{-m^*n+1-(1-n)w}\wtilde^{-(w-1)(n-1)}$.
\end{itemize}

\medskip{\sl Step 6: Properties of $\Ind_{G_M}^{G_F}\theta$ and $\Ind_{G_M}^{G_F}\theta'$.} 
We begin by addressing point 4. We define a pairing on $\Ind_{G_M}^{G_F}\theta$ 
by the formula
$$\langle \lambda,\lambda'\rangle = \sum_{\sigma\in\Gal(\overline{M}/M)\backslash 
                      \Gal(\overline{M}/F)} 
               \eps_l(\sigma)^{m^*n-1+(1-n)w} \wtilde(\sigma)^{(w-1)(n-1)}
                \lambda(\sigma)\lambda'(c\sigma)$$
where $c$ is any complex conjugation. One easily checks that this is
well defined and perfect, and that the properties (a) and (b) hold.

We can address point 5 in a similar manner, defining:
$$\langle \lambda,\lambda'\rangle = \sum_{\sigma\in\Gal(\overline{M}/M)\backslash 
                      \Gal(\overline{M}/F)} 
               \eps_l(\sigma)^{(m^*-1)n}
                \lambda(\sigma)\lambda'(c\sigma)$$
and checking the required properties.

Next, we address point 6. We will use the following notation: if $S,T$
are multisets of integers, we will write $S\oplus T$ for the `union
with multiplicity' of $S$ and $T$ (so that
$\{1\}\oplus\{1\}=\{1,1\}$), and $S\otimes T$ for the convolution of
$S$ and $T$ (i.e. $\{s+t|s\in S, t\in T\}$ with appropriate
multiplicities). Finally, for $r:G_F\to \GL_k(\overline{\Z}_l)$ a de
Rham Galois representation and $v$ a prime above $l$, we will write
$\HT_v(r)$ for the multiset of Hodge-Tate numbers of $r$ at the place
$v$. (Note that $l$ splits completely, so this is well defined.)

Now, supposing $r:\Gal(\overline{F}/F)\to \GL_2(\overline{\Z}_l)$ to
be a continuous representation with $\HT_v(r)=\{-a_v,b_v-a_v\}$ for
each place $v$ of $F$ above $l$, we can calculate $\HT_v(\Sym^{n-1} r
\tensor \Ind_{G_M}^{G_F} \theta)$ and show it has the required value;
see Figure \ref{fig:crazy-ht-no-calc}.
\begin{figure}
\begin{align*}
\HT_v(&\Sym^{n-1} r \tensor \Ind_{G_M}^{G_F} \theta) =\\
=&\HT_v(\Sym^{n-1} r) \otimes \HT_v(\Ind_{G_M}^{G_F} \theta)\\
 =&\{-(n-1)a_v,-(n-2)a_v+(b_v-a_v),\dots,(n-1)(b_v-a_v)\} \\
 &\quad\quad\otimes \{h_{v,0},\dots,h_{v,m^*-1}\}\\
 =&\{0,b_v,\dots,(n-1)b_v\} \otimes (\{-(n-1)a_v\}\otimes \{h_{v,0},\dots,h_{v,m^*-1}\})\\
 =&\{0,b_v,\dots,(n-1)b_v\} \\
 &\quad\quad\otimes \big(\{0,1,2,\dots,b_{v}-1\} 
                \oplus \{nb_{v},nb_{v}+1,\dots,nb_{v}+(b_v-1)\} \oplus\\
                &\quad\quad\quad\quad\dots\oplus
                \{(m^*-b_{v})n,(m^*-b_{v})n+1,\dots,(m^*-b_{v})n+b_{v}-1\}\big)\\
 =&\{0,b_v,\dots,(n-1)b_v\} \otimes \{0,1,2,\dots,b_{v}-1\}\\
  &\quad\oplus \{0,b_v,\dots,(n-1)b_v\}\otimes \{nb_{v},nb_{v}+1,\dots,nb_{v}+(b_v-1)\}\\
  &\quad\oplus \dots\\
  &\quad\oplus \{0,b_v,\dots,(n-1)b_v\}\\
   &\hspace{2cm}\otimes \{(m^*-b_{v})n,(m^*-b_{v})n+1,\dots,(m^*-b_{v})n+b_{v}-1\}\\
 =&\{0,1,\dots,nb_v-1\} \oplus\{nb_v,nb_v+1,\dots,2nb_v-1\} \\
 &\quad\oplus \dots \oplus \{n(m^*-b_v),n(m^*-b_v)+1,\dots,nm^*-1\}\\
 =&\{0,1,\dots,m^*n-1\}
\end{align*}
\caption{Computation of $\HT_v(\Sym^{n-1} r \tensor \Ind_{G_M}^{G_F} \theta)$\label{fig:crazy-ht-no-calc}}
\end{figure}

Next we address point 7, in a similar manner. Supposing
$r:\Gal(\overline{F}/F)\to \GL_2(\overline{\Z}_l)$ to be a continuous
representation with $\HT_v(r)=\{0,1\}$ for each place $v$ of $F$ above
$l$, we can calculate:
\begin{align*}
\HT_v(\Sym^{n-1} r \tensor \Ind_{G_M}^{G_F} \theta) &=\HT_v(\Sym^{n-1} r) \otimes HT_v(\Ind_{G_M}^{G_F} \theta) \\
  &=\{0,1,\dots, n-1\} \otimes \{h'_{v,0},\dots,h'_{v,m^*-1}\}\\
  &=\{0,1,\dots, n-1\} \otimes \{0,n,2n,\dots,(m^*-1)n\}\\
  &=\{0,1,\dots,m^*n-1\}
\end{align*}

Next, we address point 2. Let $v\in S_{\ssg}$. The assumption that
$L'$ contains $L$ means that the representations $\rho_{\chi}$ and
$\rho_{\chi'}$ are both $\GL_{nm^*}(\mc{O})$-conjugate to direct sums of characters, and the other
assumptions on $L'$ ensure that these characters are all
crystalline. The Hodge-Tate weights of these characters with respect
to any embedding $i:L'\into\Qlbar$ are determined by the restriction
of $i$ to $L$, so we may think of each character as having two
Hodge-Tate weights in the obvious way. For both $\rho_\chi$ and
$\rho_{\chi'}$, the set of ordered pairs of Hodge-Tate weights,
running over all the characters, is exactly the set of ordered pairs
of non-negative integers with sum $nm^*-1$ (this follows from the
calculations establishing points 6 and 7). Since two crystalline
characters with the same Hodge-Tate weights must differ by an
unramified twist, the result follows from Corollary \ref{cor: sums of
  unramified twists of chars tilde}.

 \medskip{\sl Step 7: Establishing the big image/avoid $\zeta_l$ properties}. All that
 remains is to prove the big image and avoiding $\zeta_l$ properties; that is, point (3). 
 We will just show the stated properties concerning
$\Ind_{G_M}^{G_F}\theta'$; the statement for $\Ind_{G_M}^{G_F}\theta$
then follow since $\theta$ and $\theta'$ are congruent.

Let $r$ be a
continuous $l$-adic Galois representation with $m^*$-big image, such that
the following properties hold:
\begin{itemize}
\item $r$ is ramified only at primes of $T$ and above $l$, and 
\item we have 
$\bar{F}^{\ker \bar{r}}(\zeta_l)\subset F^{\mathrm(bad)}$.
\end{itemize} We 
may now apply Lemma \ref{lem:big image etc}.
Applying part 2 of that Lemma will give that 
$(r\otimes\Ind_{G_M}^{G_F}\theta')|_{G_{F(\zeta_l)}}$ has big image,
(the first part of point (3) to be proved) and applying part 1 will give the fact
that we avoid $\zeta_l$ (the second part of point (3)).
 All that remains is to check the hypotheses
of Lemma \ref{lem:big image etc}.

The fact that $M$ is linearly disjoint from $\overline{F}^{\ker \rbar}(\zeta_l)$
(common to both parts) comes from the fact that 
$\overline{F}^{\ker \rbar}(\zeta_l)\subset F^{(\mathrm{bad})}$ and $M$ was
chosen to be linearly disjoint from $F^{(\mathrm{bad})}$. 

We turn now to the particular hypotheses of the second part. That $r|_{G_{F(\zeta_l)}}$ has $m^*$-big
image is by assumption.
The properties we require of $\fq$ follow directly from the bullet
points established in Step 3, the properties of $r$ just above,
and the first and last bullet points 
(concerning $\theta'\theta'{}^c$ and $\#\theta'(I_\fQ)$ respectively) 
in the list of properties of $\theta'$
given immediately after $\theta'$ is introduced in step 5.
The fact that $(\thbarpr)(\thbarpr)^c$ can be extended to $G_F$ comes
from the fact that it is a power of the cyclotomic character. 
 \end{proof}

Finally, we will prove that, when we have applied this lemma, it is in fact possible
to strengthen point 2 a little.
\begin{lem}\label{lem: choose a different basis of ind chi} Suppose that we are in the situation of Lemma \ref{lem:
    existence of chars with long list of properties}, and suppose that
$v,L,L',\chi,\chi', \sigma,K,\bigO$ are as in point (2) of that Lemma. Using $\sigma$-standard
bases, we can consider $\Ind_{G_L}^{G_{F_v}}\chi$ as a representation
$r_\chi:G_{F_v}\to \GL_2(\bigO)$, and do the same for $r_{\chi'}$.

Suppose further that
$r,r':G_{F_v}\to\GL_2(\bigO)$ are Galois representations, and that
there is a matrix
$A\in\GL_2(\bigO)$ such that $r=A r_\chi A^{-1}$, $r'=A r_{\chi'} A^{-1}$.
Let 
  \begin{align*}
  \rho_r&=(\Sym^{n-1}r)|_{G_{L'}}\otimes
  (\Ind_{G_{M}}^{G_{F}}\theta)|_{G_{L'}},\\ \text{and }
  \rho_{r'}&=(\Sym^{n-1}r')|_{G_{L'}}\otimes
  (\Ind_{G_{M}}^{G_{F}}\theta')|_{G_{L'}}.
  \end{align*}
We have a given basis of $r$, from which we inherit a basis on $\Sym^{n-1}r$
using Definition \ref{defn:inherited basis}; we have the $\tau-$standard
basis on $\Ind_{G_{M}}^{G_{F}}\theta$; and thus we inherit a natural basis on 
$\rho_r$ and can consider it as a representation
into $\GL_{nm^*}(\bigO)$. The same is true of $\rho_{r'}$.

Then $\rho_r$ and $\rho_{r'}$ are congruent, and moreover $\rho_r\sim\rho_{r'}$ in the
sense of Definition \ref{defn:tilde}.
\end{lem}
\begin{proof} The matrix $A$ gives rise to a matrix $A_n:=\Sym^{n-1} A$,
and (abbreviating $\Sym^{n-1}r$ as $r_n$, $\Sym^{n-1}r'$ as $r'_n$, 
$\Sym^{n-1}r_\chi$ as $r_{\chi,n}$ and $\Sym^{n-1}r_{\chi'}$ as $r_{\chi',n}$), 
$r_n=A_n r_{\chi,n} A_{n}^{-1}$, $r'_n=A_n r_{\chi',n} A_{n}^{-1}$.

Then we can define an element $B$ of $\GL_{nm^*}(\bigO)$ by putting 
$B:=A_n\otimes\mathrm{id}$, and see that
$\rho_r = B \rho_\chi B^{-1}$ (where $\rho_\chi$ is as defined in point (2)
of lemma \ref{lem: existence of chars with long list of properties}); similarly
$\rho_{r'} = B \rho_{\chi'} B^{-1}$.

Since point (2) of Lemma \ref{lem:
    existence of chars with long list of properties} tells us $\rhobar_\chi= \rhobar_{\chi'}$,
we have: 
$$\rhobar_{r}= \overline{B \rho_\chi B^{-1}}=\overline{ B \rho_{\chi'} B^{-1}}=\rhobar_{r'}.$$
Moreover, since $\rho_\chi\sim \rho_{\chi'}$ (again from point (2)), 
we can deduce that$$\rho_{r}=B \rho_\chi B^{-1} \sim B \rho_{\chi'} B^{-1}=\rho_{r'}$$ (since conjugation
by $B$ defines a natural isomorphism  between the lifting problems for
$\bar{\rho}_\chi$ and for $\overline{B\rho_\chi B^{-1}}$, and hence a natural isomorphism between
all the relevant universal lifting rings), as required.
\end{proof}

\section{Twisting and untwisting}\label{sec: twisting and untwisting}\subsection{Twisting}In this section we establish some basic
results about automorphic induction and Galois representations, which
are presumably well-known but for which we lack a reference. If $K$ is
a number field, we say that an automorphic representation $\pi$ of
$\GL_n(\A_K)$ is regular if $\pi_v$ is regular for all $v|\infty$, in
the sense of section 7 of \cite{BLGHT}. We caution the reader that while
``regular algebraic'' (which is also defined in section 7 of \cite{BLGHT}) implies ``regular'', the two notions are not equivalent.

\begin{lem}\label{lem: base change}
  Suppose that $L/K$ is a cyclic extension of number fields. Let
  $\kappa$ be a generator of $\Gal(L/K)^\vee$. Let $\pi$
  be a cuspidal automorphic representation of $\GL_n(\A_K)$, and
  suppose that $\pi\not\cong\pi\otimes(\kappa^i\circ\Art_K\circ\det)$
  for any $1\le i\le [L:K]-1$. Then
  there is a cuspidal automorphic representation $\Pi$ of $\GL_n(\A_L)$ such
  that for all places $w$ of $L$ lying over a place $v$ of $K$ we have \[\rec(\Pi_w)=\rec(\pi_v)|_{W_{L_w}}.\]
\end{lem}
\begin{proof}
  By induction on $[L:K]$ we may reduce to the case that $L/K$ is cyclic of
  prime degree. The result then follows from Theorems 3.4.2 and 3.5.1
  of \cite{MR1007299}, together with Lemma VII.2.6 of \cite{MR1876802}
  and the main result of \cite{MR672475}.
\end{proof}

We will write $BC_{L/K}(\pi)$ for $\Pi$.

\begin{lem}\label{lem: automorphic induction} Suppose that $L/K$ is a cyclic extension of number fields
  of degree $m$. Let $\pi$
  be a regular cuspidal automorphic representation of
  $\GL_n(\A_L)$. Let $\sigma$ be a generator of $\Gal(L/K)$. Assume that $\pi\not\cong\pi^{\sigma^i}$ for any $1\le i\le
  m-1$. Suppose further that $\Ind_{L_\infty}^{K_\infty}\pi_\infty$
  (the local automorphic induction) is
  regular. Then
  there is a regular cuspidal automorphic representation $\Pi$ of $\GL_{mn}(\A_K)$ such
  that for all places $v$ of $K$ we
  have
  \begin{equation}
    \label{eq:autoinductionlocalfactor}\rec(\Pi_v)= \oplus_{w|v}\Ind_{W_{L_w}}^{W_{K_v}}\rec(\pi_w)    
  \end{equation}(the
  sum being over places $w$ of $L$ dividing $v$).
  \end{lem}
  \begin{proof}
   The case that $m$ is prime follows from Theorem 3.5.1 and Lemma
    3.6.4 of \cite{MR1007299}, together with Lemma VII.2.6 of \cite{MR1876802}
  and the main result of \cite{MR672475} (the assumption that
  $\Ind_{L_\infty}^{K_\infty}\pi_\infty$ is regular is of course
  equivalent to the statement that $\Pi$ is regular). For the general case we use
  induction. Suppose that $L\supset L_2\supset L_1\supset K$ with
  $L_2/L_1$ cyclic of prime degree, and suppose that we have found a regular
  cuspidal automorphic representation $\Pi_{L_2}$ of
  $\GL_{[L:L_2]n}(\A_{L_2})$ satisfying the analogue of ~(\ref{eq:autoinductionlocalfactor}). The result will follow for $L_1$ provided we know that
  $\Pi_{L_2}\not\cong\Pi_{L_2}^{\sigma^{[L_1:K]i}}$ for any $1\le i\le
  [L_2:L_1]-1$; but if this fails to hold then it is easy to see
  that  \[\Ind_{L_\infty}^{K_\infty}\pi_\infty=\Ind_{(L_2)_\infty}^{K_\infty}(\Ind_{L_\infty}^{(L_2)_\infty}\pi_\infty)=\Ind_{(L_2)_\infty}^{K_\infty}(\Pi_{L_2})_\infty\]cannot
  be regular, a contradiction.
  \end{proof}

We will write $\Ind_{L}^{K}\pi$ for $\Pi$.

Let $F$ be a totally real field and let $M$ be an imaginary CM field which is a
cyclic Galois extension of $F$ of degree $m$. Fix
$\iota:\Qlbar\isoto\C$. Let $\pi$ be an RAESDC automorphic
representation of $\GL_n(\A_F)$, and let $\chi$ be an algebraic
character of $M^\times\backslash \A_M^\times$, chosen so that the Galois representation
$\Ind_{G_M}^{G_F}r_{l,\iota}(\chi)$ is essentially self-dual. Then the Galois
representation \[r_{l,\iota}(\pi)\otimes
\Ind_{G_M}^{G_F}r_{l,\iota}(\chi):G_F\to\GL_{nm}(\Qlbar)\] is also essentially-self dual. We have the following result.
\begin{prop}\label{prop: tensor product with induction is still automorphic}
  Assume
  that \[\pi_\infty\boxtimes\Ind_{M_\infty}^{F_\infty}\chi_\infty\] is
  regular; equivalently,  $r_{l,\iota}(\pi)\otimes
\Ind_{G_M}^{G_F}r_{l,\iota}(\chi)$ is
  regular. Assume also that if  $\kappa$ is a generator of
  $\Gal(M/F)^\vee$, then $\pi\not\cong\pi\otimes(\kappa^i\circ\Art_K\circ\det)$
  for any $1\le i\le [M:F]-1$. Then the representation $r_{l,\iota}(\pi)\otimes
\Ind_{G_M}^{G_F}r_{l,\iota}(\chi)$ is
  automorphic. More precisely, there is an RAESDC automorphic
  representation $\Pi$ of $\GL_{nm}(\A_F)$ with $r_{l,\iota}(\Pi)\cong
 r_{l,\iota}(\pi)\otimes
\Ind_{G_M}^{G_F}r_{l,\iota}(\chi)$. In fact, for every place $v$
  of $F$, we
  have \[\rec(\Pi_v|\cdot|^{(1-mn)/2})=\rec(\pi_v|\cdot|^{(1-n)/2})\otimes(\oplus_{w|v}\Ind_{W_{M_w}}^{W_{F_v}}\rec(\chi_w))\](the
  sum being over places $w$ of $M$ dividing $v$).
\end{prop}
\begin{proof}
  It is enough to prove that there is a regular cuspidal automorphic
  representation $\Pi$ of $\GL_{mn}(\A_F)$ satisfying the final
  assertion ($\Pi$ is then algebraic by the conditions at the infinite
  places, and is automatically essentially self-dual by the
  strong multiplicity one theorem and the conditions at the finite places, and thus has a Galois
  representation $r_{l,\iota}(\Pi)$ associated to it, which satisfies
  the required condition by the Tchebotarev density theorem).

  By Lemma \ref{lem: base change} and Lemma \ref{lem: automorphic
    induction} we have a regular cuspidal automorphic
  representation \[\Pi:=(\Ind_{M}^{F}(BC_{M/F}(\pi|\cdot|^{(1-n)/2})\otimes\chi))|\cdot|^{(mn-1)/2}\]
  (note that $BC_{M/F}(\pi|\cdot|^{(1-n)/2})\otimes\chi$ satisfies the
  hypotheses of Lemma \ref{lem: automorphic induction} by the
  assumption that
  $\pi_\infty\boxtimes\Ind_{M_\infty}^{F_\infty}\chi_\infty$ is
  regular).  By definition this choice of $\Pi$
  satisfies \[\rec(\Pi_v|\cdot|^{(1-mn)/2})=\oplus_{w|v}\Ind_{W_{M_w}}^{W_{F_v}}(\rec(\pi_v|\cdot|^{(1-n)/2})|_{W_{M_w}}\otimes\rec(\chi_w))\]
  for each place $v$ of $F$, and the result follows.
\end{proof}

\subsection{Untwisting}In this
section we explain a kind of converse to Proposition \ref{prop: tensor
  product with induction is still automorphic}, following an idea of
Harris (\cite{harristrick}, although our exposition is extremely
similar to that found in the proof of Theorems 7.5 and 7.6 of \cite{BLGHT}). 

Suppose that $F$ is a totally real field and that $M$ is an imaginary
CM field which is a cyclic extension of $F$ of degree $m$. Suppose
that $\theta$ is an algebraic character of
$M^\times\backslash\A_M^\times$ and that $\Pi$ is a RAESDC
representation of $\GL_{mn}(\A_F)$ for some $n$. Let
$\iota:\Qlbar\isoto\C$.

\begin{prop}\label{prop: untwisting}Assume that there is a continuous irreducible
  representation $r:G_F\to\GL_n(\Qlbar)$ such that  $r|_{G_M}$ is irreducible and \[r_{l,\iota}(\Pi)\cong
  r\otimes\Ind_{G_M}^{G_F}r_{l,\iota}(\theta).\] Then $r$ is automorphic.
  
\end{prop}
\begin{proof}
  Let $\sigma$ denote a generator of $\Gal(M/F)$, and $\kappa$ a
  generator of $\Gal(M/F)^\vee$. Then we have
  \begin{align*}
    r_{l,\iota}(\Pi\otimes(\kappa\circ\Art_F\circ\det))&=
    r_{l,\iota}(\Pi)\otimes r_{l,\iota}(\kappa\circ\Art_F)\\ &\cong
    r\otimes(
    r_{l,\iota}(\kappa\circ\Art_F)\otimes\Ind_{G_M}^{G_F}r_{l,\iota}(\theta))
    \\ & \cong
    r\otimes\Ind_{G_M}^{G_F}(r_{l,\iota}(\kappa\circ\Art_F)|_{G_M}\otimes
    r_{l,\iota}(\theta))\\ & \cong
    r\otimes\Ind_{G_M}^{G_F}r_{l,\iota}(\theta)\\ &\cong r_{l,\iota}(\Pi),
  \end{align*}  so that
  $\Pi\otimes(\kappa\circ\Art_F\circ\det)\cong\Pi$. 

We claim that for each intermediate field $M\supset N\supset F$ there
is a regular cuspidal automorphic representation $\Pi_N$ of $\GL_{n[M:N]}$
such that \[\Pi_N\otimes(\kappa\circ\Art_N\circ\det)\cong\Pi_N\] and
$BC_{N/F}(\Pi)$ is equivalent to
\[\Pi_N\boxplus\Pi_N^\sigma\boxplus\dots\boxplus\Pi_N^{\sigma^{[N:F]-1}}\]
in the sense that for all places $v$ of $N$, the base change from
$F_{v|_F}$ to $N_v$ of $\Pi_{v|_F}$
is \[\Pi_{N,v}\boxplus\Pi_{N,v}^\sigma\boxplus\dots\boxplus\Pi_{N,v}^{\sigma^{[N:F]-1}}.\]
We prove this claim by induction on $[N:F]$. Suppose that $M\supset
M_2\supset M_1\supset F$ with $M_2/M_1$ cyclic of prime degree, and
that we have already proved the result for
$M_1$. Since  \[\Pi_{M_1}\otimes(\kappa\circ\Art_{M_1}\circ\det)\cong\Pi_{M_1}\]we
see from Theorems 3.4.2 and 3.5.1 of \cite{MR1007299} (together with Lemma VII.2.6 of \cite{MR1876802}
  and the main result of \cite{MR672475}) that there is a cuspidal
  automorphic representation  $\Pi_{M_2}$ of $\GL_{n[M:{M_2}]}$
such that $BC_{{M_2}/F}(\Pi)$ is equivalent to
\[\Pi_{M_2}\boxplus\Pi_{M_2}^\sigma\boxplus\dots\boxplus\Pi_{M_2}^{\sigma^{[{M_2}:F]-1}}.\]
Since $\Pi$ is regular, $\Pi_{M_2}$ is regular. The representation
$\Pi_{M_2}\otimes(\kappa\circ\Art_{M_2}\circ\det)$ satisfies the same
properties (because $\Pi\otimes(\kappa\circ\Art_F\circ\det)\cong\Pi$),
so we see (by strong multiplicity one for isobaric representations) that we must
have \[\Pi_{M_2}\otimes(\kappa\circ\Art_{M_2}\circ\det)\cong\Pi_{M_2}^{\sigma^i}\]
for some $0\le i\le [M:M_2]-1$. If $i>0$
then \[\Pi_{M_2}\boxplus\Pi_{M_2}^\sigma\boxplus\dots\boxplus\Pi_{M_2}^{\sigma^{[{M_2}:F]-1}}\]cannot
be regular (note that of course $\kappa$ is a character of finite
order), a contradiction, so in fact we must have
$i=0$. Thus \[\Pi_{M_2}\otimes(\kappa\circ\Art_{M_2}\circ\det)\cong\Pi_{M_2}\]and
the claim follows.

Let $\pi:=\Pi_M$. Note that the representations $\pi^{\sigma^i}$ for
$0\le i\le m-1$ are pairwise non-isomorphic (because $\Pi$ is
regular). Note also that $\pi\otimes |\det|^{(n-nm)/2}$ is regular
algebraic (again, because $\Pi$ is regular algebraic).

Since $\Pi$ is RAESDC, there is an algebraic character $\chi$ of
$F^\times\backslash \A_F^\times$ such that $\Pi^\vee\cong
\Pi\otimes(\chi\circ\det)$. It follows (by strong multiplicity one for
isobaric representations) that for some $0\le i\le m-1$
we have \[\pi^\vee\cong \pi^{\sigma^i}\otimes(\chi\circ
N_{M/F}\circ\det).\] Then we have
\begin{align*}
  \pi&\cong(\pi^\vee)^\vee\\ &\cong  (\pi^{\sigma^i}\otimes(\chi\circ
N_{M/F}\circ\det))^\vee \\  &\cong  (\pi^\vee)^{\sigma^i}\otimes(\chi^{-1}\circ
N_{M/F}\circ\det))\\ &\cong  (\pi^{\sigma^i}\otimes(\chi\circ
N_{M/F}\circ\det))^{\sigma^i}\otimes(\chi^{-1}\circ
N_{M/F}\circ\det)) \\ &\cong \pi^{\sigma^{2i}}
\end{align*}
so that either $i=0$ or $i=m/2$. We wish to rule out the former
possibility. Assume for the sake of contradiction that \[\pi^\vee\cong \pi\otimes(\chi\circ
N_{M/F}\circ\det).\]  Since $F$ is totally real, there is an
integer $w$ such that $\chi|\cdot|^{-w}$ has finite image. Then
$\pi|\det|^{w/2}$ has unitary central character, so is itself
unitary. Since $\pi\otimes |\det|^{(n-nm)/2}$ is regular
algebraic, we see that for places $v|\infty$ of $M$ the conditions of
Lemma 7.1 of \cite{BLGHT} are satisfied for $\pi_v|\det|_v^{w/2}$, so
that
\begin{align*}
  \pi_v\boxplus\pi_v^c&\cong \pi_v\boxplus
  ((\pi_v\otimes|\det|^{w/2})^c\otimes(|\cdot|^{-w/2}\circ\det))\\
  &\cong  \pi_v\boxplus
  ((\pi_v\otimes|\det|^{w/2})^\vee\otimes(|\cdot|^{-w/2}\circ\det)) \\  &\cong  \pi_v\boxplus
  (\pi_v^\vee\otimes(|\cdot|^{-w}\circ\det)) \\  &\cong  \pi_v\boxplus
  ( \pi_v\otimes(\chi|\cdot|^{-w}\circ
N_{M/F}\circ\det))
\end{align*}
which contradicts the regularity of $\Pi_{v|_F}$. Thus we have $i=m/2$,
so that \[\pi^\vee\cong \pi^c\otimes(\chi\circ
N_{M/F}\circ\det).\] Thus $\pi\otimes |\det|^{(n-nm)/2}$ is a RAECSDC
representation, so that we have a Galois representation
$r_{l,\iota}(\pi\otimes |\det|^{(n-nm)/2})$. The condition that $BC_{M/F}(\Pi)$ is equivalent to
\[\pi\boxplus\pi^\sigma\boxplus\dots\boxplus\pi^{\sigma^{m-1}}\]
translates to the fact that
  \[r_{l,\iota}(\Pi)|_{G_M}\cong
  r_{l,\iota}(\pi\otimes|\det|^{(n-nm)/2})\oplus\dots\oplus
  r_{l,\iota}(\pi\otimes|\det|^{(n-nm)/2})^{\sigma^{m-1}}.\] By
  hypothesis, we also have  \[r_{l,\iota}(\Pi)|_{G_M} \cong
  (r|_{G_M}\otimes
  r_{l,\iota}(\theta))\oplus\dots\oplus(r|_{G_M}\otimes r_{l,\iota}(\theta)^{\sigma^{m-1}}).\]Since $r|_{G_M}$ is irreducible, there must be an $i$ such
that \[r|_{G_M}\cong r_{l,\iota}(\pi\otimes|\det|^{(n-nm)/2})\otimes
r_{l,\iota}(\theta)^{\sigma^{-i}},\] so that $r|_{G_M}$ is
automorphic. The result now follows from Lemmas 1.4 and 1.5 of \cite{BLGHT}.
\end{proof}

\section{Potential automorphy in weight 0}\label{sec: pot auto weight 0}\subsection{}
In our final arguments, we will need to rely on certain potential
automorphy results for symmetric powers of Galois representations with
Hodge-Tate numbers $\{0,1\}$ at every place. The fact that such
results are immediately available given the techniques in the
literature is well known to the experts, but because we were unable to
locate a reference which states the precise results we will need, we will give
very brief derivations of them here. We hope that providing a written
reference for these results may prove useful to other authors. 

\begin{lem}\label{Big image for symmetric powers of GL(2)} Suppose that $l>2(n-1)m+1$ is a prime; that $k$ is an algebraic
extension of $\F_l$; that $k'\subseteq k$ is a finite field and that 
$H\subset \GL_n(k)$. Suppose that 
$$k^\times \Sym^{n-1} \GL_2(k') \supseteq H \supseteq \Sym^{n-1} \SL_2(k')$$
Then $H$ is $m$-big. 
\end{lem}
\begin{proof} This may be deduced from Lemma 7.3 of \cite{BLGHT} as Corollary 2.5.4
of \cite{cht} is deduced from Lemma 2.5.2 of \emph{loc.~cit}.
\end{proof}

\begin{lem}\label{Avoid zeta-l if contain SL2}
  Suppose that $m$ is a positive integer, that $k$ is an algebraic
  extension of $\F_l$, that $k'\subseteq k$ is a finite field and that that $F$ is a totally real field.
  Suppose that $l$ is a prime such that $[F(\zeta_l):F]>2m$, and that $\rbar: \Gal(\bar{F}/F)\to \GL_2(k)$ has
$$k^\times  \GL_2(k') \supseteq \rbar(G_F) \supseteq  \SL_2(k').$$
Then, for any $n$,
$[\bar{F}^{\ker\ad\Sym^{n-1}\rbar}(\zeta_l):\bar{F}^{\ker\ad\Sym^{n-1}\rbar}]>m.$
In particular, the conclusion holds if $l$ is unramified in $F$ and $l>2m+1$.
\end{lem}
\begin{proof}
We have \[\PSL_2(k')\subset
  (\ad\Sym^{n-1}\rbar)(G_F)\subset\PGL_2(k'),\]
  $\PGL_2(k')/\PSL_2(k')$ has order $2$, and $\PSL_2(k')$ is simple. Thus
  the intersection of  $\bar{F}^{\ad\Sym^{n-1}\rbar}$
  and $F(\zeta_l)$ has degree at most $2$ over $F$. Since
  $[F(\zeta_l):F]>2m$, the result follows (for the final part,
  note that if $l$ is unramified in $F$ then $[F(\zeta_l):F]=l-1$).
\end{proof}

\begin{prop}\label{prop:potmod for small coeffs} Let 
$F$ be a totally real field, 
$\Favoid$ a finite extension of $F$,
$\cL$ a finite set of primes of $F$,
$n$ a positive integer, 
and $l>4(n-1)+1$ a prime which is unramified in $F$. Suppose that
$$r:G_F\to \GL_2(\Z_l)$$
is a continuous representation which is unramified at all but finitely
many primes, and enjoys the
following properties:
\begin{enumerate}
\item $\det r=\eps_l^{-1}$.
\item $\rbar(G_F)\supset\SL_2(\F_l)$.
\item For each prime $v|l$ of $F$, $r|_{G_{F_v}}$ is crystalline for all $\tau\in\Hom(F,\Qbar_l)$, and we have
\[
\dim_{\Qbar_l} \gr^i(r \otimes_{\tau,F_v} \BdR)^{G_{F_v}} = 
\begin{cases}
     1&  i=0,1\\ 
     0&  \text{(otherwise)}
  \end{cases}
\]
\item $\cL$ does not contain primes above $l$, and $r$ is unramified
  at places in $\cL$.
\end{enumerate}
Then there is a Galois totally real extension $F''$ of $F$, linearly
disjoint from $\Favoid$ over $F$ such that $(\Sym^{n-1} r)|_{G_{F''}}$ is
automorphic of weight 0 and level prime to $l$, and no prime 
of $\cL\cup\{v|v\text{ a prime of $F$}, v|l\}$ ramifies in $F''$.
\end{prop}
\begin{proof} We will deduce this result from Theorem 7.5 of \cite{BLGHT} 
in the same way that Theorem 3.2 of \cite{hsbt} is deduced from Theorem
3.1 of \cite{hsbt}, following the proof of Theorem 3.2 of \cite{hsbt}
very closely. Our argument is in fact simpler, because we need no longer
maintain a Steinberg place, and so when we apply the 
Theorem of Moret-Bailly (Proposition 2.1 of
\cite{hsbt}), we do not impose any local condition at auxiliary primes $v_q$,
$v_{q'}$, unlike in \cite{hsbt}. Thus all arguments earlier on in
  the proof concerning these primes become unnecessary.

In particular, we copy the argument up to the
application of Proposition 2.1 of \cite{hsbt} almost verbatim, with
only the following changes:
\begin{itemize}
\item The character $\eps_l \det r$ is trivial, so the field
  $F_1$ is simply $F$, and in particular, $F_1$ is linearly disjoint
  from $\Favoid$ over $F$.      \item Rather than choosing $l'>C(n),n$, we take $l'>4(n-1)+1$, $l'>5$ and $l'\not\in \cL$.
\item When we choose $E_1$ in the application of Proposition 2.1 of
  \cite{hsbt}, we choose
  it to have good reduction at primes of $\cL$.\item As mentioned above, we no longer impose any local condition at
the primes
 $v_q,v_{q'}$
 since we no longer need the
  conclusion, after the application of Proposition 2.1 of \cite{hsbt}, that $E$ has split
  multiplicative reduction at auxiliary primes $v_q,v_{q'}$.
\item In the application of Proposition 2.1 of \cite{hsbt}, we may also assume that
  the field $F'$ is linearly disjoint from $\Favoid$ (this is easy, as
  Proposition 2.1 of \cite{hsbt} allows us to avoid any fixed field.) We also impose a local
  condition to ask that $F$ not ramify at primes of $\cL$; we must then check
  that we can find some elliptic curve whose mod $ll'$ cohomology agrees
  with $H^1(E_1\times\Fbar,\Z/l'\Z)\times \rbar$ when restricted to inertia
  at primes of $\cL$. $E_1$ itself fills this role, all the representations involved
  being trivial in this case.
\end{itemize}
As in \cite{hsbt}, the result of all this is the construction of an elliptic curve
$E$ over a finite Galois extension $F'$ of $F$, which together have the
following properties:
\begin{itemize}
\item $F'$ is linearly disjoint from $\Favoid\bar{F}^{\ker (\rbar\times \rbar_1)}$, where
$\rbar_1$ is the Galois representation $H^1(E_1\times \bar{F},\Z/l'\Z)$ attached to 
a certain elliptic curve $E_1$ chosen earlier in the portion of the
proof we followed from \cite{hsbt}, as mentioned above.
\item In particular, since $E_1$ was chosen such that $G_F\twoheadrightarrow
\Aut H^1(E_1\times \bar{F},\Z/l'\Z)$, we also have $G_{F'}\twoheadrightarrow
\Aut H^1(E_1\times \bar{F},\Z/l'\Z)$. \hfill(A1)
\item $F'$ is totally real.
\item All primes above $ll'$ and the primes of $\cL$ are unramified in $F'$.
\item $E$ has good reduction at all places above $ll'$ and the primes of $\cL$.
\item $H^1(E\times \bar{F},\Z/l\Z)\cong r|_{G_{F'}}$
\item $E$ has good ordinary reduction at $l'$ (note that $l'$ is
  unramified in $F$, that $E$ has good reduction at $l'$, and that its $l'$ torsion is
  isomorphic to the $l'$ torsion of $E_1$, which was chosen to be
  ordinary at $l'$).\hfill (A2)
\end{itemize}

We then apply Theorem 7.5 of \cite{BLGHT}\footnote{In fact, we need a slight
extension incorporating a set of primes $\cL$ where we do not want our
extension to ramify, and a fixed field from which we want our extension to be
linearly disjoint. Adding this is a routine modification of the proof of
Theorem 6.3 of \cite{BLGHT}, in the same spirit as the modifications above
of the arguments of \cite{hsbt}.}
 to the Galois representation
$\Sym^{n-1} H^1(E\times \bar{F},\Z_{l'})$, which we will call $r_n'$ for brevity.
 Let us check the conditions of this
theorem in turn. We have two unnumbered conditions: first that 
$r_n'$ is almost everywhere unramified
(which is obvious, as it comes from geometry), and second that there is
a perfect, Galois equivariant, pairing on $r_n'$
towards $\Z_l(\mu)$, for some character $\mu$ (such a pairing on 
$H^1(E\times \Fbar,\Z_{l'})$ is furnished
by Poincar\'e duality, going towards $\Z_l(-1)$;
thus we get such a pairing on $r_n'$ with $\mu=\eps_l^{1-n}$). 
Now we address the numbered
conditions:
\begin{description}
\item[1] {\sl The sign of $\mu$ on complex conjugations agrees with the parity
 of the pairing.} Poincare duality on $H^1(E\times \Fbar,\Z_{l'})$ will be an alternating pairing,
 so the pairing on $r_n'$ will satisfy $\langle u,v\rangle=(-1)^{1-n}\langle v,u\rangle$
 and $\mu=\eps_l^{1-n}$ is $(-1)^{1-n}$ on complex conjugations.
\item[2] {\sl We have $[\bar{F}^{\ker\ad\rbar_n'}(\zeta_l):\bar{F}^{\ker\ad\rbar_n'}]>2$.} 
We can use the fact that the Galois 
 representation on the $l'$ torsion of $E$ is surjective (since it agrees with the action on the $l'$ torsion of $E_1$, and using point (A1) above), and Lemma \ref{Avoid zeta-l if contain SL2} (using the fact that $l'>5$).
\item[3] {\sl We have that $\rbar_n'(G_{F(\zeta_l)})$ is 2-big.} Again we use the fact that the Galois 
 representation on the $l'$ torsion of $E$ is surjective, this time together with the simplicity of $\PSL_2(k)$, and Lemma \ref{Big image for symmetric powers of GL(2)} (using $l'>4(n-1)+1$).
\item[4] {\sl $r_n'$ is ordinary of regular weight.} This is immediate given point (A2) above.
\end{description}
We immediately deduce that there exists an extension $F''/F'$, with
$\Sym^{n-1} H^1(E\times \Fbar,\Z_{l'})|_{G_{F''}}$
automorphic of weight 0 and level prime to $l$, and such that
\begin{itemize}
\item $F''/F$ is Galois.
\item Primes above $l$ and $l'$ are unramified in $F''$, as are the primes of $\cL$.
\item $F''$ is linearly disjoint from $\Favoid F'\bar{F}^{\ker\ad\rbar}\bar{F}^{\ker\ad\rbar_n'}$ over $F'$
(and hence linearly disjoint from $\Favoid\bar{F}^{\ker\ad\rbar}$ over $F$).
\hfill (A3)
\end{itemize}
We now apply Theorem 5.2 of \cite{guerberoff2009modularity} 
to the Galois representation $\Sym^{n-1} r|_{G_{F''}}$, which we call $r_n$ for brevity.
Let us check the conditions of this theorem in turn:
\begin{description}
\item[1] {\sl $r_n|_{G_{F''}}$ is essentially self dual.} $r$ acquires a perfect symplectic 
 pairing with cyclotomic multiplier from the determinant; from this 
 $r_n|_{G_{F''}}$ inherits a perfect pairing with the desired properties.
\item[2] {\sl $r_n|_{G_{F''}}$ is almost everywhere unramified.} This is trivial, since we assume
 $r$ has this property.
\item[3] {\sl $r_n|_{G_{F''}}$ is crystalline at each place above $l$.} This too is trivial,
 since we assume $r$ has this property (condition 3 of the theorem being proved).
\item[4] {\sl $r_n|_{G_{F''}}$ is regular with Hodge-Tate weights 
 lying in the Fontaine-Laffaille range.} It follows from condition 3 of the theorem being proved
 that the Hodge-Tate weights of $r_n|_{G_{F''}}$ are $\{0,1,\dots,n-1\}$; this 
 suffices, as $l>n$.
\item[5] {\sl $\bar{F}^{\ker\ad \rbar_n}$ does not contain $\zeta_l$.} This follows
 by condition 2 of the theorem
 being proved, the fact $l>3$, and point A3 above, using
 Lemma \ref{Avoid zeta-l if contain SL2}.
\item[6] {\sl $\rbar_n|_{G_{F''(\zeta_l)}}$ has big image.} This is true by condition 2 of the theorem
 being proved, the simplicity of $PSL_2(\F_l)$, and point A3 above, as we can see by applying Lemma \ref{Big image for symmetric powers of GL(2)}.
\item[7] {\sl $\rbar_n$ is automorphic, with the right weight.} We know $r$
  and $H^1(E\times \Fbar,\Z_{l})$ are congruent, and manifestly
  $\Sym^{n-1} H^1(E\times \Fbar,\Z_{l})$
  has Hodge-Tate weights $\{0,1,\dots,n-1\}$, which is indeed the required weight as we saw in verifying
  hypothesis 4 above.
\end{description}
  We conclude that $\Sym^{n-1}r$ is automorphic over $F''$ of weight 0
  and level prime to $l$, as required.
\end{proof}

\section{Hilbert modular forms}\label{sec: main thm}\subsection{}Let
$\pi$ be a regular algebraic cuspidal automorphic representation of $\GL_2(\A_F)$, where $F$ is a
totally real field. Assume that $\pi$ is not of CM type. Let the
weight of $\pi$ be $\lambda\in(\Z^2_+)^{\Hom(F,\C)}$. Let $m^*$ be the
least common multiple of $2$ and the values
$\lambda_{v,1}-\lambda_{v,2}+1$. Let $n>1$ be a positive
integer. Choose a prime $l>5$, and fix an isomorphism
$\iota:\Qlbar\isoto \C$. We choose $l$ so that:
\begin{itemize}
\item $l$ splits
completely in $F$.
\item $\pi_v$ is unramified for all $v|l$.
\item The residual representation
$\rbar_{l,\iota}(\pi):G_F\to\GL_2(\Flbar)$ has large image, in the
sense that there are finite fields $\F_l\subset k\subset k'$
with \[\SL_2(k)\subset \rbar_{l,\iota}(\pi)(G_F)\subset
k'^\times\GL_2(k).\] (Note that this is automatically the case for all
sufficiently large $l$ by Proposition 0.1 of \cite{MR2172950}.)
 \hfill(B1)
\item $l>2(n-1)m^*+1$.\hfill(B2)
\end{itemize}
Note that, as a consequence of points B1 and B2, the simplicity of $\PSL_2(\F_l)$, and 
Lemmas \ref{Big image for symmetric powers of GL(2)} and 
\ref{Avoid zeta-l if contain SL2}, it follows that:
\begin{itemize}
\item $\Sym^{n-1}r_{l,\iota}(\pi)$ has $m^*$-big image. \hfill (B3)
\item
  $[\bar{F}^{\ker
    \ad\Sym^{n-1}\rbar_{l,\iota}(\pi)}(\zeta_l):\bar{F}^{\ker \ad\Sym^{n-1}\rbar_{l,\iota}(\pi)}]>m^*$. 
  \hfill (B4)
\end{itemize}
Choose a solvable extension $F'/F$ of totally real fields such that
\begin{itemize}
\item $l$ splits completely in $F'$.
\item $F'$ is linearly disjoint from
  $\overline{F}^{\ker\rbar_{l,\iota}(\pi)}$ over $F$.
\item At each place $w$ of $F'$, $\pi_{F',w}$ is either unramified or
  an unramified twist of the Steinberg representation (here we let
  $\pi_{F'}$ denote the base change of $\pi$ to $F'$).
\item $[F':\Q]$ is even.
\end{itemize}
That such an extension exists follows exactly as in the proof of
Theorem 3.5.5 of \cite{kis04}. After a further quadratic base change
if necessary, we may also assume that
\begin{itemize}
\item $\pi_{F'}$ is ramified at an even number of places.
\end{itemize}

\begin{prop}\label{prop: construction of pi'}
  There is a cuspidal algebraic automorphic representation $\pi'$ of
  $\GL_2(\A_{F'})$ such that
  \begin{enumerate}
  \item $\pi'$ has weight 0.
  \item $\rbar_{l,\iota}(\pi')\cong\rbar_{l,\iota}(\pi)|_{G_{F'}}$.
  \item if $w\nmid l$ is a place of $F'$, then $\pi'_w$ is ramified
    if and only if $\pi_{F',w}$ is, in which case it is also an
    unramified twist of the Steinberg representation.
  \item $r_{l,\iota}(\pi')|_{G_{F'_v}}$ is potentially Barsotti-Tate for all
    $v|l$, and is ordinary if and only if
    $r_{l,\iota}(\pi)|_{G_{F'_v}}$ is ordinary.
  \end{enumerate}

\end{prop}
\begin{proof}
  Choose a quaternion algebra $B$ with centre $F'$ which is ramified
  at precisely the infinite places and the set $\Sigma$ of finite places at which
  $\pi_{F'}$ is ramified. 
 We will use Lemma 3.1.4 of \cite{kis04} and
  the Jacquet-Langlands correspondence to find $\pi'$. Let $v \in
  \Sigma$ and let $\rho$ be an irreducible
  representation of $B_{v}^{\times}$ corresponding to the irreducible
  admissible representation $\JL(\rho)$ of $\GL_2(F'_v)$ under the
  Jacquet-Langlands correspondence. We recall that
  $\rho^{\mc{O}_{B,v}^{\times}}$ is non-zero if and only if
  $\JL(\rho)$ is an unramified twist of the Steinberg
  representation. 
  We now introduce $l$-adic automorphic forms on $B^{\times}$. Let $K$ be a
  finite extension of $\Ql$ inside $\Qlbar$ with ring of integers
  $\mc{O}$ and residue field $k$, and assume that $K$ contains the
  images of all embeddings $F'\into\Qlbar$.
  Fix a maximal order $\mc{O}_{B}$ in $B$ and for each finite place $v \not
  \in \Sigma$ of
  $F'$ fix an isomorphism $i_v : \mc{O}_{B,v} \isoto
  M_2(\mc{O}_{F'_v})$.  Since $l$ splits completely in $F'$, we can
  and do identify embeddings $F'\into \Qlbar$ with primes of $F'$
  dividing $l$. For each place $v|l$ we let $\iota v$ denote the real place of $F'$
  corresponding to the embedding $\iota \circ v$. Similarly, if
  $\sigma : F'\into \bb{R}$ is an embedding we let $\iota^{-1} \sigma$
  denote the corresponding place of $F'$ dividing $l$.

  For each $\lambda' \in (\bb{Z}^{2}_+)^{\Hom(F',\Qlbar)}$ and $v|l$
  consider the algebraic representation
  \[ W_{\lambda'_v}:= \Sym^{\lambda'_{v,1}-\lambda'_{v,2}}\mc{O}^2
  \otimes_{\mc{O}} (\det)^{\lambda'_{v,2}} \] of
  $\GL_2(\mc{O})$. We consider it as a representation of
  $\mc{O}_{B,v}^{\times}$ via $\mc{O}_{B,v}^{\times} \stackrel{i_v}{\To}
  \GL_2(\mc{O}_{F'_v}) \stackrel{v}{\into} \GL_2(\mc{O})$.
 Let $W_{\lambda'}=\otimes_{v|l}
  W_{\lambda'_v}$ considered as a representation of
  $\GL_2(\mc{O}_{F',l})$.
  For each $v|l$ let $\tau_v$ denote a smooth
  representation of $\GL_2(\mc{O}_{F'_v})$ on a finite free
  $\mc{O}$-module $W_{\tau_v}$. Let $\tau$ denote the representation $
  \otimes_{v|l}\tau_v$ of
  $\GL_2(\mc{O}_{F',l})$ on $ W_{\tau} := \otimes_{v|l}W_{\tau_v}$. We
  let $W_{\lambda',\tau}=W_{\lambda'}\otimes_{\mc{O}}W_{\tau}$.
  Suppose that $\psi': (F')^{\times}\backslash
  (\bb{A}_{F'}^{\infty})^{\times}\rightarrow \mc{O}^{\times}$ is a
  continuous character so that for each prime $v|l$, the action of
  the centre $\mc{O}_{F'_v}^{\times}$ of $\mc{O}_{B,v}^{\times}$ on
  $W_{\lambda'_v}\otimes_{\mc{O}}W_{\tau_v}$ is given by
  $(\psi')^{-1}|_{\mc{O}_{F'_v}^{\times}}$. The existence of such a
  character implies that there exists an integer $w'$ such that $w' =
  \lambda'_{v,1}+\lambda'_{v,2} +1$ for each $v|l$.

  Let $U=\prod_v U_v \subset (B \otimes_{\bb{Q}} \bb{A}^{\infty})^{\times}$ be a
  compact open subgroup with $U_v \subset \mc{O}_{B,v}^{\times}$ for
  all $v$ and $U_v = \mc{O}_{B,v}^{\times}$ for $v |l$. We let
  $S_{\lambda',\tau,\psi'}(U,\mc{O})$ denote the space of functions
  \[ f : B^{\times}\backslash
  (B\otimes_{\bb{Q}}\bb{A}^{\infty})^{\times} \rightarrow
  W_{\lambda',\tau} \] with $f(gu)=(\lambda'\otimes\tau)(u_l)^{-1}f(g)$
  and $f(gz)=\psi'(z)f(g)$ for all $u \in U$, $z \in
  (\bb{A}_{F'}^{\infty})^{\times}$ and $g \in
  (B\otimes_{\bb{Q}}\bb{A}^{\infty})^{\times}$.
  Writing $U=U^l\times U_l$, we let
  $$ S_{\lambda',\tau,\psi'}(U_l,\mc{O})=\varinjlim_{U^l}S_{\lambda',\tau,\psi'}(U^l
  \times U_l,\mc{O})$$
  and we let
  $(B\otimes_{\bb{Q}}\bb{A}^{l,\infty})^{\times}$ act on this space by
  right translation.

  Let $\psi'_{\bb{C}}: (F')^{\times}\backslash \bb{A}_{F'}^{\times}
  \rightarrow \bb{C}^{\times}$ be the algebraic Hecke character
  defined by $$z \mapsto \mathbf{N}_{F'/\bb{Q}}(z_{\infty})^{1-w'}
  \cdot \iota\left( \mathbf{N}_{F'/\bb{Q}}(z_l)^{w'-1}
    \psi'(z^{\infty})\right).$$ 
Let $W_{\tau,\bb{C}}=W_{\tau}\otimes_{\mc{O},\iota}\bb{C}$.
 We have an isomorphism of $(B\otimes_{\bb{Q}}\bb{A}^{l,\infty})^{\times}$-modules
\begin{equation}
\label{eq:iso}
 S_{\lambda',\tau,\psi'}(U_l,\mc{O})\otimes_{\mc{O},\iota}\bb{C}
\isoto
\bigoplus_{\Pi}
\Hom_{\mc{O}_{B,l}^{\times}}(W_{\tau,\bb{C}}^{\vee},\Pi_l)\otimes
\Pi^{\infty,l} 
\end{equation} 
where the sum is over all automorphic
representations $\Pi$ of $(B\otimes_{\bb{Q}}\bb{A})^{\times}$ of
weight $\iota_* \lambda' := (\lambda'_{\iota^{-1}\sigma})_{\sigma} \in
(\bb{Z}^2_+)^{\Hom(F',\bb{C})}$ and central character $\psi'_{\bb{C}}$
(see for instance the proof of Lemma 1.3 of \cite{tay-fm2}).

  Let $U$ be as above and let $R$ denote a finite set of places of
  $F'$ containing all those places $v$ where $U_v \neq
  \mc{O}_{B,v}^{\times}$. Let  $\bb{T}^{\Sigma\cup R}$ denote the
  polynomial algebra $\mc{O}[T_{v},S_{v}]$ where $v$ runs over all
  places of $F'$ away from $l$, $R$ and $\Sigma$. For such $v$ we let
  $T_v$ and $S_v$ act on $S_{\lambda',\tau,\psi'}(U,\mc{O})$ via the
  double coset operators
\[ \left[ U i_v^{-1}\left(\begin{matrix} \varpi_v & 0 \cr 0 &
      1 \end{matrix} \right) U \right]
\mathrm{\ \ and \ \ }  \left[ U i_v^{-1} \left(\begin{matrix} \varpi_v
      & 0 \cr 0 & \varpi_v \end{matrix} \right) U \right]
\]
respectively, where $\varpi_v$ is a uniformizer in $\mc{O}_{F'_v}$.

Let $\wt{\pi}$ denote the automorphic representation of
$(B\otimes_{\bb{Q}}\bb{A})^{\times}$ of weight
$\lambda_{F'}:=(\lambda_{\sigma|_{F}})_{\sigma}\in (\bb{Z}^2_+)^{\Hom(F',\bb{C})}$ corresponding
to $\pi_{F'}$ under the Jacquet-Langlands correspondence. Let $\iota^*
\lambda = ((\lambda_{F'})_{\iota v})_v \in
(\bb{Z}^2_+)^{\Hom(F',\Qlbar)}$. Let $U= \prod_v U_v \subset
(B\otimes_{\bb{Q}}\bb{A}^{\infty})^{\times}$ be the compact open
subgroup with $U_v = \mc{O}_{B,v}^{\times}$ for all $v$. Then the
space $\wt{\pi}^U$ is non-zero. Let $\chi: G_{F'}^{\ab} \rightarrow
\Qlbar^{\times}$ denote the character $\epsilon \det
r_{l,\iota}(\pi)|_{G_{F'}}$ and let $\psi = \chi \circ \Art_{F'} :
\bb{A}_{F'}^{\times}/\overline{
  (F'_{\infty})^{\times}_{>0}(F')^{\times}} \rightarrow \Qlbar$. Note
that $\chi$ is totally even and hence we may regard $\psi$ as a
character of $(\bb{A}_{F'}^{\infty})^{\times}/(F')^{\times} \isoto \bb{A}_{F'}^{\times}/\overline{
  (F'_{\infty})^{\times}(F')^{\times}}$. Extending $K$ if
necessary, we can and do assume that $\psi$ is valued in $\mc{O}^{\times}$.
Further extending $K$ if necessary, choose a
$\bb{T}^{\Sigma\cup R}$-eigenform
$f$ in $S_{\lambda,1,\psi}(U,\mc{O})$ corresponding to an element of $(\wt{\pi}^{\infty})^{U}$
under the isomorphism \eqref{eq:iso}. The $\bb{T}^{\Sigma\cup R}$-eigenvalues on $f$ give rise to an $\mc{O}$-algebra homomorphism
$\bb{T}^{\Sigma\cup R}\rightarrow \mc{O}$ and reducing this modulo
$\mf{m}_{\mc{O}}$ gives rise to a maximal ideal $\mf{m}$ of $\bb{T}^{\Sigma\cup R}$.

Let
$\wt{\chi}:G_{F'}^{\ab}\rightarrow \mc{O}^{\times}$ denote the
Teichm\"uller lift of the reduction of $\chi$. Let $\psi' = \wt{\chi}
\circ \Art_{F'}$, which we can regard as a character
$(\bb{A}_{F'}^{\infty})^{\times}/(F')^{\times} \rightarrow
\mc{O}^{\times}$.
Let $v$ be a place of $F'$ dividing $l$.
\begin{itemize} 
\item If
$r_{l,\iota}(\pi)|_{G_{F'_v}}$ is ordinary and
$\iota^*\lambda_{v,1}\neq \iota^*\lambda_{v,2}$, let $\chi_1,\chi_2 :
\bb{F}_l^{\times} \rightarrow \Ql^{\times}$ be the characters given by $\chi_1(x)
= \wt{x}^{\iota^*\lambda_{v,1}}$ and $\chi_2(x)=\wt{x}^{\iota^*\lambda_{v,2}}$ where $\wt{x}$ denotes the Teichm\"uller lift of
$x$. Then let $\tau_v$ denote the representation
\[
I(\chi_1,\chi_2):=\Ind_{B(\bb{F}_l)}^{\GL_2(\bb{F}_l)}(\chi_1\otimes
\chi_2) \] of $\GL_2(\bb{F}_l)$ where $B$ is the Borel subgroup of
upper triangular matrices in $\GL_2$. 
\item If
$r_{l,\iota}(\pi)|_{G_{F'_v}}$ is ordinary and
$\iota^*\lambda_{v,1} = \iota^*\lambda_{v,2}$, let $\chi :
\bb{F}_l^{\times} \rightarrow \Ql^{\times}$ be the character
$\chi(x)=\wt{x}^{\iota^*\lambda_{v,2}}$. Let $\tau_v$ denote the
  representation $\chi \circ \det$ of $\GL_2(\bb{F}_l)$.
\item If
$r_{l,\iota}(\pi)|_{G_{F'_v}}$ is not ordinary, let $\chi :
\bb{F}_{l^2} \rightarrow \bb{Q}_{l^2}^{\times}$ be the character given
by $\chi(x)= \wt{x}^{\iota^*\lambda_{v,1}-\iota^*\lambda_{v,2}+2+(l+1)(\iota^*\lambda_{v,2}-1)}$. Let $\tau_v$ be the
$\Qlbar$-representation of $\GL_2(\bb{F}_l)$ denoted $\Theta(\chi)$ in
section 3 of \cite{MR1639612} (note that $\chi^l \neq \chi$ since
$0 < \iota^*\lambda_{v,1}-\iota^*\lambda_{v,2}+2 < l+1$  ). 
\end{itemize}
Extending $K$ if necessary, we can and
do fix a model for $\tau_v$ on a finite free $\mc{O}$-module
$W_{\tau_v}$. We also view $W_{\tau_v}$ as a smooth
$U_v=\mc{O}_{B,v}^{\times}$-module via $U_v \stackrel{i_v}{\To}
\GL_2(\mc{O}_{F'_v})\stackrel{v}{\To}\GL_2(\bb{Z}_l)\onto \GL_2(\bb{F}_l)$. 
By Lemma 3.1.1 of \cite{MR1639612}, $W_{\lambda_v}\otimes_{\mc{O}}k$ is
a Jordan-H\"older constituent of $W_{\tau_v}\otimes_{\mc{O}}k$.

It follows that $W_{\lambda}\otimes_{\mc{O}}k$ is a
Jordan-H\"older constituent of the $U_l$-module
$W_{\tau}\otimes_{\mc{O}}k$. Also, since $l>3$ and $l$ is unramified
in $F'$, $B^{\times}$ contains no elements of exact order $l$ and
hence the group $U$ satisfies hypothesis 3.1.2 of \cite{kis04} (with
$l$ replacing $p$).
We can therefore apply Lemma 3.1.4 of \cite{kis04} to deduce that $\mf{m}$ is
in the support of $S_{0,\tau,\psi'}(U,\mc{O})$.
Let $\Pi'$ denote the automorphic
representation of $(B\otimes_{\bb{Q}}\bb{A})^{\times}$ corresponding
to any non-zero $\bb{T}^{\Sigma\cup R}$-eigenform in
$S_{0,\tau,\psi'}(U,\mc{O})_{\mf{m}}\otimes_{\mc{O},\iota}\bb{C}$ under the isomorphism \eqref{eq:iso}.
Let $\pi'$ be the
automorphic representation of $\GL_2(\bb{A}_{F'})$ corresponding to
$\Pi'$ under the Jacquet-Langlands correspondence. Then $\pi'$ is
regular algebraic and of weight 0 by construction. The choice of
$\mf{m}$ ensures that property (2) of the theorem holds and hence that
$\pi'$ is cuspidal. Property (3) holds by the choice of $U$.

It remains to show that $\pi'$ satisfies property (4). Let $v$ be a
place of $F'$ dividing $l$ and suppose firstly that
$r_{l,\iota}(\pi)|_{G_{F'_v}}$ is non-ordinary.  Then by the choice of
$\tau_v$ and part (3) of Lemma 4.2.4 of \cite{MR1639612} we see that
$r_{l,\iota}(\pi')|_{G_{F'_v}}$ is potentially Barsotti-Tate and moreover 
\[ \WD(r_{l,\iota}(\pi')|_{G_{F'_v}})|_{I_{F'_v}} \cong
\wt{\omega}_2^{-\iota^*\lambda_{v,1}-l\iota^*\lambda_{v,2}+l-1} \oplus
\wt{\omega}_2^{-l\iota^* \lambda_{v,1}-\iota^*\lambda_{v,2}+1-l} \]
where $\omega_2$ is a fundamental character of niveau 2.
We deduce that $r_{l,\iota}(\pi')|_{G_{F'_v}}$ only becomes
Barsotti-Tate over a non-abelian extension of $G_{F'_v}$ and hence is non-ordinary.
Now suppose that $r_{l,\iota}(\pi)|_{G_{F'_v}}$ is ordinary. Then by the choice of
$\tau_v$ and parts (1) and (2) of Lemma 4.2.4 of \cite{MR1639612} we see that
$r_{l,\iota}(\pi')|_{G_{F'_v}}$ is potentially Barsotti-Tate and moreover 
\[ \WD(r_{l,\iota}(\pi')|_{G_{F'_v}})|_{I_{F'_v}} \cong
\wt{\omega}^{-\iota^*\lambda_{v,1}} \oplus \wt{\omega}^{-\iota^*
  \lambda_{v,2}} \] where $\omega$ is the mod $l$ cyclotomic
character. Since $\rbar_{l,\iota}(\pi)|_{G_{F'_v}}$ is reducible,
it follows from Theorem 6.11(3) of \cite{MR2137952} that
$r_{l,\iota}(\pi')|_{G_{F'_v}}$ is either decomposable (in which case
it is easy to see that it must be ordinary), or it corresponds to a
potentially crystalline representation as in Proposition 2.17 of
\cite{MR2137952}, with $v_l(x_1)=1$ or $v_l(x_2)=1$ (because if
neither of these hold, then by Theorem 6.11(3) of \cite{MR2137952} the
representation
$\rbar_{l,\iota}(\pi')|_{G_{F'_v}}=\rbar_{l,\iota}(\pi)|_{G_{F'_v}}$
would be irreducible, a contradiction). In either case the
representation is ordinary (for example by Lemma 2.2 of
\cite{BLGHT}).\end{proof} We would like to thank Richard Taylor for pointing out the following
lemma to us. 
\begin{lem}\label{lem: purity implies local-global compatibility}Let $F$ be a totally real field, and let $\pi$ be a RAESDC
  representation of $\GL_n(\A_F)$. Let $l$ be a prime number, and fix
  an isomorphism $\iota:\Qlbar\to\C$. Suppose that for some place
  $v\nmid l$ of $F$, the Galois
  representation $r_{l,\iota}(\pi)|_{G_{F_v}}$ is pure. Then we
  have \[\WD(r_{l,\iota}(\pi)|_{G_{F_v}})^{F-ss}=\iota^{-1}(\rec(\pi_v)\otimes|\Art_{F_v}^{-1}|_{F_v}^{(1-n)/2}),\]where
    $\WD(r_{l,\iota}(\pi)|_{G_{F_v}})$ denotes the Weil-Deligne
    representation associated to $r_{l,\iota}(\pi)|_{G_{F_v}}$.

\end{lem}
\begin{proof}
  By Theorem 1.1 of \cite{BLGHT}, the claimed equality holds on the Weil
  group (but we do not necessarily know that the monodromy is the same
  on each side). However, $\pi_{v}$ is generic (because $\pi$ is
  cuspidal) and $\rec^{-1}(\WD(r_{l,\iota}(\pi)|_{G_{F_v}})^{F-ss})$ is
  also generic (because it is tempered, by Lemma 1.4(3) of
  \cite{MR2276777}, and thus a subquotient of a unitary induction of a
  square-integrable representation of a Levi subgroup, by Theorem 2.3 of
  \cite{MR1721403}. Any such induction is irreducible (cf. page 72 of
  \cite{MR771672}), and the
  result follows from Theorem 9.7 of \cite{MR584084}.) The claimed equality
  follows (because a generic representation is determined by its
  supercuspidal support - this follows easily from the results of
  Zelevinsky recalled on page 36 of \cite{MR1876802}).
\end{proof}

\begin{thm}\label{potential automorphy of sym n for weight 0}
Continue using the setup at the beginning of this section, and let
$\pi'$ be as in Proposition \ref{prop: construction of pi'}. 
  Let $N/F'$ be a finite extension of number fields. There is a finite
  Galois extension of totally real fields $F''/F'$ such
  that
  \begin{enumerate}
\item $F''$ is linearly disjoint from
  $\overline{F}^{\ker\rbar_{l,\iota}(\pi')}(\zeta_l)N$ over
  $F'$.
\item There is a RAESDC automorphic representation $\pi'_n$ of $\GL_n(\A_{F''})$ of
  weight $0$ and level prime to $l$ such that $r_{l,\iota}(\pi'_n)\cong
  \Sym^{n-1}r_{l,\iota}(\pi')|_{G_{F''}}$.
  \end{enumerate}

\end{thm}
\begin{proof} 
  The central character $\omega_{\pi'}$ of $\pi'$ has finite order and
  is trivial at the infinite places, so
  we can choose a quadratic totally real extension $F_1$ of $F'$
  linearly disjoint from
  $N\overline{F}^{\ker\rbar_{l,\iota}(\pi')}(\zeta_l)$ (which we will
  henceforth call $\Favoid$) over $F'$, such that if
  $\pi'_{F_1}=BC_{F_1/F}(\pi')$, then $\omega_{\pi'_{F_1}}$ has a
  square root (note that the obstruction to taking a square root is in
  the 2-torsion of the Brauer group of $F'$). Say
  $\chi^2=\omega_{\pi'_{F_1}}$, and write
  $\pi''=\pi'_{F_1}\otimes(\chi^{-1}\circ\det)$. Making a further
  solvable base change (and keeping $F_1$ linearly disjoint from
  $\Favoid$ over $F'$), we may assume in addition that $\pi''$ has
  level prime to $l$ (that this is possible follows from local-global
  compatibility and Proposition \ref{prop: construction of
  pi'}(4)).

Then choose a rational prime $l'\neq l$ and an isomorphism $\iota':
\overline{\bb{Q}}_{l'}\isoto \bb{C}$ such that:
\begin{itemize}
\item $\pi''_v$ is unramified for all $v|l'$.
\item $l'$ is unramified in $F_1$.
\item $l'>4(n-1)+1$.
\item $l'$ splits completely in the field of coefficients of $\pi''$.
\item The residual representation
$\rbar_{l',\iota'}(\pi''):G_{F_1}\to\GL_2(\overline{\bb{F}}_{l'})$  
has large image, in the
sense that there are finite fields $\F_{l'}\subset k\subset k'$
with \[\SL_2(k)\subset \rbar_{l',\iota'}(\pi'')(G_{F_1})\subset
k'^\times\GL_2(k).\] (Note that this is automatically the case for all
sufficiently large $l'$ by Proposition 0.1 of \cite{MR2172950}.)
Coupled with the previous
point, this in fact means that:
\[\SL_2(\F_{l'})\subset \rbar_{l',\iota'}(\pi'')(G_{F_1})\subset \GL_2(\F_{l'}).\]
\end{itemize}

Since $\pi''$ has trivial central character, $\det
r_{l',\iota'}(\pi'')=\epsilon_l^{-1}$, and we can apply Proposition
\ref{prop:potmod for small coeffs} to find a Galois extension
$F_2/F_1$, linearly disjoint from $\Favoid$ over $F'$, such that $
\Sym^{n-1} r_{l',\iota'}(\pi'')$ is automorphic over $F_2$ of weight
0. That it is in fact automorphic of level prime to $l$ follows from
Lemma \ref{lem: purity implies local-global compatibility} (note that
$ \Sym^{n-1} r_{l',\iota'}(\pi'')|_{G_{F_2}}$ is pure, because
$r_{l',\iota'}(\pi'')$ is pure, for example by the main result of
\cite{MR2327298}, and $ \Sym^{n-1} r_{l',\iota'}(\pi'')|_{G_{F_2}}$ is
unramified at all places dividing $l$ by the choice of
$F_2$). Replacing $F_2$ by a further solvable base change, also
disjoint from $\Favoid$ over $F'$, if necessary, we may assume that
$\chi$ is unramified at all places of $F_2$ lying over $l$. We are
then done, taking $F''=F_2$ (because $ \Sym^{n-1}
r_{l',\iota'}(\pi')|_{G_{F_2}}=
r_{l',\iota'}(\chi)^{n-1}|_{G_{F_2}}\otimes \Sym^{n-1}
r_{l',\iota'}(\pi'')|_{G_{F_2}}$).
\end{proof}

\begin{thm}\label{thm: nth symmetric power of pi is potentially
    automorphic}Let $F$ be a totally real field, and let $\pi$ be a
  non-CM regular algebraic cuspidal automorphic
  representation of $\GL_2(\A_F)$. Then there is a prime $l$, an
  isomorphism $\Qlbar\isoto\C$, a finite Galois extension of
  totally real fields $F''/F$ and an RAESDC automorphic representation $\pi_n$ of
  $\GL_n(\A_{F''})$ such that $r_{l,\iota}(\pi_n)\cong
  \Sym^{n-1}r_{l,\iota}(\pi)|_{G_{F''}}$.
\end{thm}
\begin{proof}
  We continue to use the notation established above, and in particular
  we will fix $l$ and $\iota$ as above, and make use of $F'$ and
  $\pi'$. $F''$ will be as in the conclusion of Theorem \ref{potential
    automorphy of sym n for weight 0}, which we will apply with a
  particular choice of field $N$, to be determined below. We can and
  do assume that $r_{l,\iota}(\pi)$ and $r_{l,\iota}(\pi')$ both take
  values in $\GL_n(\mc{O})$ where $\mc{O}$ is the ring of integers in
  a finite extension $K$ of $\bb{Q}_l$ inside $\Qlbar$. Let $k$ denote
  the residue field of $K$. We can and do further assume that
  $r_{l,\iota}(\pi)|_{G_{F'}}$ and $r_{l,\iota}(\pi')$ are equal (as
  homomorphisms) when composed with the natural map
  $\GL_n(\mc{O})\rightarrow \GL_n(k)$. Write
  $\rbar_{l,\iota}(\pi)|_{G_{F'}}=\rbar_{l,\iota}(\pi')$ for this
  composition.

We write $r=\Sym^{n-1}r_{l,\iota}(\pi)|_{G_{F'}}$ and
  $r'=\Sym^{n-1}r_{l,\iota}(\pi')$, thought of as representations
  valued in $\GL_n(\bigO)$ via the bases specified in Definition \ref{defn:inherited basis}. We begin by 
  applying Lemma \ref{lem:
    existence of chars with long list of properties} in the following
  situation (where the field $F$ of Lemma \ref{lem: existence of chars
    with long list of properties} is $F'$):
  \begin{itemize}
  \item $S_{\ord}$ is the set of places of $F'$ dividing $l$ lying over a place for which
    $\pi$ is ordinary with respect to $\iota$.
  \item For each $v|l$, thought of as an embedding $F' \into \Qlbar$,
    $a_v=-\lambda_{\iota\circ v|_{F},2}$ and 
    $b_v=\lambda_{\iota\circ v|_{F},1}-\lambda_{\iota\circ v|_{F},2}+1$.
  \item $T$ is the set of places away from $l$ at which $\pi'$ is
    ramified.
      \item
    $(F')^{(\mathrm{bad})}=\overline{F}^{\ker \rbar}(\zeta_l)$.
  \end{itemize}
  We deduce that (after possibly extending $K$) there is a CM
  extension $M$ of $F'$ of degree $m^*$, and de Rham
  characters \[\theta,\theta':G_M\to \bigO^\times,\]satisfying various
  properties that we will now describe.  We can fix an element
  $\tilde{\tau}\in G_{F'}$ mapping to a generator $\tau\in\Gal(M/F')$,
  and we regard $\Ind_{G_M}^{G_{F'}}\theta$ and
  $\Ind_{G_M}^{G_{F'}}\theta'$ as representations valued in
  $\GL_{m^*}(\bigO)$ via their $\tilde{\tau}$-standard bases $\beta=\{
  e_0,\ldots,e_{m^* -1}\}$ and $\beta'=\{e'_0,\ldots,e'_{m^* -1}\}$
  respectively, in the sense of Definition \ref{standard bases}. Note
  that these two representations become equal when composed with the
  homomorphism $\GL_{m^*}(\mc{O}) \rightarrow \GL_{m^*}(k)$.  Then, by
  the conclusions of Lemma \ref{lem: existence of chars with long list
    of properties}, the following hold:
\begin{itemize}
\item $\bar{\theta}=\bar{\theta}'$.
\item $(r\otimes\Ind_{G_M}^{G_{F'}}\theta)|_{G_{F'(\zeta_l)}}$ has big
  image.
\item $\bar{F}^{(\ker \ad(\rbar\otimes\Ind_{G_M}^{G_{F'}}\overline{\theta}))}$ does not contain
  $\zeta_l$.
  \item $r\otimes\Ind_{G_M}^{G_{F'}}\theta$ and
    $r'\otimes\Ind_{G_M}^{G_{F'}}\theta'$ are both de Rham, and have
    the same Hodge-Tate weights at each place of $F'$ dividing $l$.
  \item $r'':=r\otimes\Ind_{G_M}^{G_{F'}}\theta$ is essentially self-dual;
    that is, there is a character $\chi:G_{F'}\to\Qlbar^\times$ with
    $\chi(c_v)$ independent of $v|\infty$ (where $c_v$ denotes a
    complex conjugation at $v$) such that $(r'')^\vee\cong
    r''\chi$.
  \item $\Ind_{G_M}^{G_{F'}}\theta'$ is essentially self-dual.
\end{itemize}
Applying Theorem \ref{potential automorphy of sym n for weight 0} with
$N=\bar{M}^{\ker\bar{\theta}}$, we find a totally real field $F''/F'$ with $r'|_{G_{F''}}$
automorphic of level prime to $l$. By Proposition \ref{prop: tensor product with induction
  is still automorphic}, the representation
$(r'\otimes\Ind_{G_M}^{G_{F'}}\theta')|_{G_{F''}}$ is automorphic. We now choose a
solvable extension $F^+/F''$ of totally real fields such that
\begin{itemize}
\item $F^+$ is linearly disjoint from $\bar{F}^{\ker\ad \rbar''}(\zeta_l)M$
  over $F'$.
\item $r''|_{G_{F^+}}=(r\otimes\Ind_{G_M}^{G_{F'}}\theta)|_{G_{F^+}}$ and
  $(r'\otimes\Ind_{G_M}^{G_{F'}}\theta')|_{G_{F^+}}$ are both
  crystalline at all places dividing $l$.
\item  $(r'\otimes\Ind_{G_M}^{G_{F'}}\theta')|_{G_{F^+}}$ is
  automorphic of level prime to $l$ (note that $r'|_{G_{F''}}$ is
  automorphic of level prime to $l$, so this is easily achieved by
  Lemma \ref{lem: base change} and Proposition \ref{prop: tensor product with induction is still automorphic}).
\item The extension $F^+M/F^+$ is unramified at all finite places, and
  is split completely at all places of $F^+$ lying over places in $T$.
\item $\theta|_{G_{F^+M}}$ and $\theta'|_{G_{F^+M}}$ are both
  unramified at all places not dividing $l$.
\item If $v|l$ is a place of $F^+$, then $F^+_v$ contains the
  unramified quadratic extension of $F'_v$, and $\rbar''|_{G_{F^+_v}}$ is trivial.
\end{itemize}

Write $\rho:=r''|_{G_{F^+}}=(r\otimes\Ind_{G_M}^{G_{F'}}\theta)|_{G_{F^+}}$,
 $\rho':=(r'\otimes\Ind_{G_M}^{G_{F'}}\theta')|_{G_{F^+}} : G_{F^+}
\rightarrow
\GL_{nm^*}(\mc{O})$, so that $\rho'$ is automorphic of level prime to $l$.
\begin{sublem}
  For each place $w|l$ of $F^+$, $\rho|_{G_{F^+_w}}\sim\rho'|_{G_{F^+_w}}$.
\end{sublem}
\begin{proof}If $w$ lies over a place of $S_{\ord}$, this follows from
  Lemma \ref{lem: crystalline ordinary lifting ring is irreducible
    when residually trivial}. Otherwise, $w$ lies over a place $v$ in
  $S_{\ssg}$. Let $L$ be the unramified quadratic extension of $F'_v$
  in $\overline{F}'_v$. Then
  $\rbar_{l,\iota}(\pi)|_{G_{F'_v}}\cong\Ind_{G_{L}}^{G_{F'_v}}\overline{\chi}$
  for some character $\overline{\chi}:G_L\to k^\times$ (see for
  example Th\'{e}or\`{e}me 3.2.1(1) of \cite{ber05}). We can and do
  (after possibly extending $K$)
  choose a crystalline lift $\chi:G_L\to\bigO^\times$ of
  $\overline{\chi}$ with Hodge-Tate weights $-a_v$ and $b_v-a_v$
  (recall that $l$ splits completely in $F'$). We can also (again,
  extending $K$ if necessary) choose a de
  Rham character $\chi': G_L\to\bigO^\times$ lifting $\overline{\chi}$
  with Hodge-Tate weights 0 and 1, which becomes crystalline over
  $F^+_w$ (we can do this by
  the assumption that $\overline{r}''|_{G_{F^+_w}}$ is trivial). 

Choose an element $\sigma\in
G_{F'_v}$ mapping to a generator of $\Gal(L/F'_v)$, and fix the $\sigma$-standard (in the sense
of Definition \ref{standard bases})
bases of 
$\Ind_{G_L}^{G_{F'_v}}\chi$ , $\Ind_{G_L}^{G_{F'_v}}\chi'$,
$\Sym^{n-1}\Ind_{G_L}^{G_{F'_v}}\chi$, 
and $\Sym^{n-1}\Ind_{G_L}^{G_{F'_v}}\chi'$.

Choose a matrix $A\in\GL_n(\bigO)$ with
$A(\Ind_{G_L}^{G_{F'_v}}\overline{\chi}) A^{-1}=
\rbar_{l,\iota}(\pi)|_{G_{F'_v}}$, and write
$r_{\chi}=A(\Ind_{G_L}^{G_{F'_v}}\chi) A^{-1}$. We have $r_\chi \sim
r_{l,\iota}(\pi)|_{G_{F'_v}}$, because both liftings are crystalline
of the same weight, $F_v$ is unramified, and the common weight is in
the Fontaine-Laffaille range (see e.g. Lemma 2.4.1 of
\cite{cht} which shows that the appropriate lifting ring is formally
smooth over $\mc{O}$). Then by the remarks following Definition
\ref{defn:tilde}, we have
$(r_\chi)|_{G_{F^+_w}}\sim
r_{l,\iota}(\pi)|_{G_{F^+_w}}$, and
$\Sym^{n-1}(r_\chi)|_{G_{F^+_w}}\sim
\Sym^{n-1}r_{l,\iota}(\pi)|_{G_{F^+_w}}$, so that (with the inherited bases)
  \[(\Sym^{n-1}(r_\chi)|_{G_{F^+_w}})\otimes (\Ind_{G_M}^{G_{F'}}\theta)|_{G_{F^+_w}}\sim
  (\Sym^{n-1}r_{l,\iota}(\pi)|_{G_{F^+_w}})\otimes
  (\Ind_{G_M}^{G_{F'}}\theta)|_{G_{F^+_w}}.\]

Similarly, write $r_{\chi'}=A(\Ind_{G_L}^{G_{F'_v}}\chi') A^{-1}$. Then $r_{\chi'}|_{G_{F^+_w}}\sim
  r_{l,\iota}(\pi')|_{G_{F^+_w}}$, because both representations are
  Barsotti-Tate and non-ordinary (see Proposition 2.3 of
  \cite{MR2280776} which shows that all non-ordinary points lie on the
  same component of the appropriate lifting ring). Then $\Sym^{n-1}(r_{\chi'})|_{G_{F^+_w}}\sim
  \Sym^{n-1}r_{l,\iota}(\pi')|_{G_{F^+_w}}$, and 
  \[(\Sym^{n-1}(r_{\chi'})|_{G_{F^+_w}})\otimes (\Ind_{G_M}^{G_{F'}}\theta')|_{G_{F^+_w}}\sim
  (\Sym^{n-1}r_{l,\iota}(\pi')|_{G_{F^+_w}})\otimes
  (\Ind_{G_M}^{G_{F'}}\theta')|_{G_{F^+_w}}.\]

  By Lemma \ref{lem: choose a different basis of ind chi}, we
  have \[(\Sym^{n-1}(r_{\chi'})|_{G_{F^+_w}})\otimes
  (\Ind_{G_M}^{G_{F'}}\theta')|_{G_{F^+_w}}\sim
  (\Sym^{n-1}(r_\chi)|_{G_{F^+_w}})\otimes
  (\Ind_{G_M}^{G_{F'}}\theta)|_{G_{F^+_w}}.\]
  Since \[\rho|_{G_{F^+_w}}=(\Sym^{n-1}r_{l,\iota}(\pi)|_{G_{F^+_w}})\otimes
  (\Ind_{G_M}^{G_{F'}}\theta)|_{G_{F^+_w}},\] \[\rho'|_{G_{F^+_w}}=(\Sym^{n-1}r_{l,\iota}(\pi')|_{G_{F^+_w}})\otimes
  (\Ind_{G_M}^{G_{F'}}\theta')|_{G_{F^+_w}},\] the result follows from
  the transitivity of $\sim$.\end{proof}

For $v$ a place of $F^+$, let $R^{\square}_{\rhobar|_{G_{F^+_v}}}$
denote the universal $\mc{O}$-lifting ring of
$\rhobar|_{G_{F^+_v}}$. Extending $\mc{O}$ if necessary, we can and do
assume that if $v$ is such that at least one $\rho$ or $\rho'$ is
ramified at $v$, then for each minimal prime ideal $\wp$ of
$R^{\square}_{\rhobar|_{G_{F^+_v}}}$, the quotient
$R^{\square}_{\rhobar|_{G_{F^+_v}}}/\wp$ is geometrically integral.

\begin{sublem}
  For all places $v\nmid l$ of $F^+$, either
  \begin{itemize}
 \item $\rho|_{G_{F^+_v}}$ and $\rho'|_{G_{F^+_v}}$ are both
   unramified, or
 \item each of the following conditions hold:
 \begin{itemize}
  \item  $\rho'|_{G_{F^+_v}} \leadsto_{\bigO} \rho|_{G_{F^+_v}}$ (see
   Definition \ref{defn: leadsto}),  
 \item  $\rho|_{G_{F^+_v}} \leadsto_{\bigO} \rho'|_{G_{F^+_v}}$, and
 \item the similitude characters of $\rho$ and $\rho'$ agree on inertia
  at $v$.
 \end{itemize}
 \end{itemize}
\end{sublem}

\begin{proof}
  If $v$ does not lie over a place in $T$ then there is nothing to
  prove, so we may suppose that $v$ lies over a place of $T$. 
  The condition on similitude characters is immediate since $\rho$ and $\rho'$ are both unipotent on inertia at $v$.
  Let us then turn to checking the condition that 
  $\rho'|_{G_{F^+_v}} \leadsto_{\bigO} \rho|_{G_{F^+_v}}$. By 
  assumption, condition (1) in Definition \ref{defn: leadsto} is
  satisfied so we just need to check condition (2).  Let $\rho_v =
  r_{l,\iota}(\pi)|_{G_{F^+_v}}$ and $\rho'_v =
  r_{l,\iota}(\pi')|_{G_{F^+_v}}$. Then $\rho_v$ and $\rho'_v$ are
  both lifts of the same reduction $\rhobar_{v}: G_{F^+_v} \rightarrow
  \GL_2(k)$. It follows easily from Corollary 2.6.7 of \cite{kis04}
  that there is a quotient $R^{\St}_{\rhobar_v}$ of
  $R^{\square}_{\rhobar_v}$ corresponding to lifts which are
  extensions of an unramified character $\gamma$ by $\gamma \epsilon$,
  and furthermore that the ring $R^{\St}_{\rhobar_v}$ is an integral
  domain of dimension 5. Let $\rho_v^{\square}$ denote the universal
  lift to $R^{\St}_{\rhobar_v}$.  Then $\rho_v$ and $\rho'_v$ arise as
  specialisations of this lift at closed points of
  $R^{\St}_{\rhobar_v}[1/l]$; let us call these points $x$ and $x'$.

  Note that $(\Ind_{G_M}^{G_{F'}}\theta)|_{G_{F^+_v}} =
  \oplus_{i=0}^{m^*-1}(\theta^{\tau^i}|_{G_{F^+_v}}) e_i$ and
  similarly $(\Ind_{G_M}^{G_{F'}}\theta')|_{G_{F^+_v}} =
  \oplus_{i=0}^{m^*-1}(\theta'^{\tau^i}|_{G_{F^+_v}}) e'_i$. For
  $i=0,\ldots,m^*-1$, let $\wt\theta_i : G_{F^+_v} \rightarrow
  \mc{O}^{\times}$ denote the Teichm\"uller lift of
  $\overline{\theta}^{\tau^i}|_{G_{F^+_v}}$.  If $R$ is an object of
  $\mc{C}_{\mc{O}}$ and $r \in R^{\times}$, we let
  $\lambda(r):G_{F^+_v} \rightarrow R^{\times}$ denote the unramified
  character sending $\Frob_{F^+_v}$ to $r$.

  Let $R=R^{\square}_{\rhobar|_{G_{F^+_v}}}$ and let
  $S=R^{\St}_{\rhobar_v}[[X_0,\ldots,X_{m^*-1}]]$. The lift \[
  \Sym^{n-1}\rho^{\square}_{v} \otimes \left( \oplus_{i=0}^{m^*-1}
    \wt\theta_i \lambda(1+X_i) \right)e_i \] of $\rhobar|_{G_{F^+_v}}$
  gives rise to a map $\Spec S \rightarrow \Spec R$. Since $S$ is a
  domain, the image of this map must be contained in an irreducible
  component of $\Spec R$. We deduce that $\rho|_{G_{F^+}}$ and
  $\rho'|_{G_{F^+}}$ are contained in a common irreducible component of
  $\Spec R[1/l]$.

 To prove that $x'$ is contained in a unique irreducible component, it
 then suffices to prove that $x'$ is a smooth point of $\Spec
 R[1/l]$. Since the completed local ring $R^{\wedge}_{x'}$ at $x'$ is
 the universal $K$-lifting ring of $(\rho'|_{G_{F_v^+}})
 \otimes_{\mc{O}} K$, a standard argument shows that $R^{\wedge}_{x'}$
 is smooth if $H^2(G_{F^+_v},\ad\rho'|_{G_{F^+_v}})=0$. By Tate local
 duality, it suffices to show that
 $H^0(G_{F^+_v},\ad\rho'|_{G_{F^+_v}}(1))=0$, i.e. that
 $\Hom_{G_{F^+_v}}(\rho'|_{G_{F^+_v}},\rho'|_{G_{F^+_v}}(1))=0$.

Let $\St_v$ denote the $n$-dimensional representation of
$G_{F^+_v}$ corresponding to the Steinberg representation. 
Then $\rho'|_{G_{F^+_v}}$ is $\GL_n(K)$-conjugate to an unramified twist of
 \[\oplus_{i=0}^{m^*-1}\St_v\otimes\theta'^{\tau^i}|_{G_{F^+_v}}.\]
 We may therefore assume that $\rho'|_{G_{F^+_v}}$ is equal to this representation.
It is easy to check (for example by considering the corresponding
Weil-Deligne representation) that the representation $\St_v$ contains
a unique $j$-dimensional subrepresentation for each $j$. 
Then a nonzero element of
$\Hom_{G_{F^+_v}}(\rho'|_{G_{F^+_v}},\rho'|_{G_{F^+_v}}(1))$ would have
to give a non-zero map from the unique $j$-dimensional quotient of
$\St_v\otimes\theta'^{\tau^i}|_{G_{F^+_v}}$ to the unique
$j$-dimensional subrepresentation of
$\St_v\otimes\theta'^{\tau^{i'}}(1)|_{G_{F^+_v}}$ for some $i$, $i'$ and
$j$. This implies that $(\theta'^{\tau^{i'}}/\theta'^{\tau^{i}})|_{G_{F^+_v}}$ is a
nonzero power of the cyclotomic character, which is impossible,
because $\theta'^{\tau^{i'}}/\theta'^{\tau^{i}}$ is a ratio of algebraic
characters of the same weight, and is thus pure of weight 0.

Finally, we can see that 
$\rho|_{G_{F^+_v}} \leadsto_{\bigO} \rho'|_{G_{F^+_v}}$
using the same argument.
\end{proof}
By Theorem \ref{thm: the main modularity lifting theorem, using
  tilde}, $\rho$ is automorphic (the conditions on the image of
$\rhobar$ follow from Lemma \ref{lem: existence of chars with long
  list of properties} and the choice of $F^+$, and the remaining
conditions follow by construction and the two sublemmas just
proved). Since $F^+/F''$ is solvable and 
$\rho=r''|_{G_{F^+}}$, $r''|_{G_{F''}}$ is automorphic by Lemma 1.3 of
\cite{BLGHT}. But
$r''|_{G_{F''}}=(\Sym^{n-1}r_{l,\iota}(\pi))|_{G_{F''}}\otimes(\Ind_{G_M}^{G_F}\theta)|_{G_{F''}}=(\Sym^{n-1}r_{l,\iota}(\pi))|_{G_{F''}}\otimes(\Ind_{G_{MF''}}^{G_{F''}}(\theta|_{G_{MF''}}))$,
so by Proposition \ref{prop:
  untwisting} $(\Sym^{n-1}r_{l,\iota}(\pi))|_{G_{F''}}$ is
automorphic, as required.
\end{proof}

\begin{cor}
  \label{cor: main L-fn results}Suppose that $F$ is a totally real
  field and that $\pi$ is a non-CM regular algebraic cuspidal automorphic
  representation of $\GL_2(\A_F)$ of weight $\lambda$, and
  let $w_\pi$ be the common value of the numbers
   $\lambda_{v,1}+\lambda_{v,2}$, $v|\infty$. 
  Suppose that $n$ is a positive
  integer and that $\psi:F^\times\backslash \A_F^\times\to\C^\times$
  is a finite order character. Then there is a meromorphic function
  $L(\Sym^{n-1}\pi\times\psi,s)$ on the whole complex plane such that:
  \begin{itemize}
  \item For any prime $l$ and any isomorphism $\iota:\Qlbar\isoto\C$
    we have   $L(\Sym^{n-1}\pi\times\psi,s)=L(\iota(\Sym^{n-1}
    r_{l,\iota}(\pi)\otimes r_{l,\iota}(\psi)),s)$.
  \item The expected functional equation holds between
    $L(\Sym^{n-1}\pi\times\psi,s)$ and   $L(\Sym^{n-1}(\pi^\vee|\det|^{-w_\pi})\times\psi,1+(n-1)w_\pi-s)$.
  \item If $n>1$ or $\psi\neq 1$ then   $L(\Sym^{n-1}\pi\times\psi,s)$ is
    holomorphic and non-zero in $\Re s\ge 1+(n-1)w_\pi/2$.
  \end{itemize}

\end{cor}
\begin{proof}This follows from Theorem \ref{thm: nth symmetric power of pi is potentially
    automorphic}, as in the proof of Theorem 4.2 of \cite{hsbt}. We
  give the details. The $L$-function $L(\iota(\Sym^{n-1}
    r_{l,\iota}(\pi)\otimes r_{l,\iota}(\psi)),s)$ is independent of
    the choice of $l$, $\iota$ by definition, so it is enough to prove
    the result for the $l$, $\iota$ fixed throughout this section. Let
    $\pi_n$, $F''$ be as in the conclusion of Theorem \ref{thm: nth symmetric power of pi is potentially
    automorphic}.

  We claim that $\rec(\pi_{n,v})=(\Sym^{n-1}\rec(\pi_v))|_{W_{F''_v}}$
  for all places $v$ of $F''$. If $v\nmid l$, this follows from Lemma
  \ref{lem: purity implies local-global compatibility} and the purity
  of $r_{l,\iota}(\pi)$ (which follows from the main result of
  \cite{MR2327298}). If $v|l$, then we choose a prime $p\ne l$ and an
  isomorphism $\iota':\Qpbar\isoto\C$. By the Tchebotarev density
  theorem we have $r_{p,\iota'}(\pi_n)=(\Sym^{n-1}
  r_{p,\iota'}(\pi))|_{G_{F''}}$, and we may argue as before.

 By Lemma 1.3 of \cite{BLGHT}, for any intermediate
  field $F\subset F_j\subset F''$ with $F''/F_j$ soluble, there is an
  automorphic representation $\pi^j$ of $\GL_n(\A_{F_j})$ with $r_{l,\iota}(\pi^j)=(\Sym^{n-1}
    r_{l,\iota}(\pi))|_{G_{F_j}}$. By Brauer's theorem, we can
    write \[1=\sum_j a_j\Ind_{G_{F_j}}^{G_F}\chi_j\]with $a_j\in\Z$
    and $\chi_j:G_{F_j}\to\C^\times$ a
    homomorphism. Then by the above discussion (applied to the
    representations $\pi^j$), we have \[L(\iota(\Sym^{n-1}
    r_{l,\iota}(\pi)\otimes r_{l,\iota}(\psi)),s)=\prod_j
    L(\pi^j\otimes(\chi_j\circ\Art_{F_j})\otimes\psi,s)^{a_j}.\] The
    result follows.
\end{proof}

We now deduce the Sato-Tate conjecture for $\pi$, following the
formulation of \cite{GeeSTwt3} (see also section 8 of
\cite{BLGHT}). Recall $w_\pi$ is the common value of the
$\lambda_{v,1}+\lambda_{v,2}$, $v|\infty$. Let $\psi$ be the product
of the central character of $\pi$ with $|\cdot|^{w_\pi}$, so that
$\psi$ has finite order. Let $a$ denote the order of $\psi$, and let
$U(2)_a$ denote the subgroup of $U(2)$ consisting of those matrices
$g\in U(2)$ with $\det(g)^a=1$. Let $U(2)_a/\sim$ denote the set of
conjugacy classes of $U(2)_a$. By ``the Haar measure on
$U(2)_a/\sim$'' we mean the push forward of the Haar measure on
$U(2)_a$ with total measure 1.

The Ramanujan conjecture is known to hold at all finite places of
$\pi$ (see \cite{MR2327298}), so for all $v$ for which $\pi_v$ is
unramified, the matrix $(\mathbf{N} v)^{-w_\pi/2}\rec(\pi_v)(\Frob_v)$ lies in
 $U(2)_a$. Let $[\pi_v]$ denote its conjugacy class 
in $U(2)_a/\sim$.

\begin{thm}
  Suppose that $F$ is a totally real field, and that $\pi$ is a non-CM
  regular algebraic cuspidal automorphic representation of $\GL_2(\A_F)$. Then the
  classes $[\pi_v]$ are equidistributed with respect to the Haar
  measure on $U(2)_a/\sim$.
\end{thm}
\begin{proof}This follows from Corollary \ref{cor: main L-fn results}, together
  with the corollary to Theorem I.A.2 of \cite{MR0263823} (note that
  the irreducible representations of $U(2)_a$ are the representations
  $\det^c\otimes\Sym^d\C^2$ for $0\le c<a$, $d\ge 0$).\end{proof}

This may be reformulated in a somewhat more explicit fashion as
follows. Note that the space $U(2)_a/\sim$ is disconnected if $a\neq 1$, so that to make
an explicit statement we choose a connected component, which amounts
to choosing $\zeta$ as in the statement below (of course, one may replace
$\zeta$ by $-\zeta$, and it is only the choice of $\zeta^2$ which
determines a component).
\begin{cor}Suppose that $F$ is a totally real field, and that $\pi$ is
  a non-CM regular algebraic cuspidal automorphic representation of
  $\GL_2(\A_F)$. Let $\psi$ be the product of the central character of
  $\pi$ with $|\cdot|^{w_\pi}$, so that $\psi$ is a finite order character. Let $\zeta$
  be a root of unity with $\zeta^2$ in the image of $\psi$. For any
  place $v$ of $F$ such that $\pi_v$ is unramified, let $t_v$ denote
  the eigenvalue of the Hecke
  operator \[\left[\GL_2(\bigO_{F_v})\begin{pmatrix}\varpi_v & 0\\ 0&
    1\end{pmatrix}\GL_2(\bigO_{F_v})\right]\](where $\varpi_v$ is a
  uniformiser of $\bigO_{F_v}$) on $\pi_v^{\GL_2(\bigO_{F_v})}$. Note
  that if $\psi_v(\varpi_v)=\zeta^2$ then
  $t_v/(2(\mathbf{N}v)^{(1+w_\pi)/2}\zeta)\in[-1,1]\subset\R$.

Then as $v$ ranges over the places of $F$ with $\pi_v$ unramified and
$\psi_v(\varpi_v)=\zeta^2$, the values $t_v/(2(\mathbf{N}v)^{(1+w_\pi)/2}\zeta)$ are equidistributed
in $[-1,1]$ with respect to the measure $(2/\pi)\sqrt{1-t^2}dt$.
  
\end{cor}

\bibliographystyle{amsalpha}
\bibliography{barnetlambgeegeraghty}

\end{document}